\documentclass[reqno]{amsart}
\usepackage[sc]{mathpazo}

\usepackage{amsthm}
\usepackage{amssymb,lineno}
\usepackage[centertags]{amsmath}
\usepackage{mathrsfs}
\usepackage{indentfirst}
\usepackage{amssymb}
\usepackage{dsfont}
\usepackage{ifpdf}
\usepackage{array}
\usepackage{graphicx}

\usepackage[colorlinks=true]{hyperref}
\hypersetup{urlcolor=blue, citecolor=red}

\let\oldsection\section
\renewcommand\section{\setcounter{equation}{0}\oldsection}

\newtheorem{theorem}{\indent Theorem}[section]

\newtheorem{lemma}[theorem]{\indent Lemma}
\newtheorem{remark}{\indent Remark}[section]
\newtheorem{definition}[remark]{\indent Definition}
\newtheorem{example}[remark]{\indent Example}
\newtheorem{proposition}[theorem]{\indent Proposition}
\newtheorem{corollary}[theorem]{\indent Corollary}
\def\pd#1#2{\frac{\partial#1}{\partial#2}}

\newcommand{\dhr}{\mathrel{\lhook\joinrel\relbar\kern-.8ex\joinrel\lhook\joinrel\rightarrow}}

\title[Some Nonlinear Boundary Value Problems II] {On the Existence of Positive Solutions for Some
Nonlinear Boundary Value Problems II}

\begin{document}
\author{Hongjing Pan}
\address{School of Mathematical Sciences \\ South China Normal University \\ 510631 Guangzhou\\ China}
\email{panhj@scnu.edu.cn}

\author{Ruixiang Xing}
\address{School of Mathematics and Computational Science \\ Sun Yat-sen University \\ 510275 Guangzhou, China}
\email{xingrx@mail.sysu.edu.cn}

\begin{abstract}
We study
a class of boundary value problems with $\varphi$-Laplacian (e.g.,
the prescribed mean curvature equation, in which
$\varphi(s)=\frac{s}{\sqrt{1+s^2}}$)
\begin{center}
 $-\left(\varphi(u')\right)'=\lambda f(u)\; \text{ on }(-L, L),\quad u(-L)=u(L)=0,$
\end{center}
where  $\lambda$ and $L$ are positive parameters. For convex $f$
with $f(0)=0$, we establish various results on the exact number of
positive solutions as well as global bifurcation diagrams. Some new
bifurcation patterns are shown. This paper is a continuation of
\cite{Pan2014a}, where the case $f(0)>0$ has been investigated.
\end{abstract}


\maketitle

\section{Introduction}
Consider the following nonlinear boundary value problem
\begin{equation}
\left\{ \begin{aligned}\label{eq:1dmemsfr0}
        &-\left(\varphi(u')\right)'=\lambda f(u),  \quad & x\in (-L, L),\\
                  &u(-L)=u(L)=0,
                          \end{aligned} \right.
                          \end{equation}
where $\varphi$ and $f$ satisfy the conditions
\begin{align}\label{con:phi}
& \varphi\in C^2(\mathbb{R}) \hbox{ is odd and  } \varphi'(t)>0\;
\hbox{ for all } t\in \mathbb{R}.\\\label{con:positive}
\begin{split}
&f \hbox{ is continuous on } [0, A) \hbox{ satisfying } f(u)>0  \; \hbox{ for all  } 0<u<A,\\
& \hbox{ where either } A=+\infty, \hbox{ or  } A<+\infty \hbox{
with  }  \lim_{u\rightarrow A^-}f(u)=+\infty. 
\end{split}
\end{align}
This paper is a continuation of the paper by Pan and Xing
\cite{Pan2014a}, where various results on the existence and exact
number of positive solutions
are obtained when $f$ is an increasing convex function with
$f(0)>0$. In the present paper, we investigate the case $f(0)=0$.

For various different purposes, $\varphi$ and $f$ are also required
to satisfy some or all of the conditions:
\begin{align}\label{con:phi2}
&z\varphi''(z)\leqslant 0 \;\hbox{ for all } z\in
\mathbb{R},\\\label{f-c1} &f \hbox{ is of class } C^1 \hbox{ on }
[0, A)\hbox{ satisfying } f'(u)u\geqslant f(u)  \text{ for } \ u\in
(0,A).\\\label{con:strict2}
\begin{split}
 & \hbox{One of the inequalities in \eqref{con:phi2} and \eqref{f-c1} is strict, except for at most}\\
& \hbox{ a finite number of $z$ or $u$}.
\end{split}
\end{align}

The problem \eqref{eq:1dmemsfr0} with \eqref{con:phi} and
\eqref{con:phi2}
includes many important examples such as
\begin{equation}\label{eq:1dffmems2}
\left\{ \begin{aligned}
        -\frac{u''}{(1+|u'|^2)^{\frac{k}{2}}}&=\lambda f(u),  \quad & x\in (-L, L),\\
                             u(-L)=u(L)&=0,
                          \end{aligned} \right.
                          \end{equation}
where $k\geqslant 0$ and
$\varphi(s)=\int_0^s(1+t^2)^{-\frac{k}{2}}\,\text{d}t$. When $k=2$,
\eqref{eq:1dffmems2} becomes
\begin{equation}
\left\{ \begin{aligned}\label{eq:1}
        &-\frac{u''}{1+|u'|^2}={\lambda f(u)},  \quad & x\in (-L, L),\\
                  &u(-L)=u(L)=0,
        \end{aligned} \right.
\end{equation}
and $\varphi(s)=\arctan s$. Problem \eqref{eq:1} is related to a
MEMS model with fringing field (see e.g. \cite{Pelesko2005}). When
$k=3$, \eqref{eq:1dffmems2} becomes
one-dimensional prescribed mean curvature problem
\begin{equation}
\left\{ \begin{aligned}\label{eq:mcemems}
        -\Big(\frac{u'}{\sqrt{1+|u'|^2}}\Big)'&=\lambda f(u),  \quad & x\in (-L, L),\\
                             u(-L)=u(L)&=0,
                          \end{aligned} \right.
                          \end{equation}
and $\varphi(s)=\frac{s}{\sqrt{1+s^2}}$. Quasilinear problem
\eqref{eq:mcemems} absorbed much attention in recent years and some
special nonlinearities $f$ satisfying $f(0)=0$ such as $u^p, e^u-1,
u^p+u^q$, were studied and many interesting results on existence and
exact multiplicity were obtained (see
\cite{Bereanu2006,BHOO072,Brubaker2012,BG10,HO07,KL10,Li2010,Obersnel2007,xp2010,xp2011a,Zhang2013}).



In \cite{Pan2014a}, due to limitations of space, we only focused on
the case $f(0)>0$. In the present paper, we still use the time-map
method, following the same line as in \cite{Pan2014a}, to
investigate the equally important case $f(0)=0$.
This method is based on the fact that the investigation of the exact
number of positive solutions of problem \eqref{eq:1dmemsfr0} is
equivalent to studying the shape of a time map $T$.
We will further reduce the problem to the shape of a simpler
function $g$,
 which describes the values of $T$ at the right endpoint of the interval of definition.
The patterns of bifurcation diagrams for problem
\eqref{eq:1dmemsfr0} with \eqref{con:phi}--\eqref{con:strict2}
finally depend on the number of local extreme points and the local
extreme values of $g$. By analytical proof or numerical simulation,
we find many interesting new examples, which suggest that there
exist more than ten types of shapes of $g$. This means that
bifurcation diagrams of \eqref{eq:1dmemsfr0} can contain very
complex patterns. We establish various results on the exact number
of positive solutions as well as bifurcation diagrams corresponding
to some different types of $g$. Although these results and figures
occupy much space of the present paper, we think that it is worth
doing in order to show important details.

In this paper, by a \emph{positive solution} we mean a positive
classical solution, that is, a function $u\in C^2[-L, L]$ satisfying
\eqref{eq:1dmemsfr0} and $u>0$ in $(-L,L)$.


%
%

We organize the paper as follows. In Section 2, we present our main
results about the existence and exact number of positive solutions
as well as global bifurcation diagrams. In Section 3, we investigate
general properties of the time map. The proofs of the main results
will be given in Section 4.

\section{Main results}

Our main results are the following theorems. Notice that
\eqref{con:positive} and \eqref{f-c1} imply that $f(0)=0$ and both
$f(u)$ and $\frac{f(u)}{u}$ are increasing for $u\in (0,A)$.
\begin{theorem}\label{thm:atmost12}
Assume conditions \eqref{con:phi}--\eqref{con:strict2} hold.
Then \eqref{eq:1dmemsfr0} has at most one positive solution for any
$\lambda>0$.
\end{theorem}





Since $f''(u)\geqslant(>) 0$ and $f(0)=0$ imply $f'(u)u\geqslant(>)
f(u)$, we have
\begin{corollary}\label{thm:atmost1}
Replacing \eqref{f-c1} and \eqref{con:strict2} in Theorem
\ref{thm:atmost12} by the following conditions
\begin{align}\tag{\ref{f-c1}$'$} \label{con:convex}
&f \hbox{ is of class } C^1 \hbox{ on } [0, A) \hbox{ and } C^2
\hbox{ on } (0, A)
\hbox{ satisfying } f''(u)\geqslant 0 \hbox{ and } f(0)=0.\\
\tag{\ref{con:strict2}$'$}\label{con:strict}
\begin{split}
 & \hbox{One of the inequalities in \eqref{con:phi2} and \eqref{con:convex} is strict, except for at most a finite}\\
& \hbox{ number of $z$ or $u$}.
\end{split}
\end{align}
Then the conclusion of Theorem \ref{thm:atmost12} is still true.
\end{corollary}

\begin{remark}
\emph{(a)}
%
%
Conditions \eqref{f-c1} and \eqref{con:convex} in the above results
are crucial. For example, for problem \eqref{eq:mcemems} with
$f(u)=u^p$ $(0<p<1)$ or $f(u)=u-u^3$, it is well known in
\cite{HO07} or \cite{BG10} that there exists  $\lambda^*>0$ such
that \eqref{eq:1dmemsfr0} has exactly two positive solutions.

\emph{(b)} Condition \eqref{con:strict2} or \eqref{con:strict}
cannot be removed from the above results. For example, the linear
eigenvalue problem
$$-u''=\lambda u\; \text{ on }(-L, L),\quad u(-L)=u(L)=0$$
has infinitely many positive solutions for
$\lambda=(\frac{\pi}{2L})^2$. Moreover, if $f$ satisfies
\eqref{f-c1} and there exists an $r_0$ such that $f(w)=m_0w$ for
$0\leqslant w\leqslant r_0$, then the problem
$$-u''=\lambda f(u)\; \text{ on }(-L, L),\quad u(-L)=u(L)=0$$
also has infinitely many positive solutions for
$\lambda=\frac{1}{m_0}(\frac{\pi}{2L})^2$ (see e.g. \cite[Thm
3.2]{Laetsch1970}).
\end{remark}

In order to further give the exact number of positive solutions for
each $\lambda$, we introduce some notations. The same as in
\cite{Pan2014a}, we denote
\begin{align*}
\Phi(z)=\int_{0}^{z}t \varphi'(t)dt \quad \hbox{and} \quad
F(u)=\int_0^u  f(s)\text{d}s.
\end{align*}
and define
\begin{align*}
B=\sup_{z\in [0, +\infty)} \Phi(z) \quad \hbox{and} \quad
C=\sup_{u\in [0,A)} F(u).
\end{align*}
Then \eqref{con:phi} and \eqref{con:positive} imply that
$B=\lim_{z\rightarrow +\infty}\Phi(u)$ and $C=\lim_{u\rightarrow
A^-}F(u)$, respectively.

%

Since \eqref{con:positive} and \eqref{f-c1} implies $f'\geqslant 0$,
it follows that $A=+\infty$ implies $C=+\infty$. The same as in
\cite{Pan2014a}, the problem considered in Theorem
\ref{thm:atmost12} can be classified into the following six cases.
\begin{equation}\tag{*}\label{eq:six}
\fbox{ \quad  Six cases of \eqref{eq:1dmemsfr0} $ \left\{
  \begin{array}{ll}
    B=+\infty & \left\{
  \begin{array}{ll}
    A=+\infty & \quad C=+\infty\qquad\quad (\text{Case I})\\
    A<+\infty & \left\{
  \begin{array}{ll}
    C=+\infty\qquad\quad (\text{Case II})& \\
    C<+\infty\qquad\quad (\text{Case III})&
  \end{array}
\right.
  \end{array}
\right.\\
&\\
    B<+\infty & \left\{
  \begin{array}{ll}
    A=+\infty &\quad C=+\infty\qquad\quad (\text{Case IV})\\
    A<+\infty & \left\{
  \begin{array}{ll}
    C=+\infty\qquad\quad (\text{Case V})& \\
    C<+\infty\qquad\quad (\text{Case VI})&
  \end{array}
\right.
  \end{array}
\right.
  \end{array}
\right. $}
\end{equation}



We shall mainly focus on the situation where the range of $\varphi$
is bounded. However, Theorems
\ref{thm:doubleinfinit02}--\ref{thm:f=0t1} and many results in
Sections 3 and 4 are also of interest when $\varphi$ is unbounded.
Notice that under \eqref{con:phi}, condition $B<+\infty$ implies
that $\varphi$ is bounded (\cite[Remark 2.2]{Pan2014a}).

\begin{example}[\cite{Pan2014a}]\label{table:fi}
Denote
$\varphi_k(s)=\int_0^s(1+t^2)^{-\frac{k}{2}}\,\text{d}t\;(k\geqslant
0)$. Then $\varphi_k$ satisfies both \eqref{con:phi} and
\eqref{con:phi2}. Moreover, we have
\begin{center}
\begin{tabular}{|c|c|l|}\hline
 $\varphi$  &  $B=+\infty$ & \qquad\qquad\qquad $B<+\infty$ \\\hline
  & $\varphi_2(s)=\arctan s$ &$\varphi_k(s),\; k>2$, e.g.  \\
Bounded &\emph{(Problem~\eqref{eq:1})}&
$\varphi_3(s)=\frac{s}{\sqrt{1+s^2}}$\:~\emph{(Mean~Curvature~Type)}\\
&&$\varphi_5(s)=\frac{s}{\sqrt{1+s^2}}-\frac{1}{3}
\frac{s^{3}}{(1+s^2)^{\frac{3}{2}}}$\\\hline Unbounded &
$\varphi_k(s)$, $0\leqslant k<2$&\\\hline
\end{tabular}
\end{center}
When $k>2$, we also have
\begin{align*}
\Phi_k(z)=\frac{1}{k-2}-\frac{1}{k-2}\frac{1}{(1+z^2)^{\frac{k-2}{2}}},\quad
B=\frac{1}{k-2}, \quad
\Phi_k^{-1}(y)=\frac{\sqrt{1-[1-(k-2)y]^{\frac{2}{k-2}}}}{[1-(k-2)y]^{\frac{1}{k-2}}}.
\end{align*}
\end{example}
%



%

We next investigate Cases I--VI in \eqref{eq:six}, respectively.
Denote $\lambda_1=\frac{\varphi'(0)}{f'(0)}(\frac{\pi}{2L})^2$.

First, we consider the three cases of $B=+\infty$.

\vskip 2mm {\bf Case I: $B=+\infty$, $A=+\infty$ and $C=+\infty$}

\begin{theorem}[Type I, see Fig.\ref{fig:f1234}]\label{thm:doubleinfinit02}
Let $A,B,C=+\infty$.
Assume conditions \eqref{con:phi}--\eqref{con:strict2} hold.
Also assume
\begin{align}\label{con:phif02}
 \lim_{t\rightarrow +\infty}\frac{\varphi\circ\Phi^{-1}(\lambda t )}{ f\circ F^{-1}\left(t\right)}=0 \qquad \hbox{for any } \lambda>0.
\end{align}
Then the following assertions hold:

(a) If $f'(0)= 0$, then \eqref{eq:1dmemsfr0} has exactly one
positive solution for any $\lambda\in (0,+\infty)$.

(b) If $f'(0)> 0$, then
\eqref{eq:1dmemsfr0} has exactly one positive solution for
$\lambda\in (0, \lambda_1)$ and none for
$\lambda\in[\lambda_1,+\infty)$.
\end{theorem}

  \ifpdf
\begin{figure}
\centering
\includegraphics[totalheight=8.2in]{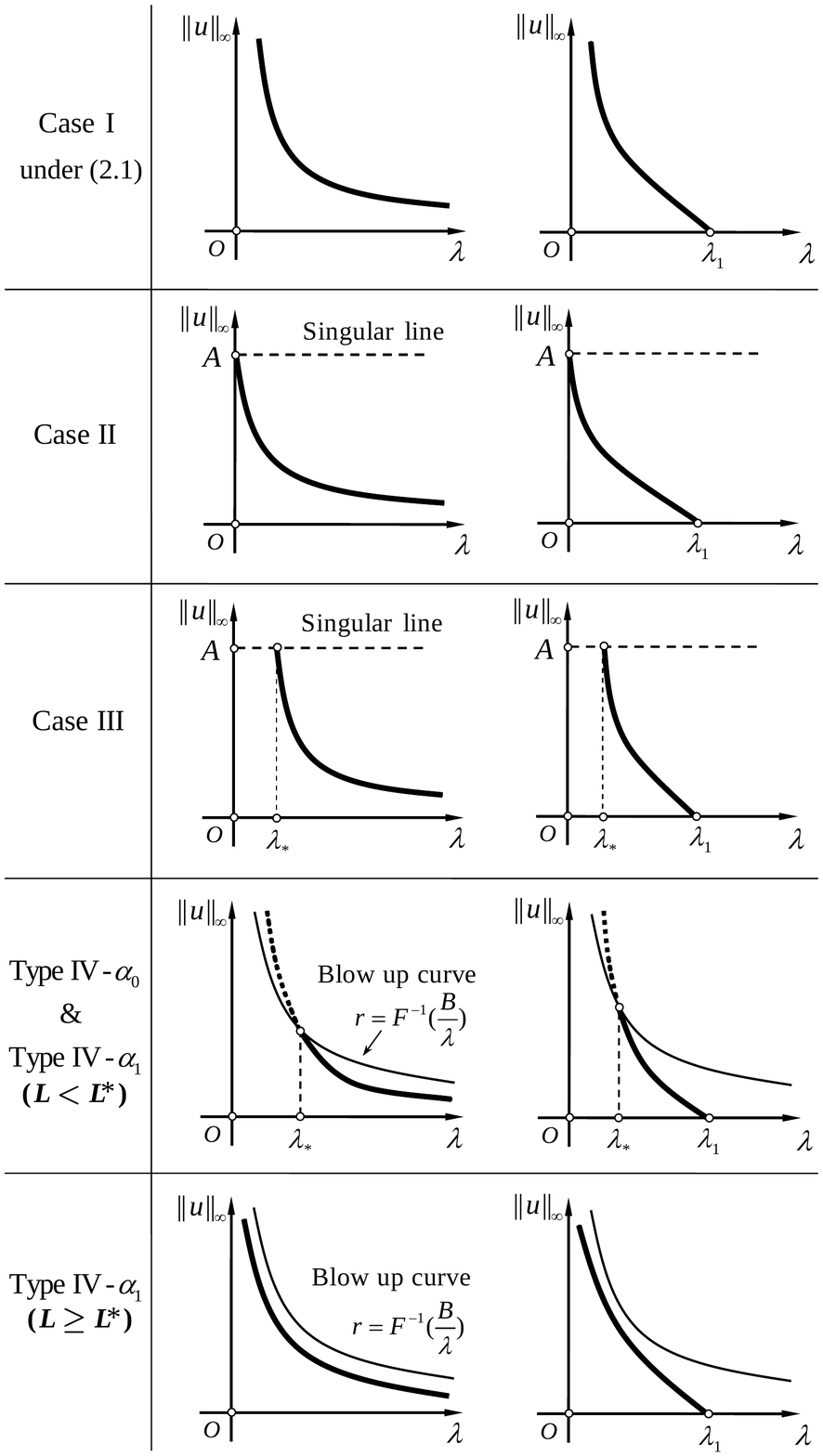}
\caption{Bifurcation Diagrams for Cases I--III, Types IV-$\alpha_0$
and IV-$\alpha_1$ with $f(0)=0$. Left: $f'(0)=0$.  Right:
$f'(0)>0$.}\label{fig:f1234}
\end{figure}
  \fi

\begin{corollary}\label{cor:doubleinfinit02}
Let $A,B,C=+\infty$. Assume conditions
\eqref{con:phi}--\eqref{con:strict2} hold.
If the range of $\varphi$ is bounded or $\frac{F(z)}{f(z)}$ is
bounded for sufficiently large $z$, then \eqref{con:phif02} holds
and hence the conclusions of Theorem \ref{thm:doubleinfinit02} hold.
\end{corollary}

\begin{example}\label{eg:i}
Let $\varphi=\varphi_k\:(0\leqslant k\leqslant 2)$, which is defined
in Example \ref{table:fi}, and let $f$ be one of the following
table.
\begin{center}
\begin{tabular}{|c|l|l|}\hline
 $f$& $\frac{F(z)}{f(z)}$ is unbounded & $\frac{F(z)}{f(z)}$ is bounded\\\hline
 $f$ convex,& $f(u)=u^p,\; p\geqslant 1$& $f(u)=e^u-1$\\
  $f(0)=0$,& $f(u)=u^p+u^q,\; q>p\geqslant 1$ & $f(u)=e^u+u^p-1,\; p\geqslant 1$\\
  $A,C=+\infty$&$f(u)=(1+u)^p-1,\; p> 1$ & $f(u)=e^{u^2}-1 $\\
 & & $f(u)=e^{u^2}+u^p-1,\; p\geqslant 1 $\\\hline
\end{tabular}
\end{center}
The following two groups of $\varphi$ and $f$ give some examples
which satisfy the conditions of Corollary \ref{cor:doubleinfinit02}:\\
(1) $\varphi_2$ and any $f$ in the above table;\\
 (2) $\varphi=\varphi_k\:(0\leqslant k< 2)$ and any $f$ in the
last column of the above table.
\end{example}

\vskip 2mm

{\bf Case II: $B=+\infty$, $A<+\infty$ and $C=+\infty$}


\begin{theorem}[see Fig.\ref{fig:f1234}]\label{thm:f=0t1}
Let $A<+\infty$, ${B=+\infty}$, ${C=+\infty}$. Assume conditions
\eqref{con:phi}--\eqref{con:strict2} hold.
Then the following assertions hold:

(a) If $f'(0)= 0$, then \eqref{eq:1dmemsfr0} has exactly one
positive solution for any $\lambda\in (0,+\infty)$.

(b) If $f'(0)> 0$, then \eqref{eq:1dmemsfr0} has exactly one
positive solution for $\lambda\in (0, \lambda_1)$ and none for
$\lambda\in[\lambda_1,+\infty)$.
\end{theorem}



\vskip 2mm {\bf Case III}: $B=+\infty$, $A<+\infty$ and $C<+\infty$

\begin{theorem}[see Fig.\ref{fig:f1234}]\label{thm:f=0t2}
Let $A<+\infty$, ${B=+\infty}$, ${C<+\infty}$. Assume conditions
\eqref{con:phi}--\eqref{con:strict2} hold. Then the following
assertions hold:

(a)~If $f'(0)= 0$, then there exists $\lambda_*>0$ such that
\eqref{eq:1dmemsfr0} has exactly one positive solution for
$\lambda\in(\lambda_*,+\infty)$ and none for $(0, \lambda_*]$.

(b)~If $f'(0)> 0$, then there exists  $\lambda_*\in (0,\lambda_1)$
such that \eqref{eq:1dmemsfr0} has exactly one positive solution for
$\lambda\in  (\lambda_*, \lambda_1)$ and none for $\lambda\in
(0,\lambda_*]\cup [\lambda_1,+\infty)$.

(c)~$\lambda_*$ is strictly decreasing with respect to $L$.
\end{theorem}

\begin{example} Let $\varphi=\varphi_k\:(0\leqslant k\leqslant 2)$, which is defined in Example \ref{table:fi},
and let $f$ be one of the following table.
\begin{center}
\begin{tabular}{|c|c|c|}\hline
   $f$ & $f(0)=0$, $f'(0)=0$& $f(0)=0$, $f'(0)>0$\\\hline
$A<+\infty, C=+\infty$ & $f(u)=\tan u^q$,\:\: $u\in [0,\sqrt[q]{\frac{\pi}{2}})\; (q>1)$ & $f(u)=\tan u$,\; $u\in [0,\frac{\pi}{2})$\\
\emph{(Case II)}& $f(u)=(1-u^q)^{-p}-1$,\; $u\in [0,1)$ &$f(u)=(1-u)^{-p}-1,\; u\in [0,1)$\\
&$(q>1, p\geqslant1)$& $(p\geqslant1)$\\\hline
$A, C<+\infty$& $f(u)=(1-u^q)^{-p}-1$,\; $u\in [0,1)$ & $f(u)=(1-u)^{-p}-1$,\; $u\in [0,1)$\\
\emph{(Case III)}&$(q>1, 0<p< 1)$&  $(0<p< 1)$\\\hline
\end{tabular}
\end{center}
Then $\varphi$ and $f$ give various examples which satisfy the
conditions of Theorem \ref{thm:f=0t1} or \ref{thm:f=0t2}.
\end{example}

\vskip 2mm We next consider the three cases of $B<+\infty$. The same
as in \cite{Pan2014a}, we introduce the function
\begin{align}\label{form:0g}
g(\lambda)
=\int_{0}^{F^{-1}(\frac{B}{\lambda})}\frac{1}{\Phi^{-1}( B -\lambda
F(u)) }du.
\end{align}
As pointed out in \cite{Pan2014a}, for given $\varphi$ and $f$, the
function $g$ plays a crucial role in determining bifurcation
diagrams of problem \eqref{eq:1dmemsfr0}:
 when the length parameter $L$ passes through the local extreme values of $g$,
the pattern of the bifurcation diagram must change; the complexity
of the graph of $g$ leads to the rich diversity of bifurcation
patterns for \eqref{eq:1dmemsfr0}.

To distinguish different types of $g$, we introduce the following
definitions.
\begin{definition}[See Fig.\ref{fig:gg1}]\label{df:g1}
We say that the function $g$ is of
\begin{align*}
 &\emph{Type } \gamma_2, \;\text{ if }\:  \lim_{\lambda\rightarrow
+\infty} g(\lambda)=0, \; \lim_{\lambda\rightarrow 0} g(\lambda)\in
(0,+\infty),\; g(\lambda)
\text{ has exactly two local extreme points in } (0, +\infty)\\
&\phantom{\emph{Type } \gamma_12 \;} \text{and the local maximum value is greater than $\lim_{\lambda\rightarrow 0} g(\lambda)$;}\\
 &\emph{Type } \gamma_3, \;\text{ if }\:  \lim_{\lambda\rightarrow
+\infty} g(\lambda)=0, \; \lim_{\lambda\rightarrow 0} g(\lambda)\in
(0,+\infty),\; g(\lambda)
\text{ has exactly two local extreme points in } (0, +\infty)\\
&\phantom{\emph{Type } \gamma_3 \;} \text{and the local maximum value is equal to $\lim_{\lambda\rightarrow 0} g(\lambda)$;}\\
&\emph{Type } \delta_2, \;\text{ if }\: \lim_{\lambda\rightarrow
+\infty} g(\lambda)=0, \;\lim_{\lambda\rightarrow 0} g(\lambda)=0,
\; g(\lambda)
\text{ has exactly three local extreme points in } (0, +\infty)\\
&\phantom{\emph{Type } \delta_2, \;} \text{and the left local maximum value is greater than the right one;}\\
&\emph{Type } \delta_3, \;\text{ if }\: \lim_{\lambda\rightarrow
+\infty} g(\lambda)=0, \;\lim_{\lambda\rightarrow 0} g(\lambda)=0,
\; g(\lambda)
\text{ has exactly three local extreme points in } (0, +\infty)\\
&\phantom{\emph{Type } \delta_3, \;} \text{and the left local maximum value is equal to the right one;}\\
&\cdots.
\end{align*}
\end{definition}
Among the graphs in Fig.\ref{fig:gg1}, the remaining, unmentioned
types in Definition \ref{df:g1} have been introduced in
\cite{Pan2014a}. Here, we use $\alpha,\beta,\gamma,\delta,\ldots$
(in the Greek alphabetical order) to represent the numbers of the
local extremum points of $g$ in $(0,+\infty)$ being zero, one, two,
three, ... , respectively.

  \ifpdf %
\begin{figure}
\centering
\includegraphics[totalheight=8.3in]{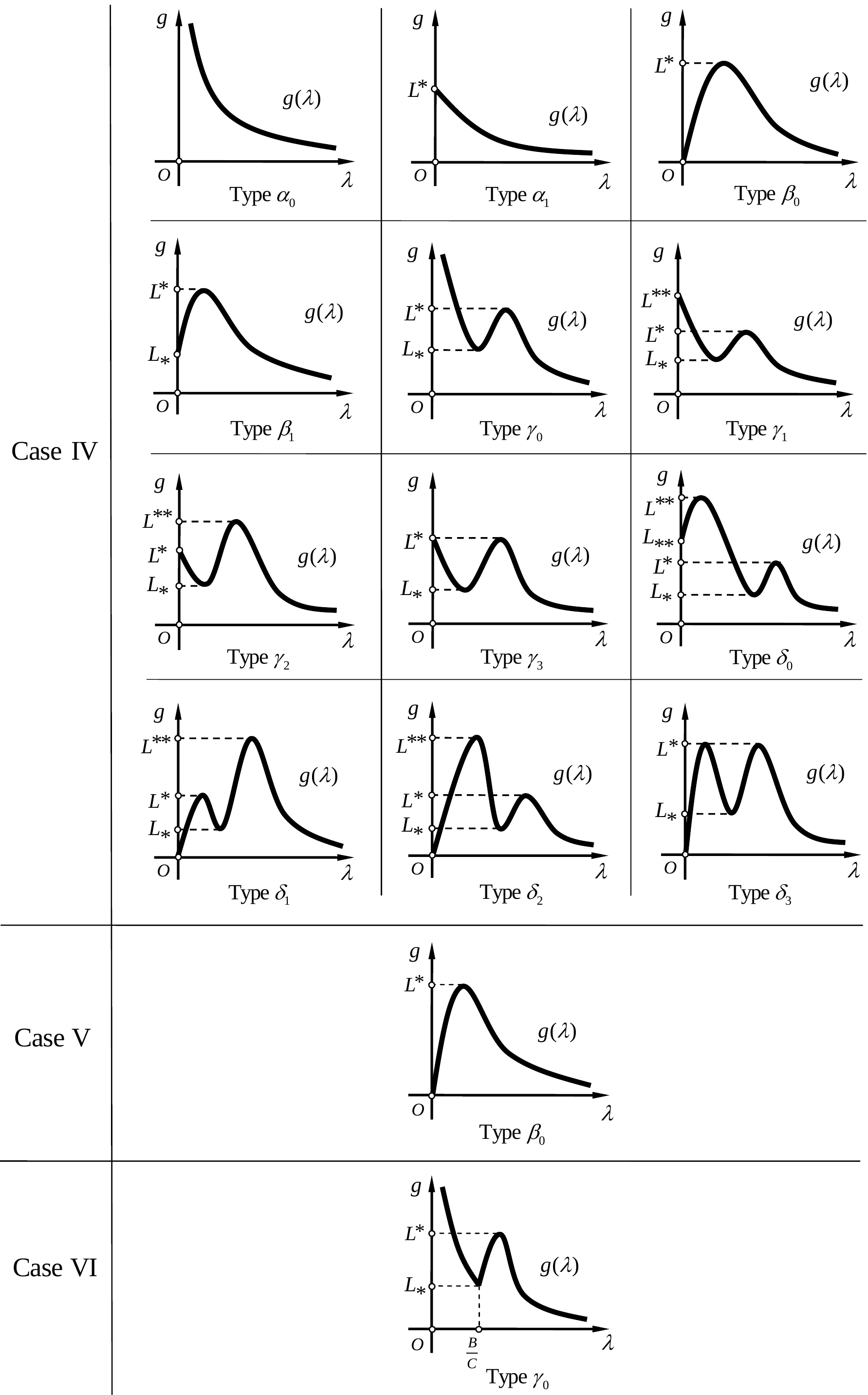}
\caption{Some shapes of the function $g$ for
$B<+\infty$.}\label{fig:gg1}
\end{figure}
   \fi %

\begin{definition}\label{df:g2}
If $\varphi$ and $f$ in \eqref{eq:1dmemsfr0} satisfy conditions
$A=+\infty$, $B<+\infty$  and $C=+\infty$, i.e., \emph{Case IV}, and
$g$ is of   \emph{Type $\gamma_2$ ($\gamma_3$, $\delta_2$, $\cdots$,
\emph{respectively})}, then we say the pair $(\varphi, f)$  is of
\emph{Type IV-$\gamma_2$ (IV-$\gamma_3$, IV-$\delta_2$, $\cdots$,
\emph{respectively})}.
\end{definition}



In order to determine the type of $g$, the same as in
\cite{Pan2014a}, we introduce the following conditions
\begin{align}\label{g-integral}
&K:=\int_{0}^{B} \frac{1}{y \Phi^{-1}(B- y) } dy<+\infty.
\\\label{f-condition}
&f'(z)F(z) \lesssim f^2(z )  \qquad \text{for }\; z\in (0,A).
\end{align}
Here, the notation ``$\lesssim$'', which lies in between
``$\leqslant$'' and ``$<$'', means that except for at most finitely
many points where ``$=$'' holds, it is always ``$<$''.

\begin{example}[\cite{Pan2014a}]\label{eg:integral}
Consider the function
$\varphi_k(s)=\int_0^s(1+t^2)^{-\frac{k}{2}}\,\text{d}t$ which is
introduced in Example \ref{table:fi}.
A direct computation shows that \eqref{g-integral} is satisfied for
all $k>2$, e.g., $K=\frac{\pi}{2}$ when $k=3$ (i.e., the mean
curvature equation).
\end{example}

\begin{example}\label{eg:Fconcave}
The following table gives some functions $f$ satisfying  both
$f(0)=0$ and condition \eqref{f-condition}.
\begin{center}
\begin{tabular}{|c|l|}\hline
 $f$& $f(0)=0$\\\hline
   & $f(u)=(1+u)^p-1,\; p>0$\\
$f'(u)F(u) < f^2(u )$ & $f(u)=e^u-1$\\
&$f(u)=u^p,\; p>0$\\
&$f(u)=u^p+u^q,\; 0\leqslant p<q<\tilde{q}(p)$\\\hline $f'(u)F(u)
\lesssim f^2(u)$&$f(u)=u^p+u^q,\; 0\leqslant
p<q=\tilde{q}(p)$\\\hline
\end{tabular}
\end{center}
Here, the optimal upper bound $\tilde{q}(p)=p+1+2\sqrt{p+1}$ is
obtained in \cite{Hung2014}.
\end{example}

\begin{example}\label{eg:twolimit}
The following table gives some functions $f$ satisfying both
$f(0)=0$ and $\lim\limits_{t\rightarrow A} \frac{F(t)}{f(t)}=D$.
\begin{center}
\begin{tabular}{|c|l|l|l|}\hline
 $f$& $D=+\infty$ & $D\in(0,+\infty)$&$D=0$\\\hline
 & $f(u)=u^p, p>0$& $f(u)=e^u-1$& $f(u)=e^{u^2}-1$\\
 $f(0)=0$ & $f(u)=u^p+u^q, q>p>0$ & $f(u)=e^u+u^p-1, p>0$& $f(u)=e^{u^2}+u^p-1, p>0$\\
 & $f(u)=(1+u)^p-1, p>0$& $f(u)=u^pe^{ku}+u^q, p,q\geqslant 1,k>0$& $f(u)=(1-u)^{-p}-1, p>0$\\\hline
\end{tabular}
\end{center}
Besides, all of them satisfy $\lim_{t\rightarrow 0}
\frac{F(t)}{f(t)} =0$. These two limits of $\frac{F(t)}{f(t)}$ are
very useful for computing the limits of $g$ at $0$ and $+\infty$
(see Lemma \ref{prop:g} below).
\end{example}

\vskip 2mm {\bf Case IV}: $B<+\infty$, $A=+\infty$ and $C=+\infty$

\begin{theorem}[Type IV-$\alpha_0$, see Fig.\ref{fig:f1234}]\label{thm:f=0t3}
Let $(\varphi, f)$ be of \emph{Type IV-$\alpha_0$} (see
Fig.\ref{fig:gg1}).
Assume conditions  \eqref{con:phi}--\eqref{con:strict2} hold.
Then the following assertions hold:

(a)~If $f'(0)= 0$, then there exists $\lambda_*>0$ such that
\eqref{eq:1dmemsfr0} has exactly one positive solution for
 $\lambda\in(\lambda_*,+\infty)$ and none for $\lambda\in (0, \lambda_*]$.

(b)~If $f'(0)>0$, then there exists $\lambda_*\in (0,\lambda_1)$
such that \eqref{eq:1dmemsfr0} has exactly one positive solution for
$\lambda\in  (\lambda_*, \lambda_1)$ and none for $\lambda\in (0,
\lambda_*]\cup [\lambda_1,+\infty)$.

(c) $\lambda_*$ is strictly decreasing with respect to $L$.
\end{theorem}

\begin{corollary}[Type IV-$\alpha_0$, see Fig.\ref{fig:f1234}]\label{cor:f=0t3}
Let  $A=+\infty$, ${B<+\infty}$, ${C=+\infty}$. Assume conditions
\eqref{con:phi}--\eqref{con:strict2} hold. If further
\eqref{g-integral}, \eqref{f-condition} and
$\lim\limits_{z\rightarrow A} \frac{F(z)}{f(z)}=+\infty$ hold, then
$(\varphi, f)$ is of   \emph{Type IV-$\alpha_0$} and hence the
conclusions of Theorem \ref{thm:f=0t3} hold.
\end{corollary}

\begin{example}\label{eg:alpha}
Let $\varphi=\varphi_k\; (k>2)$ and $f(u)=u^p\; (p\geqslant 1)$ or
$f(u)=(1+u)^p-1\; (p>1)$ or $f(u)=u^p+u^q\; (1\leqslant
p<q\leqslant\tilde{q}(p))$. Then all conditions of Corollary
\ref{cor:f=0t3} are satisfied. Here $\tilde{q}(p)$ is given in
Example \ref{eg:Fconcave}. We notice that for $\varphi=\varphi_3$
and $f(u)=u^p+u^q$, this result is obtained in a forthcoming paper
\cite{Hung2014}.
\end{example}


Let us give an explanation of the bifurcation diagrams for Type
IV-$\alpha_0$ in Fig.\ref{fig:f1234}. The continuous thick line is
the bifurcation curve and represents classical solutions (i.e., $u
\in C^2[-L,L]$), while the thin curve $r=F^{-1}(\frac{B}{\lambda})$
is the gradient blow-up curve (see Section 3 for details). We also
note that intersection points of bifurcation curves and gradient
blow-up curves
actually represent another kind of positive solutions which belong
to $C[-L,L]\cap C^2(-L,L)$ and satisfy \eqref{eq:1dmemsfr0}, but
$u'(\pm L)=\mp \infty$.
For the other bifurcation diagrams in what follows, we shall use the
same legends as explained here (also see Remark
\ref{rk:varioussolutions} below).

\begin{theorem}[Type IV-$\alpha_1$, $f'(0)=0$, see Fig.\ref{fig:f1234}]\label{thm:f=0t4}
Let $(\varphi, f)$ be of \emph{Type IV-$\alpha_1$}. Assume
conditions \eqref{con:phi}--\eqref{con:strict2} and $f'(0)=0$ hold.
Then there exists a constant $L^*$ (see Fig.\ref{fig:gg1}) such that
the following assertions hold:

(a) If $L<L^*$, there exists $\lambda_*>0$ such that
\eqref{eq:1dmemsfr0} has exactly one positive solution for
 $\lambda\in(\lambda_*,+\infty)$ and none for $\lambda\in (0, \lambda_*]$.
Moreover, $\lambda_*$ is strictly decreasing with respect to $L$.

(b) If $L\geqslant L^*$, then \eqref{eq:1dmemsfr0} has exactly one
positive solution for any $\lambda\in (0,+\infty)$.
\end{theorem}

%


\begin{theorem}[Type IV-$\alpha_1$, $f'(0)>0$, see Fig.\ref{fig:f1234}]\label{thm:f=0t5}
Let $(\varphi, f)$ be of \emph{Type IV-$\alpha_1$}. Assume
conditions \eqref{con:phi}--\eqref{con:strict2} and $f'(0)>0$ hold.
Then there exists a constant $L^*$ (see Fig.\ref{fig:gg1}) such that
the following assertions hold:

(a) If $L<L^*$,   then there exists $\lambda_*\in (0,\lambda_1)$
such that \eqref{eq:1dmemsfr0} has exactly one positive solution for
$\lambda\in  (\lambda_*, \lambda_1)$ and none for $\lambda\in (0,
\lambda_*]\cup [\lambda_1,+\infty)$. Moreover, $\lambda_*$ is
strictly decreasing with respect to $L$.

(b) If $L\geqslant L^*$, then \eqref{eq:1dmemsfr0} has exactly one
positive solution for $\lambda\in (0, \lambda_1)$ and none for
$\lambda\in[\lambda_1,+\infty)$.
\end{theorem}

\begin{corollary}[Type IV-$\alpha_1$]\label{cor:f=0t5}
Let $A=+\infty$, ${B<+\infty}$, ${C=+\infty}$. Assume condition
\eqref{con:phi}--\eqref{con:strict2}  hold. If further
\eqref{g-integral}, \eqref{f-condition} and
$\lim\limits_{z\rightarrow A} \frac{F(z)}{f(z)}\in(0,+\infty)$ hold,
then $(\varphi, f)$ is of   \emph{Type IV-$\alpha_1$} and hence the
conclusions of Theorem \ref{thm:f=0t4} or \ref{thm:f=0t5} hold.
\end{corollary}

\begin{example}
Let $\varphi=\varphi_k\; (k>2)$ and $f(u)=e^u-1$ or $f(u)=e^u-u-1$.
Then all conditions of Corollary \ref{cor:f=0t5} are satisfied.
\end{example}


\begin{theorem}[Type IV-$\beta_0$, $f'(0)=0$, see Fig.\ref{fig:f00BIV4}]\label{thm:classical43320}
Let $(\varphi, f)$ be of \emph{Type IV-$\beta_0$}. Assume conditions
\eqref{con:phi}--\eqref{con:strict2} and $f'(0)=0$ hold. Then there
exists a constant $L^*>0$ (see Fig.\ref{fig:gg1}) such that the
following assertions hold:

(a) If $L> L^*$, then \eqref{eq:1dmemsfr0} has exactly one positive
solution for any $\lambda\in (0,+\infty)$.

(b) If $L= L^*$, then there exists $\lambda^*>0$ such that
\eqref{eq:1dmemsfr0} has exactly one positive solution for $(0,
\lambda^*)\cup(\lambda^*,+\infty)$ and none for $\lambda=\lambda^*$.

(c) If $L<L^*$, then there exist two numbers $\lambda^*>\lambda_*>0$
such that \eqref{eq:1dmemsfr0} has exactly one positive solution for
$\lambda\in  (0, \lambda_*)\cup(\lambda^*, +\infty)$ and none for
$\lambda\in [\lambda_*,\lambda^*]$.

(d) $\lambda^*$ is strictly decreasing while $\lambda_*$ is strictly
increasing with respect to $L$.
\end{theorem}

  \ifpdf
\begin{figure}
\centering
\includegraphics[totalheight=8.3in]{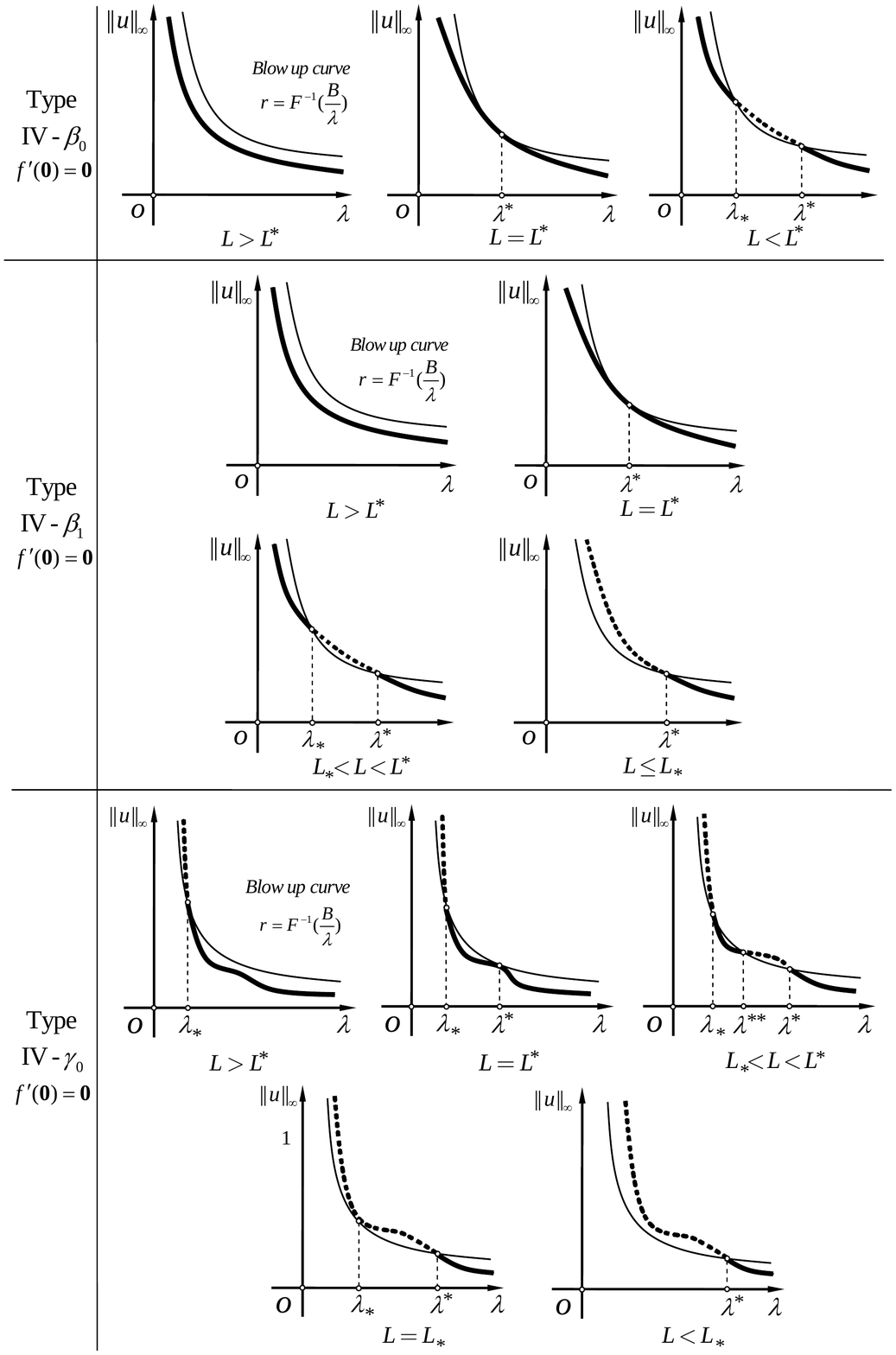}
\caption{Bifurcation Diagrams for Types IV-$\beta_0$, IV-$\beta_1$
and IV-$\gamma_0$ with $f(0)=0$ and $f'(0)=0$.}\label{fig:f00BIV4}
\end{figure}
\begin{figure}
\centering
\includegraphics[totalheight=8.3in]{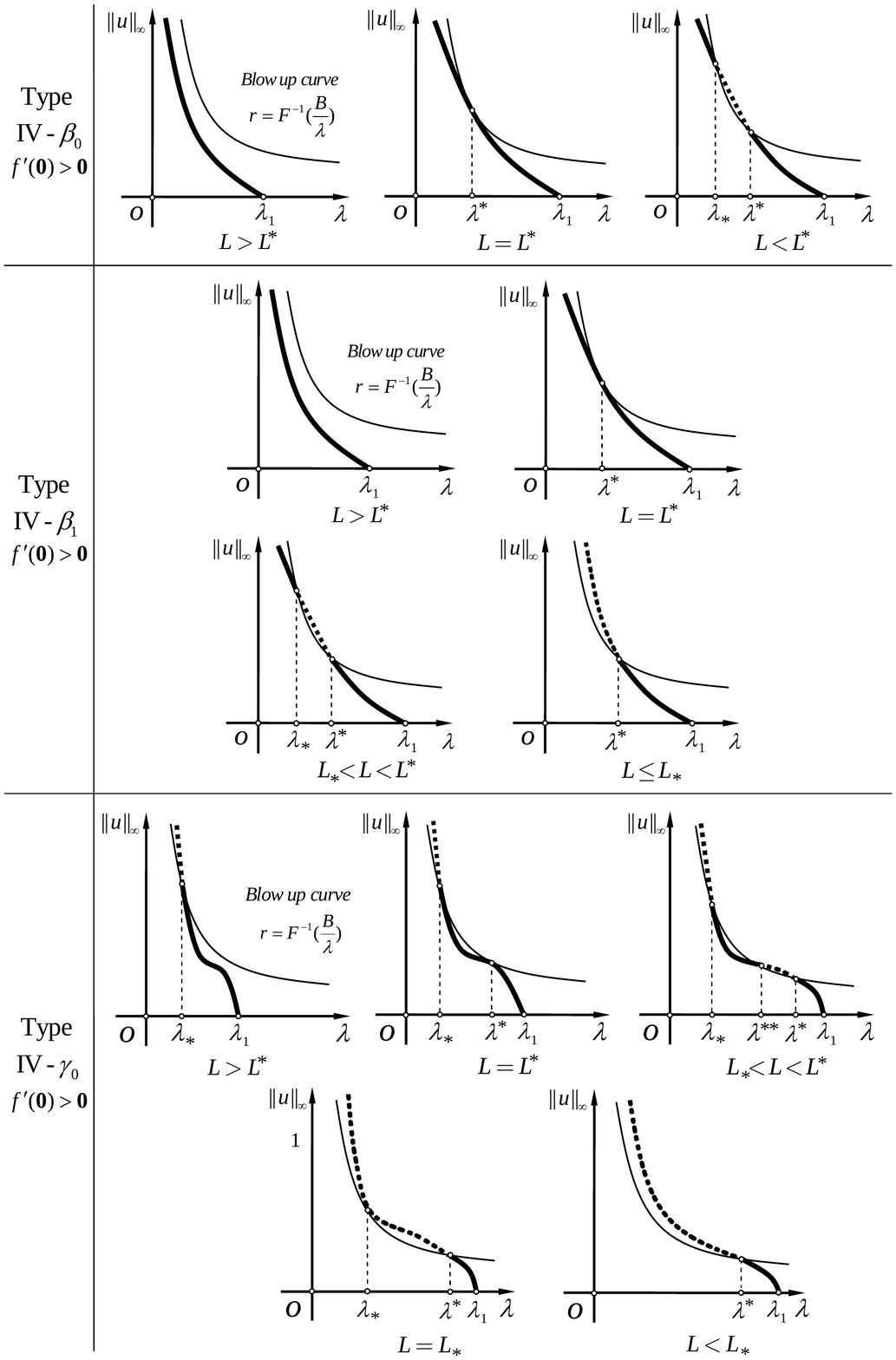}
\caption{Bifurcation Diagrams for Types IV-$\beta_0$, IV-$\beta_1$,
and IV-$\gamma_0$ with $f(0)=0$ and $f'(0)>0$.}\label{fig:f01BIV4}
\end{figure}
  \fi



\begin{theorem}[Type IV-$\beta_0$, $f'(0)>0$, see Fig.\ref{fig:f01BIV4}]\label{thm:classical43321}
Let $(\varphi, f)$ be of \emph{Type IV-$\beta_0$}. Assume conditions
\eqref{con:phi}--\eqref{con:strict2} and $f'(0)>0$ hold. Then there
exists a constant $L^*>0$ (see Fig.\ref{fig:gg1}) such that the
following assertions hold:

(a) If $L> L^*$, then \eqref{eq:1dmemsfr0} has exactly one positive
solution for $\lambda\in (0, \lambda_1)$ and none for
$\lambda\in[\lambda_1,+\infty)$.

(b) If $L= L^*$, then there exists $\lambda^*\in (0,\lambda_1)$ such
that \eqref{eq:1dmemsfr0} has exactly one positive solution for $(0,
\lambda^*)\cup(\lambda^*,\lambda_1)$ and none for
$\lambda\in\{\lambda^*\}\cup[\lambda_1,+\infty)$.

(c) If $L<L^*$, then there exist two numbers
$0<\lambda_*<\lambda^*<\lambda_1$ such that \eqref{eq:1dmemsfr0} has
exactly one positive solution for $\lambda\in  (0,
\lambda_*)\cup(\lambda^*, \lambda_1)$ and none for
$\lambda\in[\lambda_*,\lambda^*]\cup[\lambda_1,+\infty)$.

(d) $\lambda^*$ is strictly decreasing while $\lambda_*$ is strictly
increasing with respect to $L$.
\end{theorem}

\begin{example}\label{eg:beta00}
Let $\varphi=\varphi_3$ and $f(u)=e^{u^2}-1$ or $f(u)=e^{u^2}+u-1$.
Then $\lim_{\lambda\rightarrow +\infty}
g(\lambda)=0=\lim_{\lambda\rightarrow 0} g(\lambda)$ (by Lemma
\ref{prop:g} below). Numerical simulation indicates that
$g(\lambda)$ has exactly one local extreme point in $(0, +\infty)$
(see Remark \ref{rk:g(r)} and Fig.\ref{fig:gg3} below). We provide
an analytic proof for this in \cite{Pan2014c}. Thus $g$ is of Type
$\beta_0$ and hence both $(\varphi_3,e^{u^2}-1)$ and
$(\varphi_3,e^{u^2}+u-1)$ are of Type IV-$\beta_0$. From Theorems
\ref{thm:classical43320} and \ref{thm:classical43321}, we obtain the
exact numbers of positive solutions as well as global bifurcation
diagrams.
\end{example}


\begin{theorem}[Type IV-$\beta_1$, $f'(0)=0$, see Fig.\ref{fig:f00BIV4}]\label{thm:classical43330}
Let $(\varphi, f)$ be of \emph{Type IV-$\beta_1$}. Assume conditions
\eqref{con:phi}--\eqref{con:strict2} and $f'(0)=0$ hold. Then there
exist constants $L^*>L_*>0$ (see Fig.\ref{fig:gg1}) such that the
following assertions hold:

(a) If $L> L^*$, then \eqref{eq:1dmemsfr0} has exactly one positive
solution for any $\lambda\in (0,+\infty)$.

(b) If $L= L^*$, then there exists $\lambda^*>0$ such that
\eqref{eq:1dmemsfr0} has exactly one positive solution for $(0,
\lambda^*)\cup(\lambda^*,+\infty)$ and none for $\lambda=\lambda^*$.

(c) If $L_*<L<L^*$, then there exist two numbers
$\lambda^*>\lambda_*>0$ such that \eqref{eq:1dmemsfr0} has exactly
one positive solution for $\lambda\in  (0, \lambda_*)\cup(\lambda^*,
+\infty)$ and none for $\lambda\in [\lambda_*,\lambda^*]$.

(d) If $L\leqslant L_*$, then there exists $\lambda^*>0$ such that
\eqref{eq:1dmemsfr0} has exactly one positive solution for
$\lambda\in  (\lambda^*, +\infty)$ and none for $\lambda\in (0,
\lambda_*]$.

(e) $\lambda^*$ is strictly decreasing while $\lambda_{*}$ is
strictly increasing with respect to $L$.
\end{theorem}


\begin{theorem}[Type IV-$\beta_1$, $f'(0)>0$, see Fig.\ref{fig:f01BIV4}]\label{thm:classical43331}
Let $(\varphi, f)$ be of \emph{Type IV-$\beta_1$}. Assume conditions
\eqref{con:phi}--\eqref{con:strict2} and $f'(0)>0$ hold. Then there
exist constants $L^*>L_*>0$ (see Fig.\ref{fig:gg1}) such that the
following assertions hold:

(a) If $L> L^{*}$, then \eqref{eq:1dmemsfr0} has exactly one
positive solution for $\lambda\in (0, \lambda_1)$ and none for
$\lambda\in[\lambda_1,+\infty)$.

(b) If $L= L^{*}$, then there exists $\lambda^*\in (0,\lambda_1)$
such that \eqref{eq:1dmemsfr0} has exactly one positive solution for
$(0, \lambda^*)\cup(\lambda^*,\lambda_1)$ and none for
$\lambda\in\{\lambda^*\}\cup[\lambda_1,+\infty)$.

(c) If $L_{*}<L< L^{*}$, then there exist two numbers
$0<\lambda_{*}<\lambda^*<\lambda_1$ such that \eqref{eq:1dmemsfr0}
has exactly one positive solution for $\lambda\in  (0,
\lambda_{*})\cup(\lambda^*, \lambda_1)$ and none for
$\lambda\in[\lambda_{*},\lambda^*]\cup[\lambda_1,+\infty)$.

(d) If $L\leqslant L_*$, then there exists $\lambda^*\in
(0,\lambda_1)$ such that \eqref{eq:1dmemsfr0} has exactly one
positive solutions for $\lambda\in  (\lambda^*, \lambda_1)$ and none
for $\lambda\in  (0, \lambda^*]\cup[\lambda_1, +\infty)$.

(e) $\lambda^*$ is strictly decreasing while $\lambda_{*}$ is
strictly increasing with respect to $L$.
\end{theorem}


\begin{example}
Let $\varphi=\varphi_3$ and $f(u)=e^{u}+u^2-u-1$ or
$f(u)=e^{u}+u-1$. Then $\lim_{\lambda\rightarrow +\infty}
g(\lambda)=0$ and $\lim_{\lambda\rightarrow 0}
g(\lambda)=\frac{\pi}{2}$ (by Lemma \ref{prop:g} below). Numerical
simulation indicates that $g(\lambda)$ has exactly one local extreme
point in $(0, +\infty)$ (see Remark \ref{rk:g(r)} and
Fig.\ref{fig:gg3} below). We provide an analytic proof for this in
\cite{Pan2014c}. Thus $g$ is of   Type $\beta_1$ and hence both
$(\varphi_3,e^{u}+u^2-u-1)$ and $(\varphi_3,e^{u}+u-1)$ are of Type
IV-$\beta_1$. From Theorems \ref{thm:classical43330} and
\ref{thm:classical43331}, we obtain the exact numbers of positive
solutions as well as global bifurcation diagrams.
\end{example}

\begin{theorem}[Type IV-$\gamma_0$, $f'(0)=0$, see Fig.\ref{fig:f00BIV4}]\label{thm:f=0t6}
Let $(\varphi, f)$ be of \emph{Type IV-$\gamma_0$}. Assume
conditions \eqref{con:phi}--\eqref{con:strict2} and $f'(0)=0$ hold.
Then there exist constants $L^*>L_*>0$ (see Fig.\ref{fig:gg1}) such
that the following assertions hold:

(a) If $L>L^*$, then there exists $\lambda_*>0$ such that
\eqref{eq:1dmemsfr0} has exactly one positive solution for
$\lambda\in(\lambda_*,+\infty)$ and none for $\lambda\in (0,
\lambda_*]$.

(b) If $L=L^*$, then there exist two numbers
$\lambda^*>\lambda_*>0$ such that \eqref{eq:1dmemsfr0} has exactly
one positive solution for $\lambda\in (\lambda_*,
\lambda^{*})\cup(\lambda^{*},+\infty)$ and none for $\lambda\in (0,
\lambda_*]\cup\{\lambda^{*}\}$.

(c) If $L_*<L<L^*$, then there exist three numbers
$\lambda^*>\lambda^{**}>\lambda_*>0$
such that \eqref{eq:1dmemsfr0} has exactly one positive solution for
$\lambda\in (\lambda_*, \lambda^{**})\cup (\lambda^*,+\infty)$ and
none for $\lambda\in (0, \lambda_*]\cup[\lambda^{**},\lambda^*]$.

(d) If $L\leqslant L_*$, then there exists $\lambda^*>0$ such that
\eqref{eq:1dmemsfr0} has exactly one positive solution for
$\lambda\in (\lambda^*,+\infty)$ and none for $\lambda\in (0,
\lambda^*]$.

(e) $\lambda_*$ and $\lambda^*$ are strictly decreasing while
$\lambda^{**}$ is strictly increasing with respect to $L$.
\end{theorem}

\begin{theorem}[Type IV-$\gamma_0$, $f'(0)>0$, see Fig.\ref{fig:f01BIV4}]\label{thm:f=0t7}
Let $(\varphi, f)$ be of \emph{Type IV-$\gamma_0$}. Assume
conditions \eqref{con:phi}--\eqref{con:strict2} and $f'(0)>0$ hold.
Then there exist constants $L^*>L_*>0$ (see Fig.\ref{fig:gg1}) such
that the following assertions hold:

(a) If $L> L^*$, then there exists $\lambda_*\in (0,\lambda_1)$ such
that \eqref{eq:1dmemsfr0} has exactly one positive solution for
$(\lambda_*,\lambda_1)$ and none for $\lambda\in(0,
\lambda_*]\cup[\lambda_1,+\infty)$.

(b) If $L=L^*$, then there exist two numbers
$0<\lambda_*<\lambda^{*}<\lambda_1$ such that \eqref{eq:1dmemsfr0}
has exactly one positive solution for $\lambda\in (\lambda_*,
\lambda^{*})\cup(\lambda^{*},\lambda_1)$ and none for $\lambda\in
(0, \lambda_*]\cup\{\lambda^{*}\}\cup[\lambda_1,+\infty)$.

(c) If $L_*<L< L^*$, three numbers
$0<\lambda_{*}<\lambda^{**}<\lambda^*<\lambda_1$ such that
\eqref{eq:1dmemsfr0} has exactly one positive solution for
$\lambda\in (\lambda_*, \lambda^{**})\cup(\lambda^{*},\lambda_1)$,
none for $\lambda\in (0,
\lambda_*]\cup[\lambda^{**},\lambda^{*}]\cup[\lambda_1,+\infty)$.

(d) If $L\leqslant L_*$, then there exists $\lambda^*\in
(0,\lambda_1)$ such that \eqref{eq:1dmemsfr0} has exactly one
positive solution for $(\lambda^*,\lambda_1)$ and none for
$\lambda\in(0, \lambda^*]\cup[\lambda_1,+\infty)$.

(e) $\lambda_*$ and $\lambda^*$ are strictly decreasing while
$\lambda^{**}$ is strictly increasing with respect to $L$.
\end{theorem}

\begin{example}
Let $\varphi=\varphi_3$ and $f(u)=u^2+u^7$ or $f(u)=u+u^6$. Then
$\lim_{\lambda\rightarrow +\infty} g(\lambda)=0$ and
$\lim_{\lambda\rightarrow 0} g(\lambda)=+\infty$ (by Lemma
\ref{prop:g} below). Numerical simulation indicates that
$g(\lambda)$ has exactly two local extreme points in $(0, +\infty)$
(see Remark \ref{rk:g(r)} and Fig.\ref{fig:gg3} below). We provide
an analytic proof for this in \cite{Pan2014c}. Thus $g$ is of Type
$\gamma_0$ and hence both $(\varphi_3,u^2+u^7)$ and
$(\varphi_3,u+u^6)$ are of Type IV-$\gamma_0$. From Theorems
\ref{thm:f=0t6} and \ref{thm:f=0t7}, we obtain the exact numbers of
positive solutions as well as global bifurcation diagrams.
\end{example}

\begin{theorem}[Type IV-$\gamma_1$, $f'(0)=0$, see Fig.\ref{fig:f00BIV45a}]\label{thm:f=0t77}
Let $(\varphi, f)$ be of \emph{Type IV-$\gamma_1$}. Assume
conditions \eqref{con:phi}--\eqref{con:strict2} and $f'(0)>0$ hold.
Then there exist constants $L^{**}>L^*>L_*>0$ (see
Fig.\ref{fig:gg1}) such that the following assertions hold:

(a) If $L\geqslant L^{**}$, then \eqref{eq:1dmemsfr0} has exactly
one positive solution for any $\lambda\in (0,+\infty)$.

(b) If $L^*<L< L^{**}$, then there exists $\lambda_*>0$ such that
\eqref{eq:1dmemsfr0} has exactly one positive solution for
$\lambda\in(\lambda_*,+\infty)$ and none for $\lambda\in (0,
\lambda_*]$.

(c) If $L=L^*$, then there exist two numbers
$\lambda^*>\lambda_*>0$ such that \eqref{eq:1dmemsfr0} has exactly
one positive solution for $\lambda\in (\lambda_*,
\lambda^{*})\cup(\lambda^{*},+\infty)$ and none for $\lambda\in (0,
\lambda_*]\cup\{\lambda^{*}\}$.

(d) If $L_*<L<L^*$, then there exist three numbers
$\lambda^*>\lambda^{**}>\lambda_*>0$
such that \eqref{eq:1dmemsfr0} has exactly one positive solution for
$\lambda\in (\lambda_*, \lambda^{**})\cup (\lambda^*,+\infty)$ and
none for $\lambda\in (0, \lambda_*]\cup[\lambda^{**},\lambda^*]$.

(e) If $L\leqslant L_*$, then there exists $\lambda^*>0$ such that
\eqref{eq:1dmemsfr0} has exactly one positive solution for
$\lambda\in (\lambda^*,+\infty)$ and none for $\lambda\in (0,
\lambda^*]$.

(f) $\lambda_*$ and $\lambda^*$ are strictly decreasing while
$\lambda^{**}$ is strictly increasing with respect to $L$.
\end{theorem}

  \ifpdf
\begin{figure}
\centering
\includegraphics[totalheight=8.8in]{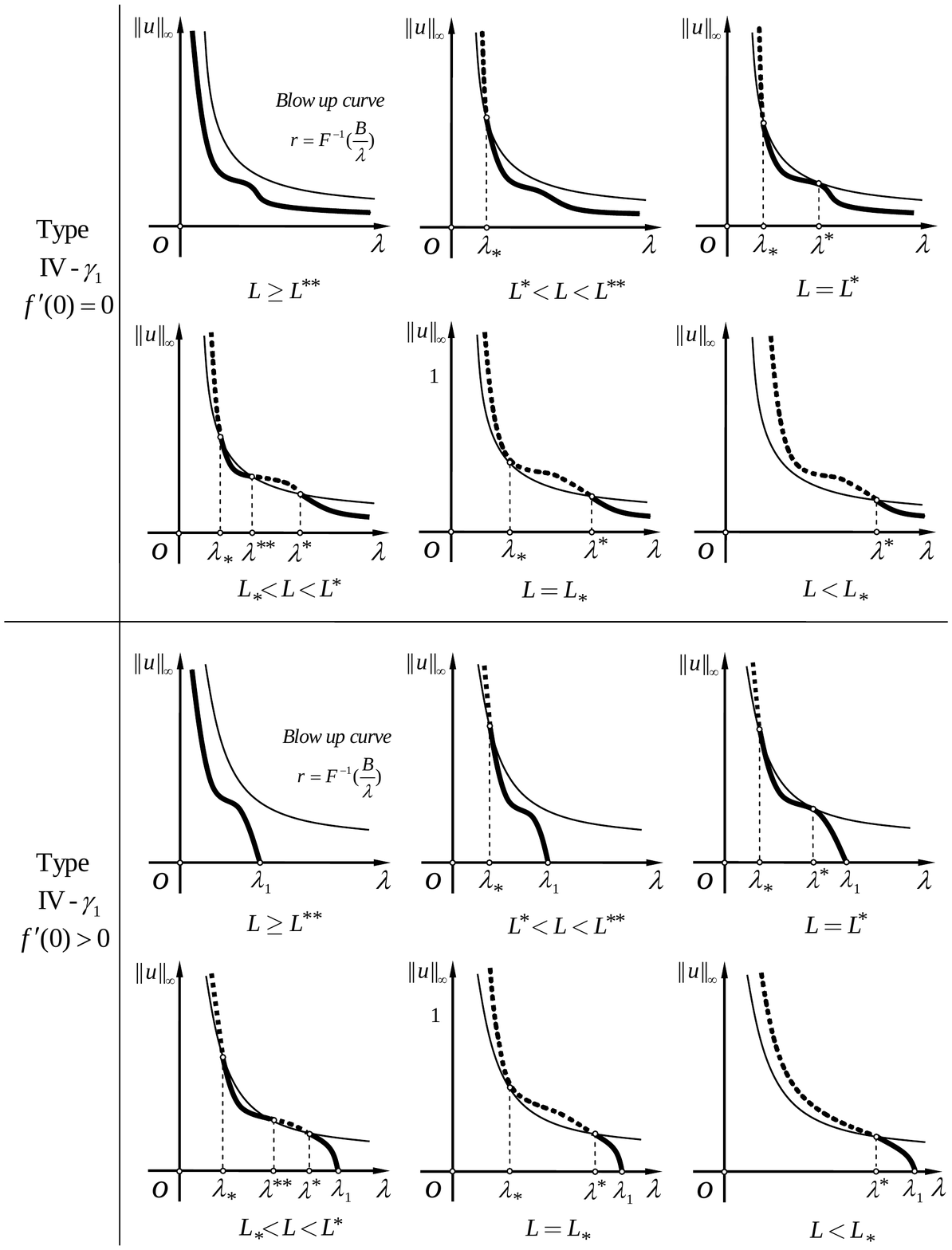}
\caption{Bifurcation Diagrams for Type IV-$\gamma_1$ with
$f(0)=0$.}\label{fig:f00BIV45a}
\end{figure}
  \fi

\begin{theorem}[Type IV-$\gamma_1$, $f'(0)>0$, see Fig.\ref{fig:f00BIV45a}]\label{thm:f=0t777}
Let $(\varphi, f)$ be of \emph{Type IV-$\gamma_1$}. Assume
conditions \eqref{con:phi}--\eqref{con:strict2} and $f'(0)>0$ hold.
Then there exist constants $L^{**}>L^*>L_*>0$ (see
Fig.\ref{fig:gg1}) such that the following assertions hold:

(a) If $L\geqslant L^{**}$, then \eqref{eq:1dmemsfr0} has exactly
one positive solution for $\lambda\in (0, \lambda_1)$ and none for
$\lambda\in[\lambda_1,+\infty)$.

(b) If $L^*<L< L^{**}$, then there exists $\lambda_*\in
(0,\lambda_1)$ such that \eqref{eq:1dmemsfr0} has exactly one
positive solution for $(\lambda_*,\lambda_1)$ and none for
$\lambda\in(0, \lambda_*]\cup[\lambda_1,+\infty)$.

(c) If $L=L^*$, then there exist two numbers
$0<\lambda_*<\lambda^{*}<\lambda_1$ such that \eqref{eq:1dmemsfr0}
has exactly one positive solution for $\lambda\in (\lambda_*,
\lambda^{*})\cup(\lambda^{*},\lambda_1)$ and none for $\lambda\in
(0, \lambda_*]\cup\{\lambda^{*}\}\cup[\lambda_1,+\infty)$.

(d) If $L_*<L< L^*$, three numbers
$0<\lambda_{*}<\lambda^{**}<\lambda^*<\lambda_1$ such that
\eqref{eq:1dmemsfr0} has exactly one positive solution for
$\lambda\in (\lambda_*, \lambda^{**})\cup(\lambda^{*},\lambda_1)$
and none for $\lambda\in (0,
\lambda_*]\cup[\lambda^{**},\lambda^{*}]\cup[\lambda_1,+\infty)$.

(e) If $L\leqslant L_*$, then there exists $\lambda^*\in
(0,\lambda_1)$ such that \eqref{eq:1dmemsfr0} has exactly one
positive solution for $(\lambda^*,\lambda_1)$ and none for
$\lambda\in(0, \lambda^*]\cup[\lambda_1,+\infty)$.

(f) $\lambda_*$ and $\lambda^*$ are strictly decreasing while
$\lambda^{**}$ is strictly increasing with respect to $L$.
\end{theorem}

\begin{example}\label{eg:gamma1}
Let $\varphi=\varphi_3$ and $f(u)=u^7e^u+u^2$ or $f(u)=u^5e^u+u$.
Then $\lim_{\lambda\rightarrow +\infty} g(\lambda)=0$ and
$\lim_{\lambda\rightarrow 0} g(\lambda)=\frac{\pi}{2}$ (by Lemma
\ref{prop:g} below). Numerical simulation indicates that
$g(\lambda)$ has exactly two local extreme points in $(0, +\infty)$
and $\lim_{\lambda\rightarrow 0} g(\lambda)$ is greater than the
local maximum value (see Remark \ref{rk:g(r)} and Fig.\ref{fig:gg3}
below), which suggests that $g$ is of   Type $\gamma_1$ and both
$(\varphi_3,u^7e^u+u^2)$ and $(\varphi_3,u^5e^u+u)$ are of Type
IV-$\gamma_1$.
\end{example}

%

\begin{theorem}[Type IV-$\gamma_2$, $f'(0)=0$, see Fig.\ref{fig:f00BIV45+}]\label{thm:f=0t77new}
Let $(\varphi, f)$ be of \emph{Type IV-$\gamma_2$}. Assume
conditions \eqref{con:phi}--\eqref{con:strict2} and $f'(0)=0$ hold.
Then there exist constants $L^{**}>L^*>L_*>0$ (see
Fig.\ref{fig:gg1}) such that the following assertions hold:

(a) If $L> L^{**}$, then \eqref{eq:1dmemsfr0} has exactly one
positive solution for any $\lambda\in (0,+\infty)$.

(b) If $L= L^{**}$, then there exists $\lambda^*>0$ such that
\eqref{eq:1dmemsfr0} has exactly one positive solution for $(0,
\lambda^{*})\cup(\lambda^{*},+\infty)$ and none for
$\lambda=\lambda^{*}$.

(c) If $L^*\leqslant L< L^{**}$, then there exist two numbers
$\lambda^*>\lambda^{**}>0$ such that \eqref{eq:1dmemsfr0} has
exactly one positive solutions for $\lambda\in  (0,
\lambda^{**})\cup(\lambda^*, +\infty)$ and none for $\lambda\in
[\lambda^{**}, \lambda^*]$.

(d) If $L_*<L< L^*$, then there exist three numbers
$\lambda^*>\lambda^{**}>\lambda_*>0$ such that \eqref{eq:1dmemsfr0}
has exactly one positive solutions for $\lambda\in  (\lambda_*,
\lambda^{**})\cup(\lambda^*, +\infty)$ and none for $\lambda\in  (0,
\lambda_*]\cup[\lambda^{**}, \lambda^*]$.

(e) If $L\leqslant L_*$, then there exists $\lambda^*>0$ such that
\eqref{eq:1dmemsfr0} has exactly one positive solutions for
$\lambda\in  (\lambda^*, +\infty)$ and none for $\lambda\in  (0,
\lambda^*]$.

(f) $\lambda_*$ and $\lambda^*$ are strictly decreasing while
$\lambda^{**}$ is strictly increasing with respect to $L$.
\end{theorem}

  \ifpdf
\begin{figure}
\centering
\includegraphics[totalheight=8.3in]{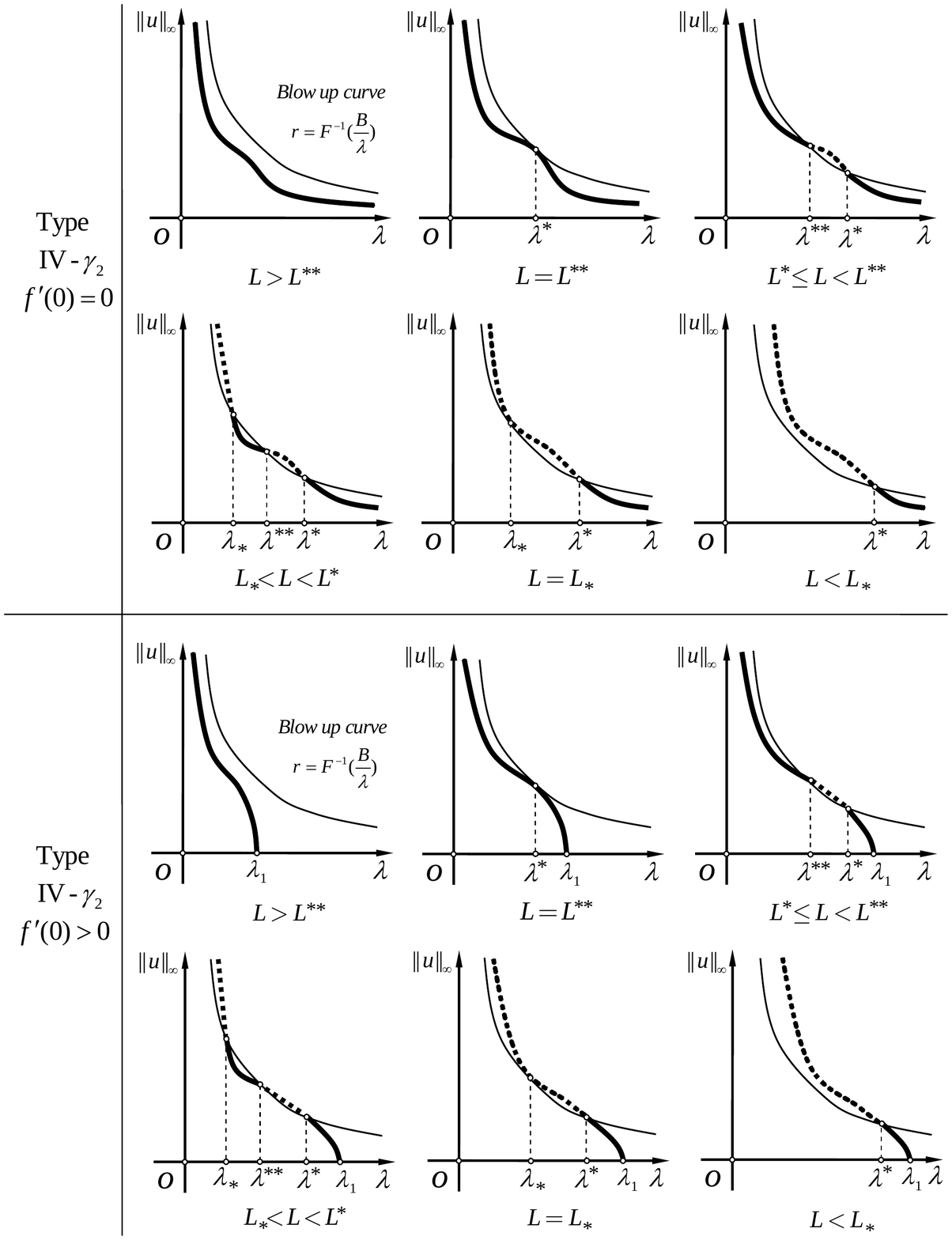}
\caption{Bifurcation Diagrams for Type IV-$\gamma_2$ with
$f(0)=0$.}\label{fig:f00BIV45+}
\end{figure}
  \fi

\begin{theorem}[Type IV-$\gamma_2$, $f'(0)>0$, see Fig.\ref{fig:f00BIV45+}]\label{thm:f=0t777new}
Let $(\varphi, f)$ be of \emph{Type IV-$\gamma_2$}. Assume
conditions \eqref{con:phi}--\eqref{con:strict2} and $f'(0)>0$ hold.
Then there exist constants $L^{**}>L^*>L_*>0$ (see
Fig.\ref{fig:gg1}) such that the following assertions hold:

(a) If $L> L^{**}$, then \eqref{eq:1dmemsfr0} has exactly one
positive solution for $\lambda\in (0, \lambda_1)$ and none for
$\lambda\in[\lambda_1,+\infty)$.

(b) If $L= L^{**}$, then there exists $\lambda^*\in (0,\lambda_1)$
such that \eqref{eq:1dmemsfr0} has exactly one positive solution for
$(0, \lambda^*)\cup(\lambda^*,\lambda_1)$ and none for
$\lambda\in\{\lambda^*\}\cup[\lambda_1,+\infty)$.

(c) If $L^*\leqslant L< L^{**}$, then there exist two numbers
$0<\lambda^{**}<\lambda^*<\lambda_1$ such that \eqref{eq:1dmemsfr0}
has exactly one positive solution for $\lambda\in  (0,
\lambda^{**})\cup(\lambda^*, \lambda_1)$ and none for
$\lambda\in[\lambda^{**},\lambda^*]\cup[\lambda_1,+\infty)$.

(d) If $L_*<L< L^*$, then there exist three numbers
$0<\lambda_{*}<\lambda^{**}<\lambda^*<\lambda_1$ such that
\eqref{eq:1dmemsfr0} has exactly one positive solutions for
$\lambda\in  (\lambda_*, \lambda^{**})\cup(\lambda^{*}, \lambda_1)$
and none for $\lambda\in  (0, \lambda_{*}]\cup[\lambda^{**},
\lambda^{*}]\cup[\lambda_1, +\infty)$.

(e) If $L\leqslant L_*$, then there exists $\lambda^*\in
(0,\lambda_1)$ such that \eqref{eq:1dmemsfr0} has exactly one
positive solutions for $\lambda\in  (\lambda^{*}, \lambda_1)$ and
none for $\lambda\in  (0, \lambda^{*}]\cup[\lambda_1, +\infty)$.

(f) $\lambda_*$ and $\lambda^*$ are strictly decreasing while
$\lambda^{**}$ is strictly increasing with respect to $L$.
\end{theorem}

\begin{example}\label{eg:gamma2}
Let $\varphi=\varphi_3$ and $f(u)=u^7e^{12u}+u^2$ or
$f(u)=u^5e^{8u}+u$. Then $\lim_{\lambda\rightarrow +\infty}
g(\lambda)=0$ and $\lim_{\lambda\rightarrow 0}
g(\lambda)=\frac{\pi}{24}$ or $\frac{\pi}{16}$ (by Lemma
\ref{prop:g} below). Numerical simulation indicates that
$g(\lambda)$ has exactly two local extreme points in $(0, +\infty)$
and $\lim_{\lambda\rightarrow 0}g(\lambda)$ is in between the local
extreme values
(see Remark \ref{rk:g(r)} and Fig.\ref{fig:gg3} below), which
suggests that $g$ is of   Type $\gamma_1$ and both
$(\varphi_3,u^7e^{12u}+u^2)$ and $(\varphi_3,u^5e^{8u}+u)$ are of
Type IV-$\gamma_2$
\end{example}

\begin{theorem}[Type IV-$\gamma_3$, $f'(0)=0$, see Fig.\ref{fig:f00BIV45}]\label{thm:f=0t77newplus}
Let $(\varphi, f)$ be of \emph{Type IV-$\gamma_3$}. Assume
conditions \eqref{con:phi}--\eqref{con:strict2} and $f'(0)=0$ hold.
Then there exist constants $L^*>L_*>0$ (see Fig.\ref{fig:gg1}) such
that the following assertions hold:

(a) If $L> L^{*}$, then \eqref{eq:1dmemsfr0} has exactly one
positive solution for any $\lambda\in (0,+\infty)$.

(b) If $L= L^{*}$, then there exists $\lambda^*>0$ such that
\eqref{eq:1dmemsfr0} has exactly one positive solution for $(0,
\lambda^{*})\cup(\lambda^{*},+\infty)$ and none for
$\lambda=\lambda^{*}$.

(c)  If $L_*<L< L^*$, then there exist three numbers
$\lambda^*>\lambda^{**}>\lambda_*>0$ such that \eqref{eq:1dmemsfr0}
has exactly one positive solutions for $\lambda\in  (\lambda_*,
\lambda^{**})\cup(\lambda^*, +\infty)$ and none for $\lambda\in  (0,
\lambda_*]\cup[\lambda^{**}, \lambda^*]$.

(d) If $L\leqslant L_*$, then there exists $\lambda^*>0$ such that
\eqref{eq:1dmemsfr0} has exactly one positive solutions for
$\lambda\in  (\lambda^*, +\infty)$ and none for $\lambda\in  (0,
\lambda^*]$.

(e) $\lambda_*$ and $\lambda^*$ are strictly decreasing while
$\lambda^{**}$ is strictly increasing with respect to $L$.
\end{theorem}

  \ifpdf
\begin{figure}
\centering
\includegraphics[totalheight=8.3in]{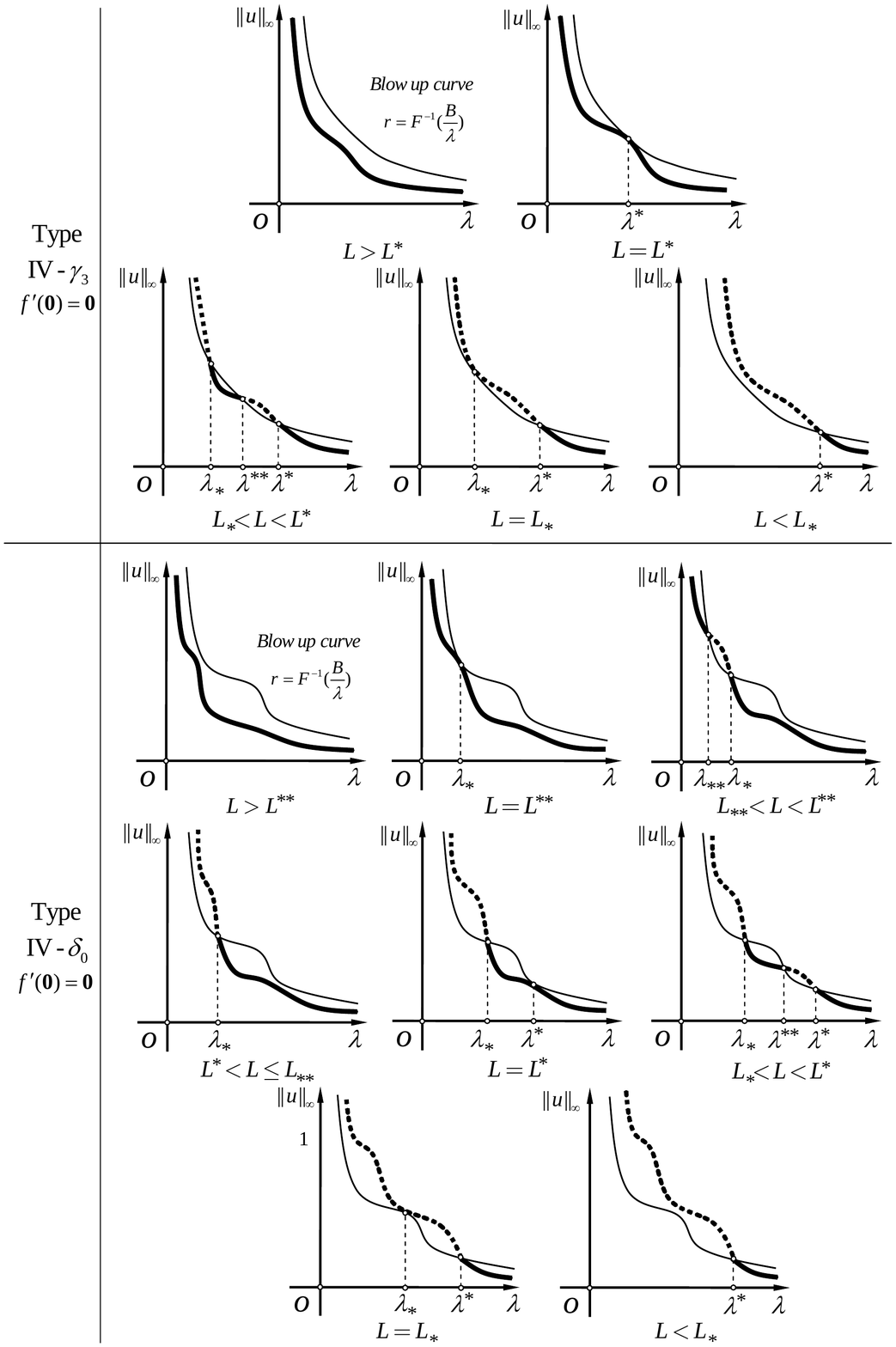}
\caption{Bifurcation Diagrams for Types IV-$\gamma_3$ and
IV-$\delta_0$ with $f(0)=0$ and $f'(0)=0$.}\label{fig:f00BIV45}
\end{figure}
\begin{figure}
\centering
\includegraphics[totalheight=8.3in]{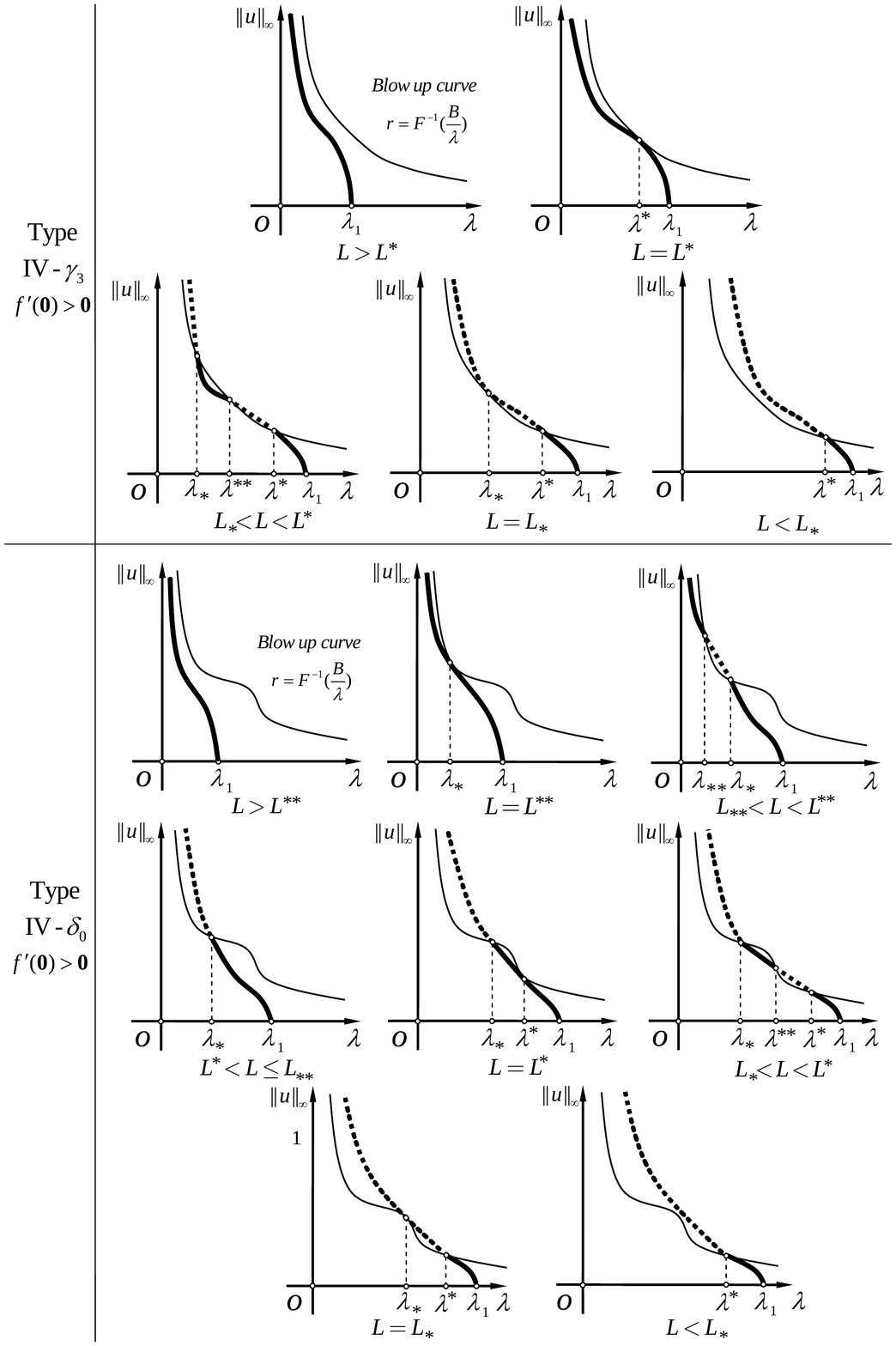}
\caption{Bifurcation Diagrams for Types IV-$\gamma_3$ and
IV-$\delta_0$ with $f(0)=0$ and $f'(0)>0$.}\label{fig:f00BIV5}
\end{figure}
  \fi

\begin{theorem}[Type IV-$\gamma_3$, $f'(0)>0$, see Fig.\ref{fig:f00BIV5}]\label{thm:f=0t777newplus}
Let $(\varphi, f)$ be of \emph{Type IV-$\gamma_3$}. Assume
conditions \eqref{con:phi}--\eqref{con:strict2} and $f'(0)>0$ hold.
Then there exist constants $L^*>L_*>0$ (see Fig.\ref{fig:gg1}) such
that the following assertions hold:

(a) If $L> L^{*}$, then \eqref{eq:1dmemsfr0} has exactly one
positive solution for $\lambda\in (0, \lambda_1)$ and none for
$\lambda\in[\lambda_1,+\infty)$.

(b) If $L= L^{*}$, then there exists $\lambda^*\in (0,\lambda_1)$
such that \eqref{eq:1dmemsfr0} has exactly one positive solution for
$(0, \lambda^*)\cup(\lambda^*,\lambda_1)$ and none for
$\lambda\in\{\lambda^*\}\cup[\lambda_1,+\infty)$.

(c) If $L_*<L< L^*$, then there exist three numbers
$0<\lambda_{*}<\lambda^{**}<\lambda^*<\lambda_1$ such that
\eqref{eq:1dmemsfr0} has exactly one positive solutions for
$\lambda\in  (\lambda_*, \lambda^{**})\cup(\lambda^{*}, \lambda_1)$
and none for $\lambda\in  (0, \lambda_{*}]\cup[\lambda^{**},
\lambda^{*}]\cup[\lambda_1, +\infty)$.

(d) If $L\leqslant L_*$, then there exists $\lambda^*\in
(0,\lambda_1)$ such that \eqref{eq:1dmemsfr0} has exactly one
positive solutions for $\lambda\in  (\lambda^{*}, \lambda_1)$ and
none for $\lambda\in  (0, \lambda^{*}]\cup[\lambda_1, +\infty)$.

(e) $\lambda_*$ and $\lambda^*$ are strictly decreasing while
$\lambda^{**}$ is strictly increasing with respect to $L$.
\end{theorem}

\begin{example}\label{eg:gamma3}
Let $\varphi=\varphi_3$. By Lemma \ref{prop:g} and a comparison
between Examples \ref{eg:gamma1} and \ref{eg:gamma2} (also see
Figs.\ref{fig:gg2} and \ref{fig:gg3}), we conjecture that there
exist $k_1\in (1,12)$ and $k_2\in (1,8)$ such that
$(\varphi_3,u^7e^{k_1u}+u^2)$ and $(\varphi_3,u^5e^{k_2u}+u)$ are of
Type IV-$\gamma_3$, i.e., $g$ has exactly two local extreme points
and the local maximum value is equal to $\lim_{\lambda\rightarrow
0}g(\lambda)$.
\end{example}


\begin{theorem}[Type IV-$\delta_0$, $f'(0)=0$, see Fig.\ref{fig:f00BIV45}]\label{thm:f=0t8}
Let $(\varphi, f)$ be of \emph{Type IV-$\delta_0$}. Assume
conditions \eqref{con:phi}--\eqref{con:strict2} and $f'(0)=0$ hold.
Then there exist constants $L^{**}>L_{**}>L^*>L_*>0$ (see
Fig.\ref{fig:gg1}) such that the following assertions hold:

(a) If $L> L^{**}$, then \eqref{eq:1dmemsfr0} has exactly one
positive solution for any $\lambda\in (0,+\infty)$.

(b) If $L= L^{**}$, then there exists $\lambda_*>0$ such that
\eqref{eq:1dmemsfr0} has exactly one positive solution for $(0,
\lambda_{*})\cup(\lambda_{*},+\infty)$ and none for
$\lambda=\lambda_{*}$.

(c) If $L_{**}<L< L^{**}$, then there exist two numbers
$\lambda_*>\lambda_{**}>0$ such that \eqref{eq:1dmemsfr0} has
exactly one positive solutions for $\lambda\in  (0,
\lambda_{**})\cup(\lambda_*, +\infty)$ and none for $\lambda\in
[\lambda_{**}, \lambda_*]$.

(d) If $L^{*}<L\leqslant L_{**}$, then there exists $\lambda_*>0$
such that \eqref{eq:1dmemsfr0} has exactly one positive solutions
for $\lambda\in  (\lambda_*, +\infty)$ and none for $\lambda\in  (0,
\lambda_*]$.

(e) If $L= L^*$, then there exist two numbers
$\lambda^*>\lambda_*>0$ such that \eqref{eq:1dmemsfr0} has exactly
one positive solutions for $\lambda\in  (\lambda_*,
\lambda^{*})\cup(\lambda^*, +\infty)$ and none for $\lambda\in  (0,
\lambda_*]\cup\{\lambda^*\}$.

(f) If $L_*<L< L^*$, then there exist three numbers
$\lambda^*>\lambda^{**}>\lambda_*>0$ such that \eqref{eq:1dmemsfr0}
has exactly one positive solutions for $\lambda\in  (\lambda_*,
\lambda^{**})\cup(\lambda^*, +\infty)$ and none for $\lambda\in  (0,
\lambda_*]\cup[\lambda^{**}, \lambda^*]$.

(g) If $L\leqslant L_*$, then there exists $\lambda^*>0$ such that
\eqref{eq:1dmemsfr0} has exactly one positive solutions for
$\lambda\in  (\lambda^*, +\infty)$ and none for $\lambda\in  (0,
\lambda^*]$.

(h) $\lambda_*$ and $\lambda^*$ are strictly decreasing while
$\lambda_{**}$ and $\lambda^{**}$ are strictly increasing with
respect to $L$.
\end{theorem}

\begin{theorem}[Type IV-$\delta_0$, $f'(0)>0$, see Fig.\ref{fig:f00BIV5}]\label{thm:f=0t9}
Let $(\varphi, f)$ be of \emph{Type IV-$\delta_0$}. Assume
conditions \eqref{con:phi}--\eqref{con:strict2} and $f'(0)>0$ hold.
Then there exist constants $L^{**}>L_{**}>L^*>L_*>0$ (see
Fig.\ref{fig:gg1}) such that the following assertions hold:

(a) If $L> L^{**}$, then \eqref{eq:1dmemsfr0} has exactly one
positive solution for $\lambda\in (0, \lambda_1)$ and none for
$\lambda\in[\lambda_1,+\infty)$.

(b) If $L= L^{**}$, then there exists $\lambda_*\in (0,\lambda_1)$
such that \eqref{eq:1dmemsfr0} has exactly one positive solution for
$(0, \lambda_*)\cup(\lambda_*,\lambda_1)$ and none for
$\lambda\in\{\lambda_*\}\cup[\lambda_1,+\infty)$.

(c) If $L_{**}<L< L^{**}$, then there exist two numbers
$0<\lambda_{**}<\lambda_*<\lambda_1$ such that \eqref{eq:1dmemsfr0}
has exactly one positive solution for $\lambda\in  (0,
\lambda_{**})\cup(\lambda_*, \lambda_1)$ and none for
$\lambda\in[\lambda_{**},\lambda_*]\cup[\lambda_1,+\infty)$.

(d) If $L^{*}<L\leqslant L_{**}$, then there exists $\lambda_*\in
(0,\lambda_1)$ such that \eqref{eq:1dmemsfr0} has exactly one
positive solutions for $\lambda\in  (\lambda_*, \lambda_1)$ and none
for $\lambda\in  (0, \lambda_*]\cup[\lambda_1, +\infty)$.

(e) If $L= L^*$, then there exist two numbers
$0<\lambda_{*}<\lambda^*<\lambda_1$ such that \eqref{eq:1dmemsfr0}
has exactly one positive solutions for $\lambda\in  (\lambda_*,
\lambda^{*})\cup(\lambda^{*}, \lambda_1)$ and none for $\lambda\in
(0, \lambda_{*}]\cup\{\lambda^{*}\}\cup[\lambda_1, +\infty)$.

(f) If $L_*<L< L^*$, then there exist three numbers
$0<\lambda_{*}<\lambda^{**}<\lambda^*<\lambda_1$ such that
\eqref{eq:1dmemsfr0} has exactly one positive solutions for
$\lambda\in  (\lambda_*, \lambda^{**})\cup(\lambda^{*}, \lambda_1)$
and none for $\lambda\in  (0, \lambda_{*}]\cup[\lambda^{**},
\lambda^{*}]\cup[\lambda_1, +\infty)$.


(g) If $L\leqslant L_*$, then there exists $\lambda^*\in
(0,\lambda_1)$ such that \eqref{eq:1dmemsfr0} has exactly one
positive solutions for $\lambda\in  (\lambda^{*}, \lambda_1)$ and
none for $\lambda\in  (0, \lambda^{*}]\cup[\lambda_1, +\infty)$.

(h) $\lambda_*$ and $\lambda^*$ are strictly decreasing while
$\lambda^{**}$ are strictly increasing with respect to $L$.
\end{theorem}

  \ifpdf
\begin{figure}
\centering
\includegraphics[totalheight=8.3in]{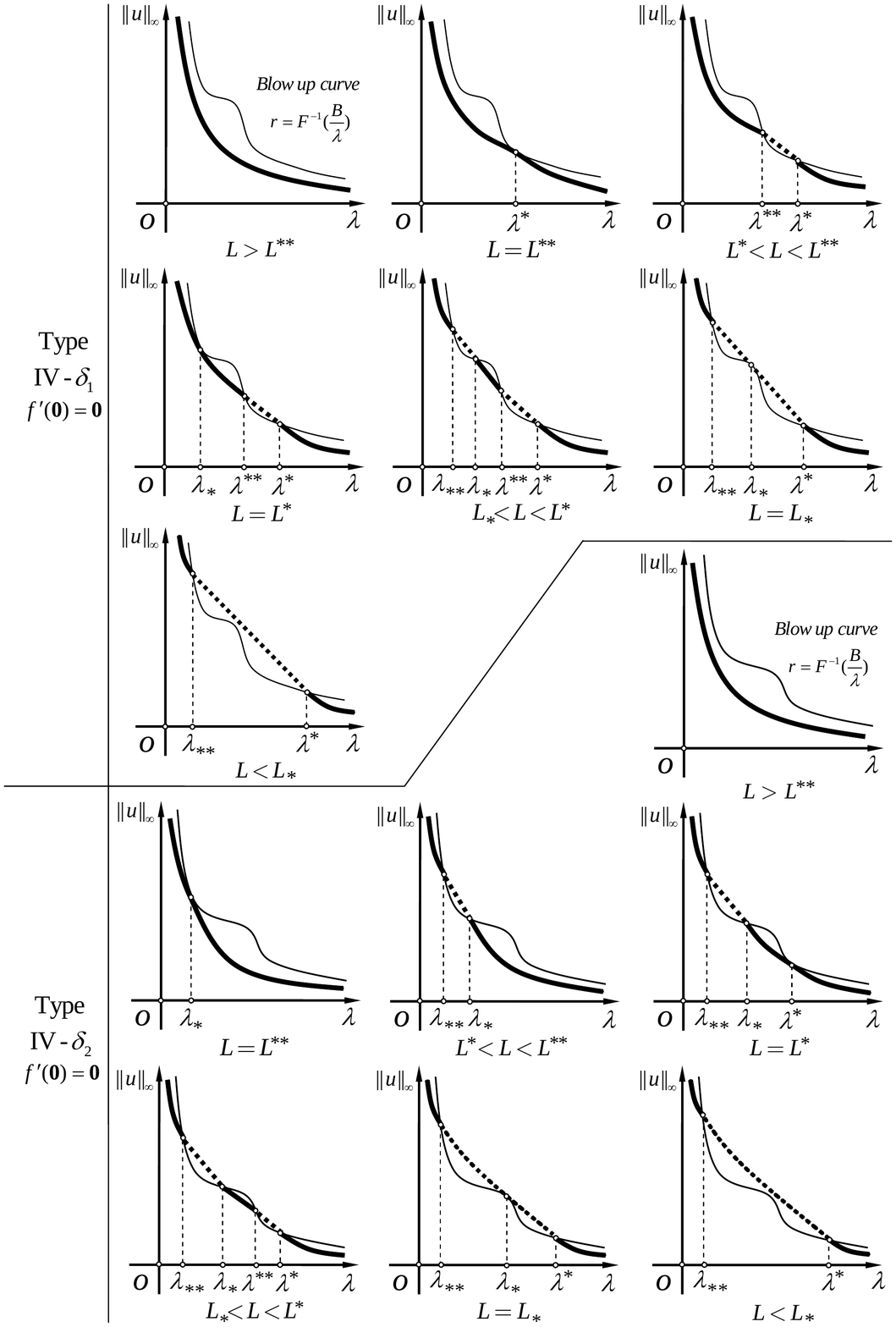}
\caption{Bifurcation Diagrams for Types IV-$\delta_1$ and
IV-$\delta_2$ with $f(0)=0$ and $f'(0)=0$.}\label{fig:f00BIV6}
\end{figure}
\begin{figure}
\centering
\includegraphics[totalheight=8.3in]{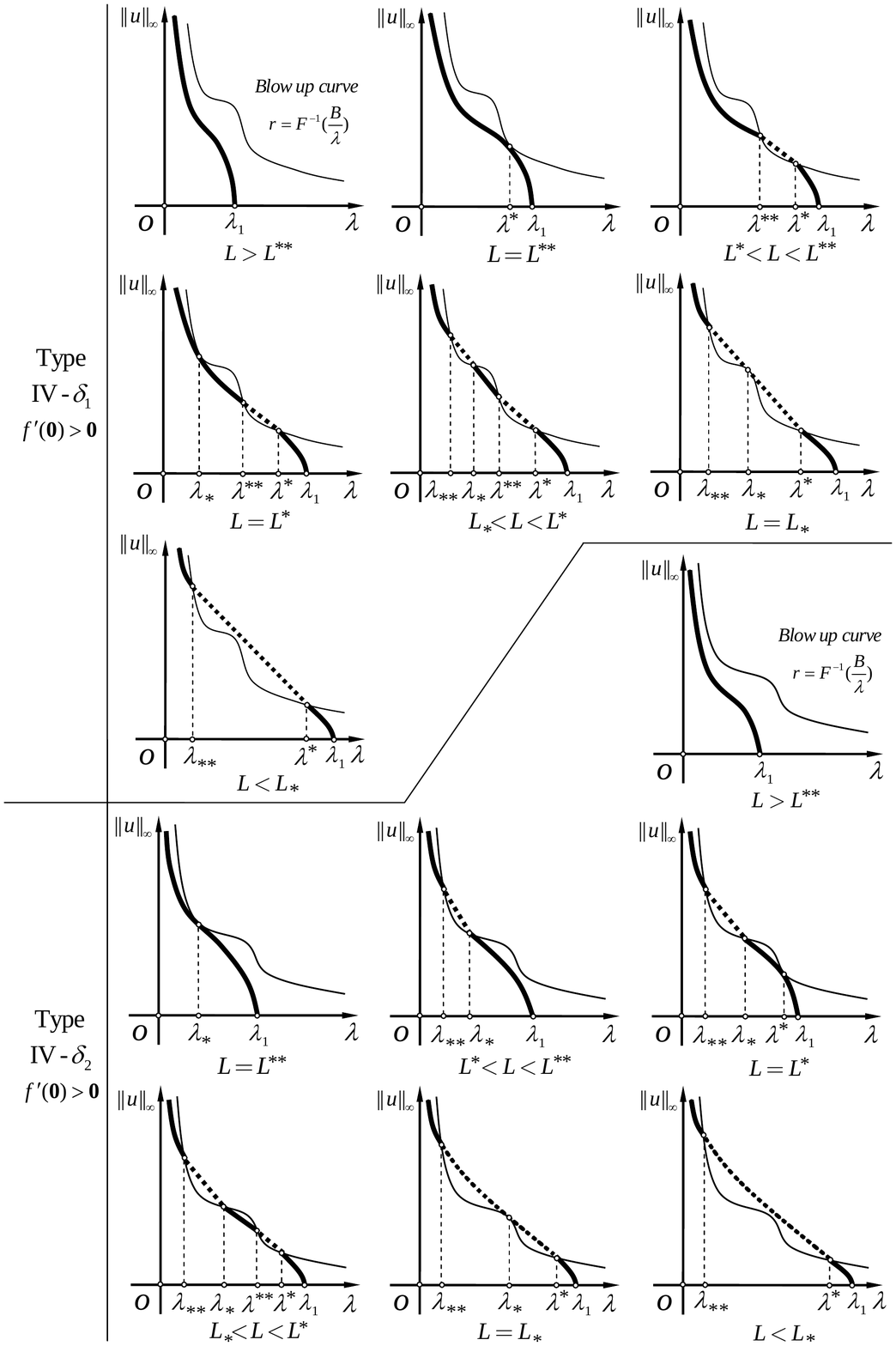}
\caption{Bifurcation Diagrams for Types IV-$\delta_1$ and
IV-$\delta_2$ with $f(0)=0$ and $f'(0)>0$.}\label{fig:f00BIV7}
\end{figure}
  \fi

\begin{example}\label{eg:4delta0}
Let $\varphi=\varphi_3$ and $f(u)=e^u+u^8-u-1$ or $f(u)=e^u+u^8-1$.
Then $\lim_{\lambda\rightarrow +\infty} g(\lambda)=0$ and
$\lim_{\lambda\rightarrow 0} g(\lambda)=\frac{\pi}{2}$ (by Lemma
\ref{prop:g} below). Numerical simulation indicates that
$g(\lambda)$ has exactly three local extreme points in $(0,
+\infty)$ and $\lim_{\lambda\rightarrow 0} g(\lambda)$ is in between
the local maximum values (see Remark \ref{rk:g(r)} and
Fig.\ref{fig:gg3} below), which suggests that $g$ is of   Type
$\delta_0$ and both $(\varphi_3,e^u+u^8-u-1)$ and
$(\varphi_3,e^u+u^8-1)$ are of Type IV-$\delta_0$.
\end{example}

\begin{theorem}[Type IV-$\delta_1$, $f'(0)=0$, see Fig.\ref{fig:f00BIV6}]\label{thm:f=0t10}
Let $(\varphi, f)$ be of \emph{Type IV-$\delta_1$}. Assume
conditions \eqref{con:phi}--\eqref{con:strict2} and $f'(0)=0$ hold.
Then there exist constants $L^{**}>L^*>L_*>0$ (see
Fig.\ref{fig:gg1}) such that the following assertions hold:

(a) If $L> L^{**}$, then \eqref{eq:1dmemsfr0} has exactly one
positive solution for any $\lambda\in (0,+\infty)$.

(b) If $L= L^{**}$, then there exists $\lambda^*>0$ such that
\eqref{eq:1dmemsfr0} has exactly one positive solution for $(0,
\lambda^{*})\cup(\lambda^{*},+\infty)$ and none for
$\lambda=\lambda^{*}$.

(c) If $L^*<L< L^{**}$, then there exist two numbers
$\lambda^*>\lambda^{**}>0$ such that \eqref{eq:1dmemsfr0} has
exactly one positive solutions for $\lambda\in  (0,
\lambda^{**})\cup(\lambda^*, +\infty)$ and none for $\lambda\in
[\lambda^{**}, \lambda^*]$.

(d) If $L= L^*$, then there exist three numbers
$\lambda^*>\lambda^{**}>\lambda_{*}>0$ such that
\eqref{eq:1dmemsfr0} has exactly one positive solutions for
$\lambda\in  (0, \lambda_{*})\cup(\lambda_*,
\lambda^{**})\cup(\lambda^*, +\infty)$ and none for $\lambda\in
\{\lambda_*\}\cup[\lambda^{**}, \lambda^*]$.

(e) If $L_*<L< L^*$, then there exist four numbers
$\lambda^*>\lambda^{**}>\lambda_*>\lambda_{**}>0$ such that
\eqref{eq:1dmemsfr0} has exactly one positive solutions for
$\lambda\in  (0, \lambda_{**})\cup(\lambda_*,
\lambda^{**})\cup(\lambda^*, +\infty)$ and none for $\lambda\in
[\lambda_{**}, \lambda_*]\cup[\lambda^{**}, \lambda^*]$.

(f) If $L\leqslant L_*$, then there exist two numbers
$\lambda^*>\lambda_{**}>0$ such that \eqref{eq:1dmemsfr0} has
exactly one positive solutions for $\lambda\in  (0,
\lambda_{**})\cup(\lambda^*, +\infty)$ and none for $\lambda\in
[\lambda_{**}, \lambda^*]$.

(g) $\lambda_*$ and $\lambda^*$ are strictly decreasing while
$\lambda_{**}$ and $\lambda^{**}$ are strictly increasing with
respect to $L$.
\end{theorem}

\begin{theorem}[Type IV-$\delta_1$, $f'(0)>0$, see Fig.\ref{fig:f00BIV7}]\label{thm:f=0t11}
Let $(\varphi, f)$ be of \emph{Type IV-$\delta_1$}. Assume
conditions \eqref{con:phi}--\eqref{con:strict2} and $f'(0)>0$ hold.
Then there exist constants $L^{**}>L^*>L_*>0$ (see
Fig.\ref{fig:gg1}) such that the following assertions hold:

(a) If $L> L^{**}$, then \eqref{eq:1dmemsfr0} has exactly one
positive solution for $\lambda\in (0, \lambda_1)$ and none for
$\lambda\in[\lambda_1,+\infty)$.

(b) If $L= L^{**}$, then there exists $\lambda^*\in (0,\lambda_1)$
such that \eqref{eq:1dmemsfr0} has exactly one positive solution for
$(0, \lambda^*)\cup(\lambda^*,\lambda_1)$ and none for
$\lambda\in\{\lambda^*\}\cup[\lambda_1,+\infty)$.

(c) If $L^*<L< L^{**}$, then there exist two numbers
$0<\lambda^{**}<\lambda^*<\lambda_1$ such that \eqref{eq:1dmemsfr0}
has exactly one positive solution for $\lambda\in  (0,
\lambda^{**})\cup(\lambda^*, \lambda_1)$ and none for
$\lambda\in[\lambda^{**},\lambda^*]\cup[\lambda_1,+\infty)$.

(d) If $L= L^*$, then there exist three numbers
$0<\lambda_{*}<\lambda^{**}<\lambda^*<\lambda_1$ such that
\eqref{eq:1dmemsfr0} has exactly one positive solutions for
$\lambda\in  (0, \lambda_{*})\cup(\lambda_*,
\lambda^{**})\cup(\lambda^{*}, \lambda_1)$ and none for $\lambda\in
\{\lambda_{*}\}\cup[\lambda_{**}, \lambda_*]\cup[\lambda_1,
+\infty)$.

(e) If $L_*<L< L^*$, then there exist four numbers
$0<\lambda_{**}<\lambda_*<\lambda^{**}<\lambda^*<\lambda_1$ such
that \eqref{eq:1dmemsfr0} has exactly one positive solutions for
$\lambda\in  (0, \lambda_{**})\cup(\lambda_*,
\lambda^{**})\cup(\lambda^*, \lambda_1)$ and none for $\lambda\in
[\lambda_{**}, \lambda_*]\cup[\lambda^{**},
\lambda^*]\cup[\lambda_1, +\infty)$.

(f) If $L\leqslant L_*$, then there exist two numbers two numbers
$0<\lambda_{**}<\lambda^*<\lambda_1$ such that \eqref{eq:1dmemsfr0}
has exactly one positive solution for $\lambda\in  (0,
\lambda_{**})\cup(\lambda^*, \lambda_1)$ and none for
$\lambda\in[\lambda_{**},\lambda^*]\cup[\lambda_1,+\infty)$.

(g) $\lambda_*$ and $\lambda^*$ are strictly decreasing while
$\lambda_{**}$ and $\lambda^{**}$ are strictly increasing with
respect to $L$.
\end{theorem}

\begin{example}\label{eg:4delta1}
Let $\varphi=\varphi_3$ and $f(u)=e^{u^2}+u^8+u-1$. Then
$\lim_{\lambda\rightarrow +\infty}
g(\lambda)=0=\lim_{\lambda\rightarrow 0} g(\lambda)$ (by Lemma
\ref{prop:g} below). Numerical simulation indicates that
$g(\lambda)$ has exactly three local extreme points in $(0,
+\infty)$ and the left local maximum value is less than the right
one (see Remark \ref{rk:g(r)} and Fig.\ref{fig:gg4} below), which
suggests that $g$ is of   Type $\delta_1$ and
$(\varphi_3,e^{u^2}+u^8+u-1)$ is of Type IV-$\delta_1$.
\end{example}

\begin{theorem}[Type IV-$\delta_2$, $f'(0)=0$, see Fig.\ref{fig:f00BIV6}]\label{thm:f=0t12}
Let $(\varphi, f)$ be of \emph{Type IV-$\delta_2$}. Assume
conditions \eqref{con:phi}--\eqref{con:strict2} and $f'(0)=0$ hold.
Then there exist constants $L^{**}>L^*>L_*>0$ (see
Fig.\ref{fig:gg1}) such that the following assertions hold:

(a) If $L> L^{**}$, then \eqref{eq:1dmemsfr0} has exactly one
positive solution for any $\lambda\in (0,+\infty)$.

(b) If $L= L^{**}$, then there exists $\lambda_*>0$ such that
\eqref{eq:1dmemsfr0} has exactly one positive solution for $(0,
\lambda_{*})\cup(\lambda_{*},+\infty)$ and none for
$\lambda=\lambda_{*}$.

(c) If $L^*<L< L^{**}$, then there exist two numbers
$\lambda_*>\lambda_{**}>0$ such that \eqref{eq:1dmemsfr0} has
exactly one positive solutions for $\lambda\in  (0,
\lambda_{**})\cup(\lambda_*, +\infty)$ and none for $\lambda\in
[\lambda_{**}, \lambda_*]$.

(d) If $L= L^*$, then there exist three numbers
$\lambda^*>\lambda_*>\lambda_{**}>0$ such that \eqref{eq:1dmemsfr0}
has exactly one positive solutions for $\lambda\in  (0,
\lambda_{**})\cup(\lambda_*, \lambda^{*})\cup(\lambda^*, +\infty)$
and none for $\lambda\in  [\lambda_{**},
\lambda_*]\cup\{\lambda^*\}$.

(e) If $L_*<L< L^*$, then there exist five numbers
$\lambda^*>\lambda^{**}>\lambda_*>\lambda_{**}>0$ such that
\eqref{eq:1dmemsfr0} has exactly one positive solutions for
$\lambda\in  (0, \lambda_{**})\cup(\lambda_*,
\lambda^{**})\cup(\lambda^*, +\infty)$ and none for $\lambda\in
[\lambda_{**}, \lambda_*]\cup[\lambda^{**}, \lambda^*]$.

(f) If $L\leqslant L_*$, then there exist two numbers
$\lambda^*>\lambda_{**}>0$ such that \eqref{eq:1dmemsfr0} has
exactly one positive solutions for $\lambda\in  (0,
\lambda_{**})\cup(\lambda^*, +\infty)$ and none for $\lambda\in
[\lambda_{**}, \lambda^*]$.

(g) $\lambda_*$ and $\lambda^*$ are strictly decreasing while
$\lambda_{**}$ and $\lambda^{**}$ are strictly increasing with
respect to $L$.
\end{theorem}

\begin{theorem}[Type IV-$\delta_2$, $f'(0)>0$, see Fig.\ref{fig:f00BIV7}]\label{thm:f=0t13}
Let $(\varphi, f)$ be of \emph{Type IV-$\delta_2$}. Assume
conditions \eqref{con:phi}--\eqref{con:strict2} and $f'(0)>0$ hold.
Then there exist constants $L^{**}>L^*>L_*>0$ (see
Fig.\ref{fig:gg1}) such that the following assertions hold:

(a) If $L> L^{**}$, then \eqref{eq:1dmemsfr0} has exactly one
positive solution for $\lambda\in (0, \lambda_1)$ and none for
$\lambda\in[\lambda_1,+\infty)$.

(b) If $L= L^{**}$, then there exists $\lambda_*\in (0,\lambda_1)$
such that \eqref{eq:1dmemsfr0} has exactly one positive solution for
$(0, \lambda_*)\cup(\lambda_*,\lambda_1)$ and none for
$\lambda\in\{\lambda_*\}\cup[\lambda_1,+\infty)$.

(c) If $L^*<L< L^{**}$, then there exist two numbers
$0<\lambda_{**}<\lambda_*<\lambda_1$ such that \eqref{eq:1dmemsfr0}
has exactly one positive solution for $\lambda\in  (0,
\lambda_{**})\cup(\lambda_*, \lambda_1)$ and none for
$\lambda\in[\lambda_{**},\lambda_*]\cup[\lambda_1,+\infty)$.

(d) If $L= L^*$, then there exist three numbers
$0<\lambda_{**}<\lambda_*<\lambda^*<\lambda_1$ such that
\eqref{eq:1dmemsfr0} has exactly one positive solutions for
$\lambda\in  (0, \lambda_{**})\cup(\lambda_*,
\lambda^{*})\cup(\lambda^{*}, \lambda_1)$ and none for $\lambda\in
[\lambda_{**}, \lambda_*]\cup\{\lambda^{*}\}\cup[\lambda_1,
+\infty)$.

(e) If $L_*<L< L^*$, then there exist four numbers
$0<\lambda_{**}<\lambda_*<\lambda^{**}<\lambda^*<\lambda_1$ such
that \eqref{eq:1dmemsfr0} has exactly one positive solutions for
$\lambda\in  (0, \lambda_{**})\cup(\lambda_*,
\lambda^{**})\cup(\lambda^*, \lambda_1)$ and none for $\lambda\in
[\lambda_{**}, \lambda_*]\cup[\lambda^{**},
\lambda^*]\cup[\lambda_1, +\infty)$.

(f) If $L\leqslant L_*$, then there exist two numbers two numbers
$0<\lambda_{**}<\lambda^*<\lambda_1$ such that \eqref{eq:1dmemsfr0}
has exactly one positive solution for $\lambda\in  (0,
\lambda_{**})\cup(\lambda^*, \lambda_1)$ and none for
$\lambda\in[\lambda_{**},\lambda^*]\cup[\lambda_1,+\infty)$.

(g) $\lambda_*$ and $\lambda^*$ are strictly decreasing while
$\lambda_{**}$ and $\lambda^{**}$ are strictly increasing with
respect to $L$.
\end{theorem}

\begin{example}\label{eg:ivdelta3}
Let $\varphi=\varphi_3$ and $f(u)=e^{u^2}+u^8-1$. Then
$\lim_{\lambda\rightarrow +\infty}
g(\lambda)=0=\lim_{\lambda\rightarrow 0} g(\lambda)$ (by Lemma
\ref{prop:g} below). Numerical simulation indicates that
$g(\lambda)$ has exactly three local extreme points in $(0,
+\infty)$ and the left local maximum value is greater than the right
one  (see Remark \ref{rk:g(r)} and Fig.\ref{fig:gg4} below), which
suggests that $g$ is of   Type $\delta_2$ and
$(\varphi_3,e^{u^2}+u^8-1)$ is of Type IV-$\delta_2$.
\end{example}

\begin{remark} We conjecture that there exist suitable $\varphi$ and $f$ such that the corresponding $g$ is of Type $\delta_3$.
Since it is difficult to find some examples, we omit it.
\end{remark}

For Cases V and VI, so far, we only know the existence of two types
of $g$: Type $\beta_0$ for Case V and Type $\gamma_0$ for Case VI
(see Fig.\ref{fig:gg1}), which have been defined in \cite{Pan2014a}.

\vskip 2mm {\bf Case V: $B<+\infty$, $A<+\infty$ and $C=+\infty$}

For Type V-$\beta_0$,
the statements of results are the same as Type IV-$\beta_0$ (see
Theorems \ref{thm:classical43320} and  \ref{thm:classical43321}), so
we omit them. The bifurcation diagrams are sketched in
Figs.\ref{fig:last0} and \ref{fig:last}, which are slightly
different from those of Type IV-$\beta_0$ (see
Figs.\ref{fig:f00BIV4} and \ref{fig:f01BIV4} ).
  \ifpdf
\begin{figure}
\centering
\includegraphics[totalheight=1.8in]{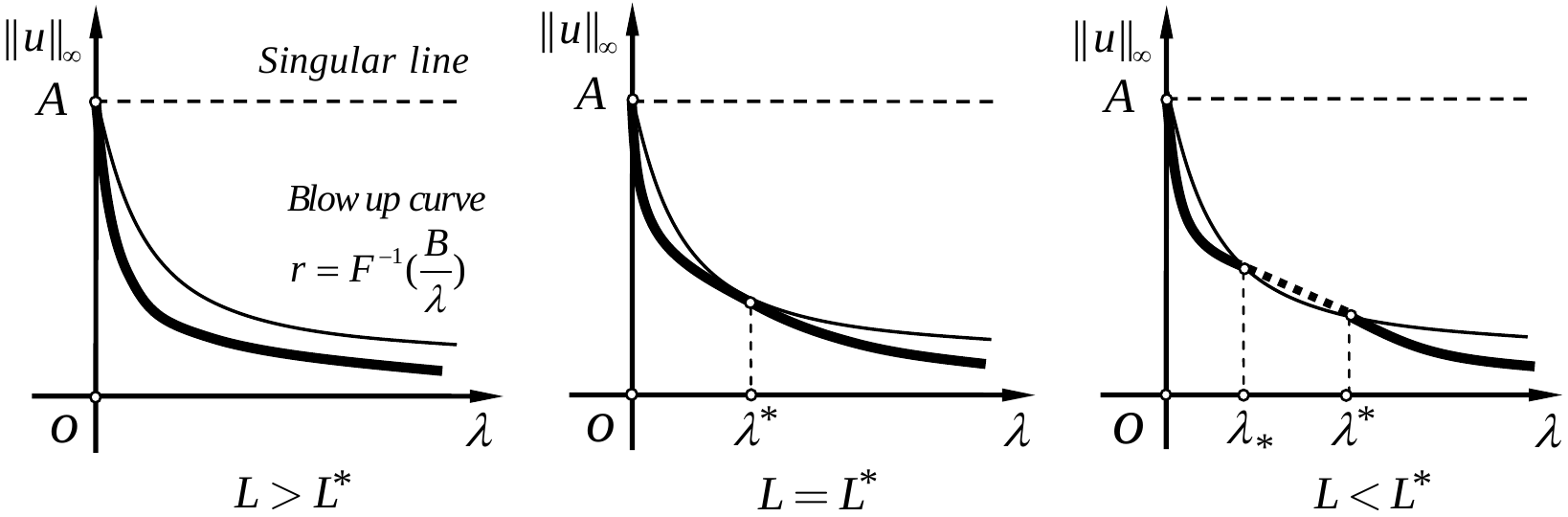}
\caption{Bifurcation Diagrams for Type V-$\beta_0$ with $f(0)=0$ and
$f'(0)=0$.}\label{fig:last0}
\end{figure}
  \fi
  \ifpdf
\begin{figure}
\centering
\includegraphics[totalheight=8.3in]{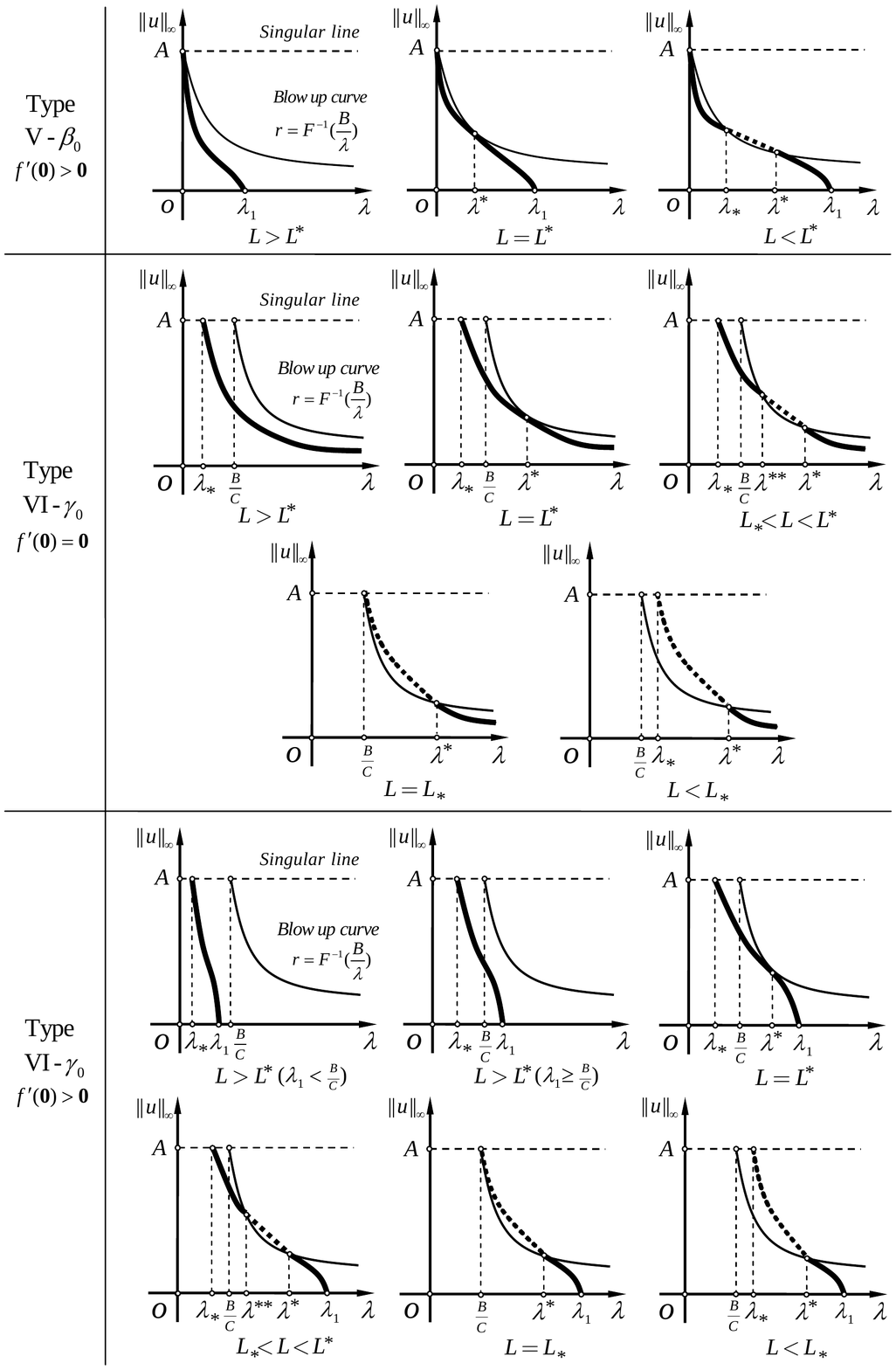}
\caption{Bifurcation Diagrams for Types V-$\beta_0$ and
VI-$\gamma_0$ with $f(0)=0$.}\label{fig:last}
\end{figure}
  \fi

\begin{example}\label{eg:vbeta0}
Let $\varphi=\varphi_3$ and $f$ is one of the following\\
\indent (1) $\tan u$, $u\in (0,\frac{\pi}{2})$;\; (2) $\tan u^2$,
$u\in (0,\sqrt{\frac{\pi}{2}})$;\;
(3) $\frac{u}{1-u}$, $u\in (0,1)$;\; (4) $\frac{u^2}{1-u^2}$, $u\in (0,1)$.\\
Then $\lim_{\lambda\rightarrow +\infty}
g(\lambda)=0=\lim_{\lambda\rightarrow 0} g(\lambda)$ (by Lemma
\ref{prop:g} below). Numerical simulation indicates that
$g(\lambda)$ has exactly one critical point in $(0, A)$ (see Remark
\ref{rk:g(r)} and Fig.\ref{fig:gg4} below), which suggests that $g$
is of Type $\beta_0$ and all of $(\varphi_3,f)$ are of Type
V-$\beta_0$.
\end{example}

\vskip 2mm {\bf Case VI: $B<+\infty$, $A<+\infty$ and $C<+\infty$}

For Type VI-$\gamma_0$,
the statements of results are the same as Type V-$\gamma_0$ (see
Theorems \ref{thm:f=0t6} and  \ref{thm:f=0t7}), so we omit them. The
bifurcation diagrams are sketched in Fig.\ref{fig:last}, which has
some important differences from those of Type IV-$\gamma_0$ (see
Figs.\ref{fig:f00BIV4} and \ref{fig:f01BIV4}).

\begin{example}\label{eg:vigamma0}
Let $\varphi=\varphi_3$ and $f(u)=\frac{1}{\sqrt{1-u}}-1$ or
$f(u)=\frac{1}{\sqrt{1-u^2}}-1$ with $u\in (0,1)$. Then
$\lim_{\lambda\rightarrow 0} g(\lambda)=+\infty$,
$\lim_{\lambda\rightarrow +\infty} g(\lambda)=0$,
$\lim_{\lambda\rightarrow 1} g(\lambda)>0$ or
$\lim_{\lambda\rightarrow \frac{2}{\pi-2}} g(\lambda)>0$ (by Lemmas
\ref{prop:g} and \ref{a-prop:twosidee23333b} below), and
$g(\lambda)$ is strictly decreasing in $(0,1)$ or
$(0,\frac{2}{\pi-2})$ (by Lemma \ref{a-prop:twosidee23333}).
Moreover, numerical simulation indicates that $g(\lambda)$ has
exactly one local maximum point in $(1, +\infty)$ or
$(\frac{2}{\pi-2}, +\infty)$ (see Remark \ref{rk:g(r)} and
Fig.\ref{fig:gg4}), which suggests that $g$ is of   Type $\delta_2$
and both $(\varphi_3,\frac{1}{\sqrt{1-u}}-1)$ and
$(\varphi_3,\frac{1}{\sqrt{1-u^2}}-1)$ are of Type VI-$\gamma_0$.
\end{example}

\begin{remark}
In Examples \ref{eg:gamma1}--\ref{eg:vigamma0}, we only give
numerical results on the number of local extreme points of $g$. It
is still lack of strictly analytic proofs. Thus we leave some open
problems. Moreover, it is also an interesting open problem to find
out a complete list of all possible types of $g$, corresponding to
$\varphi_3$ and convex, increasing $f$ in Cases IV, V and VI.
\end{remark}

\begin{remark}\label{rk:varioussolutions}
Although we here focus on classical solutions of
\eqref{eq:1dmemsfr0}, we also deviate a little from the topics of
discussion and gives some additional information on ``non-classical
solutions'' in Figs.\ref{fig:f1234},
\ref{fig:f00BIV4}--\ref{fig:last}.
In fact, the thick dashed lines represent discontinuous
non-classical solutions (because of limitations of space, about
non-classical solutions refer to \cite{xp2011a,Pan2012}, also see
\cite{BHOO07,BHOO072,BG10,Obersnel2007}). We also note that the
intersection points of the bifurcation curve and the singular line
in Figs.\ref{fig:f1234}, \ref{fig:last0} and \ref{fig:last}
represent ``singular solutions'' (see e.g. \cite{Levine1989} for the
semilinear case).
\end{remark}

\section{Time map and general properties}

In this section, we recall and investigate various properties of the
so-called time map $T$ for problem \eqref{eq:1dmemsfr0} under
suitable conditions.
These properties will be used to obtain the shape of the time map
and prove the main theorems.

\subsection{Some well known results about time map}

Since $f(u)$ does not contain $x$ explicitly and the problem is
autonomous, positive solutions of \eqref{eq:1dmemsfr0} are always
symmetric (see e.g. \cite[Lem 2.1]{KL10}).
Hence $u'(0)=0$.

Due to the symmetry, we consider the equation
\begin{equation*} \label{eq:achange}
        -\varphi'(u')u''={\lambda}f(u),  \qquad x\in (0, L),
\end{equation*}
with the initial conditions
\begin{equation*}\label{con:ainitial}
u(0)= r>0,\; u'(0)=0.
\end{equation*}
From the energy conservation relation
\begin{equation*}\label{eq:aconversionlaw}
\Phi(u')  + \lambda F(u) =\lambda F( r),
\end{equation*}
%
we obtain the following time map for positive solutions
\begin{align}\label{eq:time}
T(r, \lambda)=\int_{0}^ r \frac{1}{\Phi^{-1}(\lambda\left[F(
r)-F(u)\right]) }du,\qquad r\in I,\; \lambda>0,
\end{align}
where $\Phi^{-1}$ is always taken to be positive and
\begin{equation}\label{eq:alpha}
I:= \left\{
 \begin{array}{ll}
  (0, A), & \hbox{if } B=+\infty; \\
  (0,  F^{-1}(\frac{B}{\lambda})], & \hbox{if } B<+\infty \hbox{ and } C=+\infty;\\
 (0, A), & \hbox{if } B<+\infty,  C<+\infty, \hbox{ and } \lambda\leqslant\frac{B}{C}; \\
 (0,  F^{-1}(\frac{B}{\lambda})], & \hbox{if } B<+\infty,  C<+\infty, \hbox{ and }  \lambda> \frac{B}{C}.
 \end{array}
\right.
\end{equation}
Denote by $r^*$ the right endpoint of $I$. Under conditions
\eqref{con:phi} and \eqref{con:positive}, $T(r, \lambda)$ is well
defined and continuous for all $( r, \lambda)\in I\times (0,
+\infty)$.
 When $B<+\infty$, the curve $ r=F^{-1}(\frac{B}{\lambda})$, due to $|u'(\pm L;  r)|=+\infty$,
 is known as ``(gradient or derivative) blow-up curve'';
when $A<+\infty$, the straight line $r=A$ is referred as ``singular
line''. See \cite{Pan2014a} for details.


%
%
%

Since the solutions of \eqref{eq:1dmemsfr0} correspond to the
bifurcation curve which is determined by $T(r, \lambda)=L$, this
leads us to investigate the graph of $T(r, \lambda)$.



About the gradient blow-up curve, the following result is well known
(define $\frac{B}{C}=0$ if $C=+\infty$).
\begin{lemma}[\cite{Pan2014a}]\label{thm:blowupcurve}
  Assume condition \eqref{con:positive} holds. Let $B<+\infty$ and $r(\lambda)=F^{-1}(\frac{B}{\lambda})$. Then
the following assertions hold:

(a) The function $r(\lambda)$ is well defined on $(\frac{B}{C},
+\infty)$ and strictly decreasing.

(b) $\lim_{\lambda\rightarrow \frac{B}{C}} r(\lambda)=A$ and
$\lim_{\lambda\rightarrow \infty} r(\lambda)=0$.
%
\end{lemma}


\begin{lemma}[\cite{Pan2014a}]\label{a-prop:twosidee1}
Assume conditions \eqref{con:phi} and \eqref{con:positive} hold.
Let $T(r, \lambda)$ be defined in \eqref{eq:time}. Then for fixed $
r\in I$, $T(r, \lambda)$ is strictly decreasing in $\lambda$ and
$\lim_{\lambda\rightarrow0}T(r, \lambda)=+\infty$.
\end{lemma}

\begin{lemma}[\cite{Pan2014a}]\label{a-prop:twosidee12}
Assume conditions \eqref{con:phi} and \eqref{con:positive} hold. Let
$T(r, \lambda)$ be defined in \eqref{eq:time}. If $B=+\infty$, then
$\lim_{\lambda\rightarrow+\infty}T(r, \lambda)=0$ for any $r\in
(0,A)$.
\end{lemma}

The following lemma, which is a generalization of Habets and Omari
\cite[Lem 3.1]{HO07}, gives smoothness of $T$ with respect to $r$.
\begin{lemma}[\cite{Pan2014a}]\label{lem:first}
 If for any $ r\in I$ there exists a locally bounded function $K(r)>0$ such that for
every $s\in (\frac{1}{2}, 1)$,
\begin{equation}\label{con:lip}
\Big|\frac{f( r)-sf( r s)}{F( r)-F( r s)}\Big|\leqslant K(r),
\end{equation}
then $T$ is differentiable at each point $ r\in I$, with derivative
\begin{align}\nonumber
\pd{T}{ r}=\int_{0}^1&\frac{1}{\Phi^{-1}(\lambda\left[F( r)-F( r
s)\right]) }ds\\\label{eq:1.5} &-\lambda r \int_0^1
\frac{f(r)-sf(rs)}{\left[\Phi^{-1}(\lambda\left[F( r)-F( r
s)\right])\right]^2 \Phi'(\Phi^{-1}(\lambda\left[F(
r)-F(rs)\right])) } ds
\end{align}
and $\pd{T}{ r}(r,\lambda)$ is continuous at each point $
(r,\lambda)\in I\times(0,\infty)$.
\end{lemma}
Notice that
the interval $(\frac{1}{2},1)$ can be replaced with any $(c,1), c\in
(0,1)$.

\subsection{More properties on time maps}

For simplicity, in what follows, we usually denote $T(r, \lambda)$
by $T( r)$, $\pd{T}{ r} $ by $T'$,  and $\pd{T}{\lambda} $ by
$T_\lambda$.
\begin{theorem}\label{T'<0}
Assume that \eqref{con:phi}--\eqref{con:strict2} hold. Then
$T'(r)<0$ for all $r\in I$.
\end{theorem}

\begin{proof}
Since  $f\in C^1$ 
implies \eqref{con:lip}, $T$ is differentiable.
By \eqref{eq:1.5}, we have
\begin{align}\label{eq:1.7}
T'(r) &=\int_0^1 \frac{\left[\Phi^{-1}(\lambda\left[F( r)-F( r
s)\right])\right]^2 \varphi'(\Phi^{-1}(\lambda\left[F( r)-F( r
s)\right]))  -  \lambda   r      \left[f( r)-sf( r
s)\right]}{\left[\Phi^{-1}(\lambda\left[F( r)-F( r
s)\right])\right]^3 \varphi'(\Phi^{-1}(\lambda\left[F( r)-F( r
s)\right])) } ds.
\end{align}
Denote the numerator of the integrand by
\begin{align*}
H(s)=\left[\Phi^{-1}(\lambda\left[F( r)-F( r s)\right])\right]^2
\varphi'(\Phi^{-1}(\lambda\left[F( r)-F( r s)\right]))  -  \lambda r
\left[f( r)-sf( r s)\right].
\end{align*}
Clearly, $H(1)=0$. Moreover, for $s\in (0, 1)$, we have
\begin{align*}
H'(s)=&-2\lambda f(rs)r- \lambda f(rs)r
\frac{\Phi^{-1}(\lambda\left[F( r)-F( r s)\right])
\varphi''(\Phi^{-1}(\lambda\left[F( r)-F( r s)\right]))   }{
\varphi'(\Phi^{-1}(\lambda\left[F( r)-F( r s)\right]))   }
+\lambda f(rs)r + \lambda f'(rs)r^2s\\
= &- \lambda f(rs)r \frac{\Phi^{-1}(\lambda\left[F( r)-F( r
s)\right]) \varphi''(\Phi^{-1}(\lambda\left[F( r)-F( r s)\right]))
}{ \varphi'(\Phi^{-1}(\lambda\left[F( r)-F( r s)\right]))   }  +
\lambda r [f'(rs)rs - f(rs)]> 0,
\end{align*}
where the last inequality holds due to
\eqref{con:phi}--\eqref{con:strict2}. Hence, $H(s)< 0$ on $[0,1)$.
Then $T'(r)<0$ on $I$.
\end{proof}

%

The following two propositions give some important information of
$T(r)$ at the left endpoint of $I$.


\begin{proposition}\label{T-left}
Let $\lambda>0$ be fixed. Assume conditions \eqref{con:phi} and
\eqref{con:positive} hold. If $f(0)=0$ and
$0<\lim\limits_{r\rightarrow 0} \frac{f(r)}{r^\alpha}=E <+\infty$
for some $\alpha>0$, then
    \begin{align*}
    \lim\limits_{r\rightarrow 0} T(r) =
    \left\{
    \begin{array}{ll}
    0, & \qquad\text{if}~~0<\alpha<1;\\
    \frac{\pi}{ 2}  \sqrt{  \frac{ \varphi'(0)}{  {\lambda {E}} }}, & \qquad\text{if}~~\alpha=1;\\
   +\infty, & \qquad\text{if}~~\alpha>1.
    \end{array}
    \right.
    \end{align*}
\end{proposition}


\begin{proof}
Letting $u =  r s$ in \eqref{eq:time}, we obtain
\begin{align}\label{r-2}
T( r)&= r \int_{0}^1
\frac{1}{\Phi^{-1}(\lambda\left[F(r)-F(rs)\right]) }ds
=r^{\frac{1-\alpha}{2}}\int_{0}^1 \frac{ r^{\frac{1+\alpha}{2}} }{
\Phi^{-1}(\lambda\left[F( r)-F( r s)\right]) }\text{d}s .
\end{align}

Since $f(0)=0$ and $\lim\limits_{r\rightarrow 0}
\frac{f(r)}{r^\alpha}=E$, we obtain
\begin{align}\label{f-r}
   \frac{ F(r)-F(rs)}{r^{\alpha+1}}=\frac{1}{r^{\alpha+1}} \int_{rs}^{r}f(z)\, \text{d}z=
   \int_{s}^{1}\frac{f(r\,\tau)}{{(r\tau})^{\alpha}}\tau^{\alpha}\, \text{d}\tau
   \rightarrow \frac{E(1-s^{\alpha+1})}{\alpha+1} \qquad \text{as} ~~r\rightarrow 0.
\end{align}
Moreover, we know
\begin{align}\label{f_rr}
    &\lim_{y\rightarrow 0} \frac{y}{ (\Phi^{-1}(y))^2} =  \lim_{z\rightarrow 0} \frac{\Phi(z)}{z^2 } =
 \lim_{z\rightarrow 0} \frac{z \varphi '(z)}{2z }=\frac{\varphi'(0)}{2}>0.
\end{align}
By the proof of Lemma 2.1 of \cite{Pan2014a}, we have
\begin{align} \nonumber
\lim\limits_{r\rightarrow 0} \int_{0}^1
\frac{r^{\frac{1+\alpha}{2}}}{\Phi^{-1}(\lambda\left[F( r)-F(
rs)\right]) }\text{d}s &=\int_{0}^1 \lim_{r\rightarrow 0}\frac{
r^{\frac{1+\alpha}{2}}}{\Phi^{-1}(\lambda\left[F( r)-F( rs)\right])
}\text{d}s.
\end{align}
This, together with \eqref{f-r} and \eqref{f_rr}, implies that
\begin{align} \nonumber
\lim\limits_{r\rightarrow 0} \int_{0}^1 \frac{
r^{\frac{1+\alpha}{2}}}{\Phi^{-1}(\lambda\left[F( r)-F( rs)\right])
}\text{d}s &= \int_{0}^1\lim_{r\rightarrow 0} \frac{\sqrt{\lambda
[F(r)-  F(rs)]} }{\Phi^{-1}(\lambda\left[F( r)-F( r s)\right]) }
 \frac{r^{\frac{1+\alpha}{2}}}{ \sqrt{\lambda [F(r)-F(rs)]}}
\text{d}s\\\label{gg}
&=\int_{0}^{1} \sqrt{\frac{{\varphi'(0)(\alpha+1)}}{2\lambda
{E(1-s^{\alpha+1})}}}\text{d}s.
\end{align}

If $\alpha\in(0,1)$, we have
\begin{align*}
\int_{0}^{1} \sqrt{\frac{{\varphi'(0)(\alpha+1)}}{2\lambda
{E(1-s^{\alpha+1})}}}\text{d}s \leqslant \int_{0}^{1}
\sqrt{\frac{{\varphi'(0)(\alpha+1)}}{2\lambda {E(1-s)}}}\text{d}s=
 \sqrt{\frac{2\varphi'(0)(\alpha+1)}{\lambda {E}}}.
\end{align*}
Then  \eqref{r-2} and \eqref{gg} imply $\lim\limits_{r\rightarrow 0}
T(r) = 0$.

If $\alpha=1$, we get
\begin{align*}
\int_{0}^{1} \sqrt{\frac{{\varphi'(0)(\alpha+1)}}{2\lambda
{E(1-s^{\alpha+1})}}}\text{d}s = \int_{0}^{1}
\sqrt{\frac{\varphi'(0)}{{\lambda {E(1-s^2)}}}}\text{d}s=
\frac{\pi}{ 2}  \sqrt{  \frac{ \varphi'(0)}{  {\lambda {E}} }}.
\end{align*}
Then  \eqref{r-2} and \eqref{gg} imply that
$\lim\limits_{r\rightarrow 0} T(r) = \frac{\pi}{ 2}  \sqrt{  \frac{
\varphi'(0)}{  {\lambda {E}} }}$.

If $\alpha>1$, we obtain
\begin{align*}
\int_{0}^{1} \sqrt{\frac{{\varphi'(0)(\alpha+1)}}{2\lambda
{E(1-s^{\alpha+1})}}}\text{d}s \geqslant \int_{0}^{1}
\sqrt{\frac{{\varphi'(0)(\alpha+1) }}{2\lambda {E}}}\text{d}s=
\sqrt{\frac{{\varphi'(0)(\alpha+1)}}{2\lambda {E}}}.
\end{align*}
Thus  $\lim\limits_{r\rightarrow 0} T(r) = +\infty$ follows from
 \eqref{r-2} and \eqref{gg}.
\end{proof}

\begin{proposition}\label{cor:twokind}
Let $\lambda>0$ be fixed. Assume conditions \eqref{con:phi},
\eqref{con:positive} and \eqref{f-c1} hold. Then
    \begin{align*}
    \lim_{r\rightarrow 0} T(r) =
    \left\{
    \begin{array}{ll}
       \frac{\pi}{ 2}  \sqrt{  \frac{ \varphi'(0)}{  {\lambda f'(0)} }}, & \qquad\text{if } f'(0)>0;\\
   +\infty, & \qquad\text{if } f'(0)=0.
    \end{array}
    \right.
    \end{align*}
\end{proposition}

\begin{proof} Notice that conditions \eqref{con:positive} and \eqref{f-c1} imply $f(0)=0$.

If $f'(0)>0$,  then $\lim_{r\rightarrow 0} \frac{f(r)}{r}=f'(0)>0$.
By Proposition \ref{T-left}, we have
$$\lim_{r\rightarrow 0} T(r) =\frac{\pi}{ 2}  \sqrt{  \frac{ \varphi'(0)}{  {\lambda f'(0)} }}.$$

We next consider the case $f'(0)=0$. Since
  $$\lim_{r\rightarrow 0}\frac{r^{2}}{F(r)}
  =\lim_{r\rightarrow 0}\frac{2r}{f(r)}=\lim_{r\rightarrow 0}\frac{2}{f'(r)}=+\infty,$$
it follows that  for any $M>0$, there exists $\delta_M>0$ such that
\begin{equation*}
\frac{r^{2}}{F(r)}>M,\qquad \hbox{for }\; 0<r<\delta_M.
\end{equation*}
Then for all $s\in[0,1]$, we have
\begin{equation}\label{eq:middle0}
\frac{r^{2} }{F(r)-F(rs)}\geqslant\frac{r^{2}}{F(r)}>M,\qquad
\hbox{for }\; 0<r<\delta_M.
\end{equation}
This, together with \eqref{r-2}, implies
\begin{align}\nonumber
T( r)&=\int_{0}^1 \frac{\sqrt{\lambda [F(r)-  F(rs)]}
}{\Phi^{-1}(\lambda\left[F( r)-F( r s)\right]) }
 \frac{r}{ \sqrt{\lambda [F(r)-F(rs)]}}ds\\\label{eq:middle1}
&\geqslant\sqrt{\frac{M}{\lambda}} \int_{0}^1\frac{\sqrt{\lambda
[F(r)-  F(rs)]} }{\Phi^{-1}(\lambda\left[F( r)-F( r s)\right]) }ds.
\end{align}
From \eqref{f_rr}, it follows that
$$\lim_{r\rightarrow 0}\int_{0}^1 \frac{\sqrt{\lambda [F(r)-  F(rs)]} }{\Phi^{-1}(\lambda\left[F( r)-F( r s)\right]) }
 ds=\sqrt{\frac{\varphi'(0)}{2}}.$$
This, together with \eqref{eq:middle0} and \eqref{eq:middle1},
implies $\lim_{r\rightarrow 0} T(r)=+\infty$.
\end{proof}

\subsection{The function $g$}\label{sec:3.3}

  \ifpdf
\begin{figure}
\centering
\includegraphics[totalheight=2.0in]{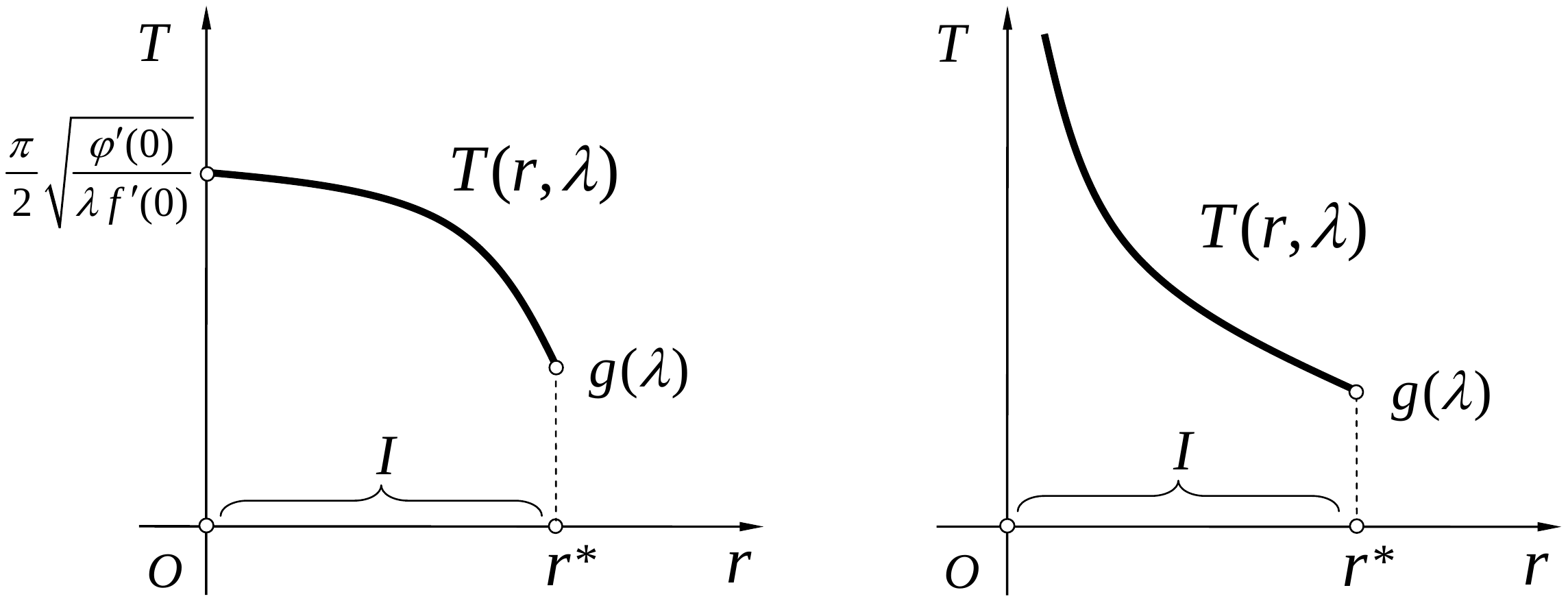}
\caption{The graphs of time map $T$ for $f(0)=0$ when $\lambda$ is
fixed. Left: $f'(0)>0$. Right: $f'(0)=0$.}\label{fig:twotimemaps}
\end{figure}
  \fi
By Theorem \ref{T'<0} and Proposition \ref{cor:twokind}, we know
that the graph of $T(r,\lambda)$ looks like the left or the right in
Fig.\ref{fig:twotimemaps}. From which, we find that the position of
$T$ at the right endpoint $r^*$ of interval $I$ is crucial in the
existence of solutions for $T(r,\lambda)=L$. We next discuss how the
position of $T(r^*,\lambda)$ changes as the parameter $\lambda$
varies through positive values.
To analyze the behaviors of $T$ at $r^*$, the same as in
\cite{Pan2014a}, we define
$$g(\lambda)=\lim_{ r\rightarrow  r^*}T(r, \lambda).$$
From \eqref{eq:alpha}, it follows that
 \begin{equation}\label{eq:rstar2}
g(\lambda)= \left\{
 \begin{array}{ll}
   \lim_{ r\rightarrow A^-}T(r, \lambda), & \hbox{if } B=+\infty; \\
  T( F^{-1}(\frac{B}{\lambda}), \lambda), & \hbox{if } B<+\infty \hbox{ and } C=+\infty;\\
  \lim_{ r\rightarrow A^-}T(r, \lambda), & \hbox{if } B<+\infty,  C<+\infty, \hbox{ and } \lambda\leqslant\frac{B}{C}; \\
 T( F^{-1}(\frac{B}{\lambda}), \lambda), & \hbox{if } B<+\infty,  C<+\infty, \hbox{ and }  \lambda> \frac{B}{C}.
 \end{array}
\right.
\end{equation}

The graph of $g$ has been well investigated in \cite[Sec
5]{Pan2014a}. Let us recall some useful results.

\vskip 2mm {\bf Case I:  $\mathbf{A,B,C=+\infty}$}

In this case,  from Theorem 4.2 of \cite{Pan2014a}, we know that if
condition \eqref{con:phif02} holds, then
\begin{align*} 
g(\lambda)=
  \lim_{ r\rightarrow +\infty}T(r, \lambda)\equiv 0.
\end{align*}
In particular,  similar to Corollary 4.4 of \cite{Pan2014a}, we also
have that if the range of $\varphi$ is bounded or
$\frac{F(z)}{f(z)}$ is bounded for sufficiently large $z$, then
$g\equiv 0$.

\vskip 2mm {\bf Case II:  $\mathbf{A<+\infty}$, $\mathbf{B=+\infty}$
and $\mathbf{C=+\infty}$}

In this case, from Corollary 4.3 of \cite{Pan2014a}, we obtain that
$g(\lambda)=\lim_{ r\rightarrow A^-}T(r, \lambda)\equiv 0.$

\vskip 2mm {\bf Case III:  $\mathbf{A<+\infty}$,
$\mathbf{B=+\infty}$ and $\mathbf{C<+\infty}$}

In this case, we have
\begin{equation*} 
g(\lambda)=  \lim_{ r\rightarrow A^-}T(r, \lambda)=\int_{0}^{{C}}
\frac{1}{\Phi^{-1}(\lambda y) } \frac{1}{f{\circ}
F^{-1}(C-y)}\text{d}y.
\end{equation*}

By Proposition 5.1 of \cite{Pan2014a}, it is easy to see
\begin{lemma}\label{a-prop:twosidee23}
Let $A<+\infty$, $B=+\infty$ and $C<+\infty$. Assume conditions
\eqref{con:phi}--\eqref{f-c1} hold.
Let $g(\lambda)$ be defined in \eqref{eq:rstar2}. Then the following assertions hold:\\
\indent (a) $g(\lambda)>0$ for all $\lambda>0$ and $g$ is strictly decreasing in $\lambda$.\\
\indent (b) $\lim_{\lambda\rightarrow0}g(\lambda)=+\infty$ and $\lim_{\lambda\rightarrow+\infty }g(\lambda)=0$.\\
\indent (c) For every $L\in (0, +\infty)$ there exists a unique
$\lambda_*$ such that $g(\lambda_*)=L$. Moreover, $\lambda_*$ is
strictly decreasing with respect to $L$.
\end{lemma}

\vskip 2mm
{\bf Cases IV and V:  $\mathbf{B<+\infty}$ and $\mathbf{C=+\infty}$}

In these cases, we have
\begin{align}\label{form:gfucntion1}
g(\lambda) = T\left(F^{-1}(\frac{B}{\lambda})\right)
&=\int_{0}^{F^{-1}(\frac{B}{\lambda})}\frac{1}{\Phi^{-1}( B -\lambda
F(u)) }du\\\label{form:gfucntion2} &=\int_{0}^{B}
\frac{1}{\Phi^{-1}(B- y) }\frac{1}{ \lambda
f(F^{-1}(\frac{y}{\lambda}))}\text{d}y.
\end{align}
Due to the separation of $\Phi^{-1}$ and $F$, it is usually more
convenient to use the expression \eqref{form:gfucntion2} to analyze
the shape of $g$.

The next lemma is very useful for computing the limits of $g$ at
$+\infty$ and $0$.
\begin{lemma}[\cite{Pan2014a}]\label{prop:g}
Let $B<+\infty$ and $C=+\infty$. Assume conditions \eqref{con:phi},
\eqref{con:positive} and $K:=\int_{0}^{B} \frac{1}{y \Phi^{-1}(B- y)
} dy<+\infty$ hold.
Then the following assertions hold:

(a) If $   \lim\limits_{t\rightarrow 0} \frac{F(t)}{f(t)} =0,$ then
$\lim\limits_{\lambda\rightarrow +\infty} g(\lambda)=0$.

(b) If $
   \lim\limits_{t\rightarrow A} \frac{F(t)}{f(t)}=D\in[0, +\infty],
$ then $\lim\limits_{\lambda\rightarrow 0}g(\lambda)=DK$.
\end{lemma}
About some concrete functions $\varphi$ satisfying $K<+\infty$, see
Examples \ref{table:fi} and \ref{eg:integral}.

The next result gives a sufficient condition for the monotonicity of
$g$.

\begin{lemma}[\cite{Pan2014a}]\label{prop:monotoneg}
Let $B<+\infty$ and $C=+\infty$. Assume conditions \eqref{con:phi}
and \eqref{con:positive} hold. If $f$ is of class $C^1$ on $(0, A)$
satisfying
\begin{align}\label{ineq:3f}
f'(t)F(t) \leqslant (\lesssim) f^2(t)  \qquad \text{for}\quad t\in
(0,A),
\end{align}
then $g(\lambda)$ is (strictly) decreasing in $\lambda$.
\end{lemma}
See Example \ref{eg:Fconcave} for some concrete functions $f$
satisfying \eqref{ineq:3f}.


When $g$ is not monotone, the situation for $g$ in Case IV may be
quite complicated. In Section 2, we have given a brief introduction
to $g$. In particular, we establish a classification of $g$ based on
the number of local extreme points in $(0,+\infty)$, the local
extreme values, and the limits at $0$ and $+\infty$ (see Definition
\ref{df:g1}). Various types of $g$ are illustrated in
Fig.\ref{fig:gg1}.

Letting $\lambda=\frac{B}{F(r)}$ in \eqref{form:gfucntion1}, we
transform $g(\lambda)$ to
\begin{align}\label{tr:g(r)}
\tilde{g}(r)=g(\lambda)|_{\lambda=\frac{B}{F(r)}}=\int_{0}^r\frac{1}{\Phi^{-1}(
B -B \frac{F(u)}{F(r)}) }du.
 \end{align}
Notice that when we plot the graph of $T$ under the coordinate
system $(r, T)$, we actually need the shape of $\tilde{g}(r)$ (not
$g(\lambda)$!), because the curve of $\tilde{g}(r)$ (i.e.
$g(\frac{B}{F(r)})$) is just the path along which the right endpoint
of $T$ moves as $\lambda$ varies through positive values.
See Figs.\ref{fig:gg1} and \ref{fig:gg2} for a comparison.


  \ifpdf %
\begin{figure}
\centering
\includegraphics[totalheight=8.3in]{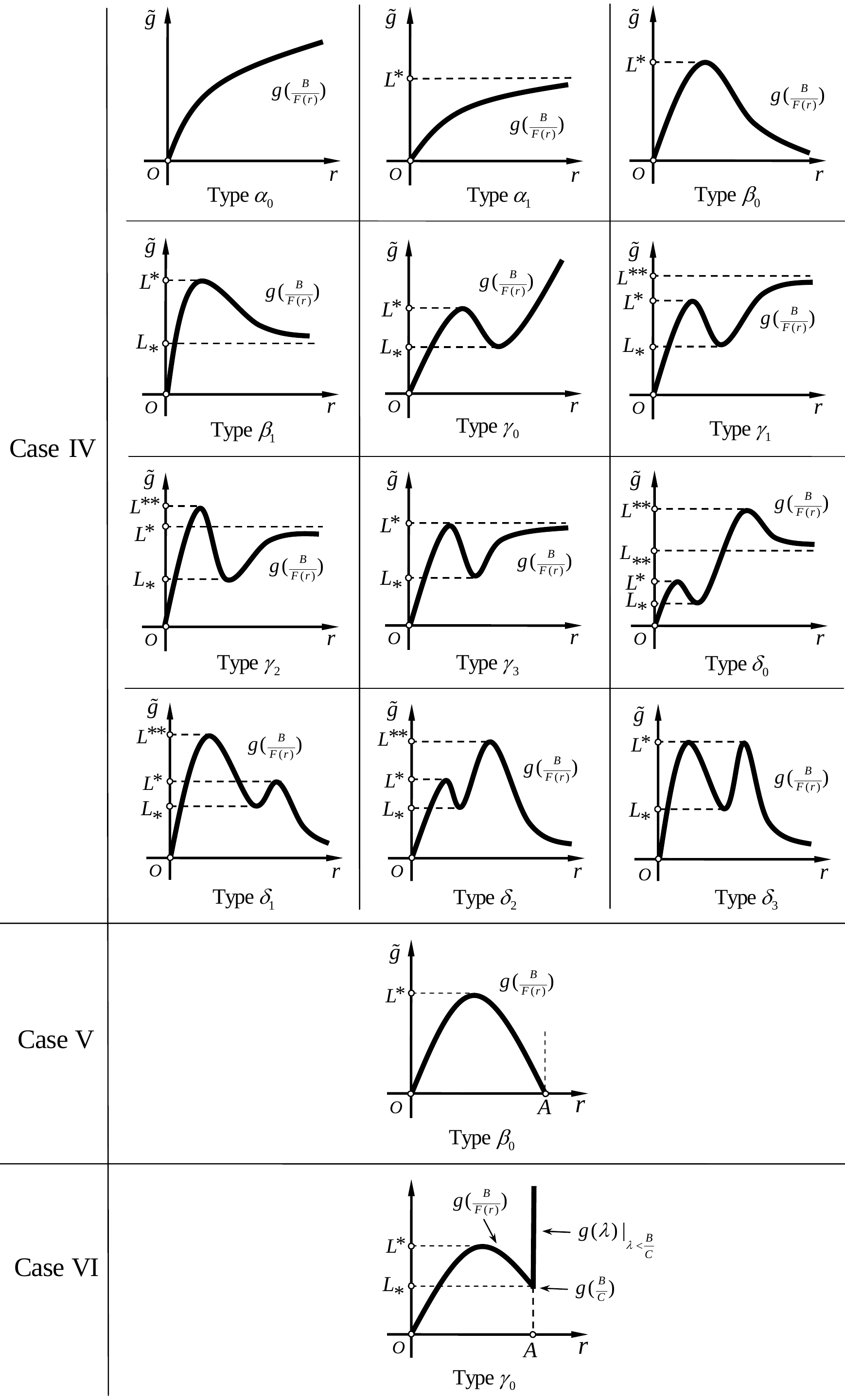}
\caption{Some shapes of $g(\lambda)$ in the coordinate system
$(r,\tilde{g})$, i.e. the shapes of $g(\frac{B}{F(r)})$. See
Fig.\ref{fig:gg1} for a comparison.}\label{fig:gg2}
\end{figure}
   \fi %

 \begin{remark}\label{rk:g(r)}
Consider an important example $\varphi=\varphi_3$, i.e., the mean
curvature equation \eqref{eq:mcemems}. Then $B=1$ and we obtain from
\eqref{form:gfucntion2} and \eqref{tr:g(r)}
\begin{align}\label{form:g}
g(\lambda)=\int_{0}^1 \frac{y}{\sqrt{1-y^2}}\frac{1}{\lambda f(
F^{-1}(\frac{y}{\lambda}))}\text{d}y
\end{align}
and
\begin{align}\label{fo:g(r)}
\tilde{g}(r):=g(\lambda)|_{\lambda=\frac{1}{F(r)}}=\int_{0}^r
\frac{F(u)}{\sqrt{F(r)^2-F(u)^2}}du.
 \end{align}
In Figs.\ref{fig:gg3} and \ref{fig:gg4}, we give some numerical
simulations of $\tilde{g}(r)$ in Cases IV and V.
For these $\tilde{g}(r)$, the corresponding functions $g(\lambda)$ arise in Examples \ref{eg:alpha}--\ref{eg:vbeta0}. 
From Figs.\ref{fig:gg3} and \ref{fig:gg4}, using the monotonicity of
$\lambda=\frac{1}{F(r)}$, one can indirectly obtain the numerical
results of $g(\lambda)$. As shown in Figs.\ref{fig:f1234},
\ref{fig:f00BIV4}--\ref{fig:f00BIV7}, the various types of $g$ lead
to the rich diversity of bifurcation diagrams for problem
\eqref{eq:1dmemsfr0}.
 \end{remark}

  \ifpdf %
\begin{figure}
\centering
\includegraphics[totalheight=8.2in]{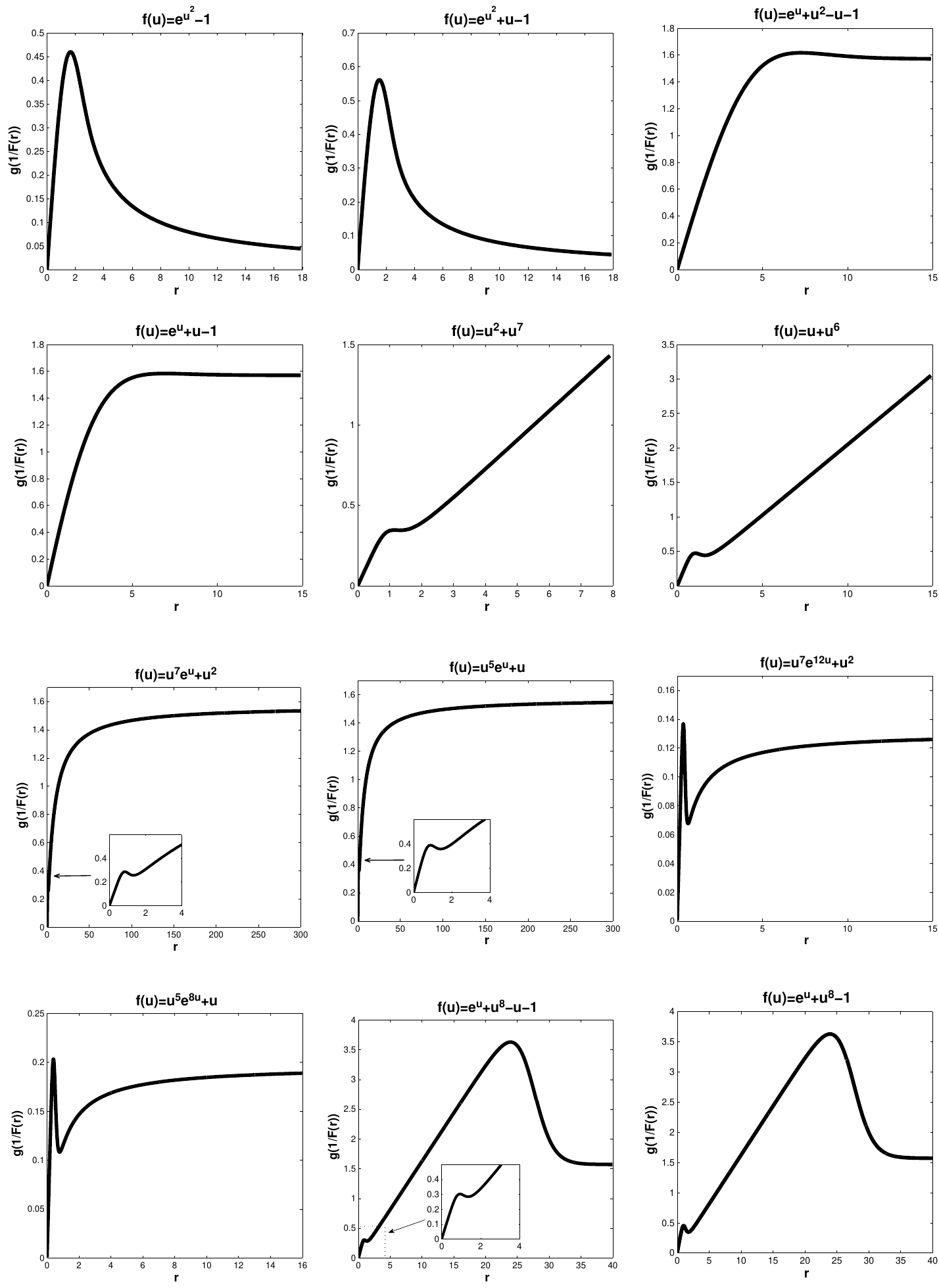}
\caption{Some numerical simulations of $g(\frac{1}{F(r)})$ for
$(\varphi_3,f)$, i.e., the mean curvature equation, in Examples
\ref{eg:beta00}--\ref{eg:4delta0}.}\label{fig:gg3}
\end{figure}
\begin{figure}
\centering
\includegraphics[totalheight=8.2in]{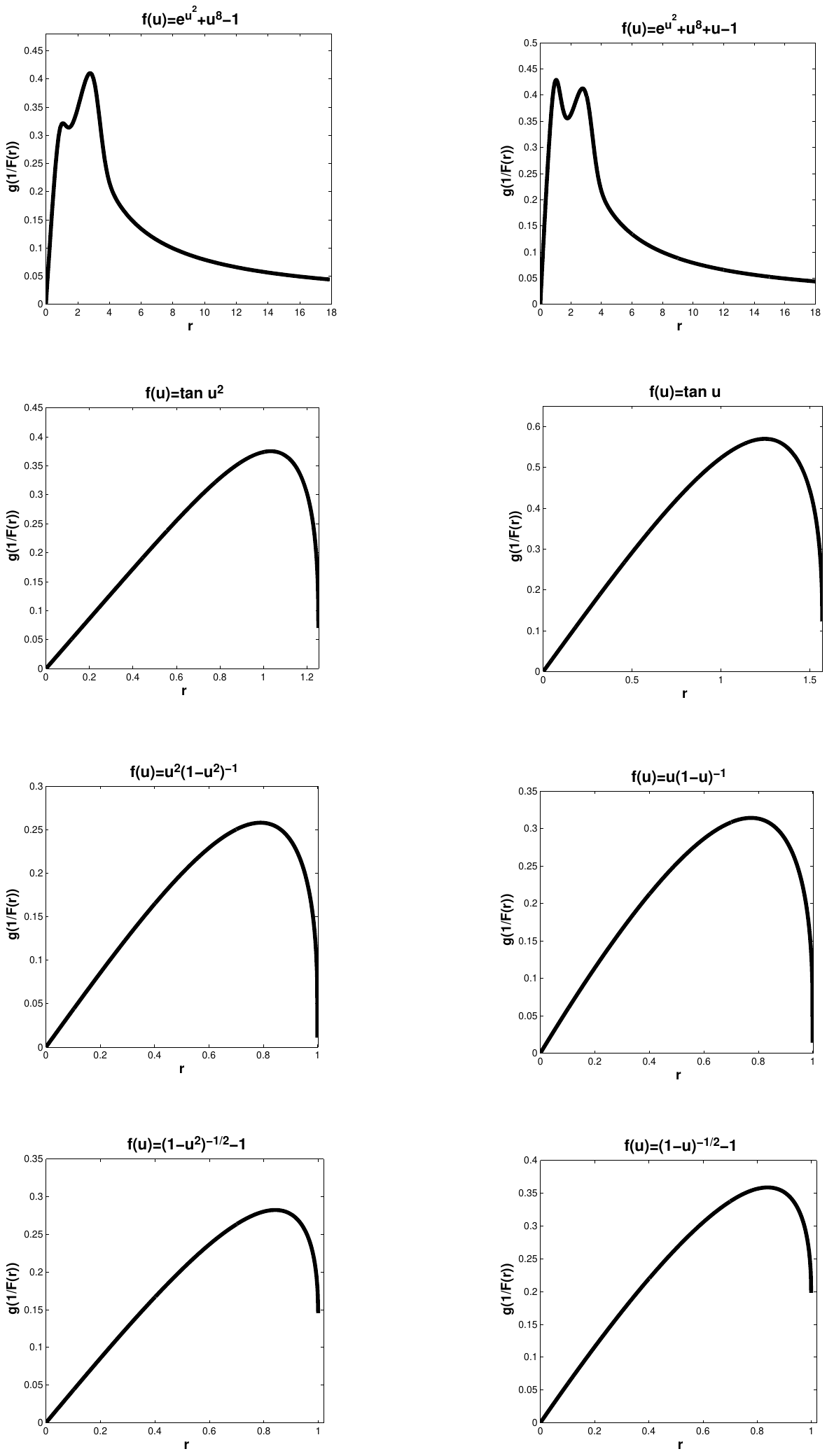}
\caption{Some numerical simulations of $g(\frac{1}{F(r)})$ for
$(\varphi_3,f)$, i.e., the mean curvature equation, in Examples
\ref{eg:4delta1}, \ref{eg:ivdelta3} and \ref{eg:vigamma0}. Left:
$f'(0)=0$. Right: $f'(0)>0$.}\label{fig:gg4}
\end{figure} 
   \fi %


%

\vskip 2mm {\bf Case VI:  $\mathbf{A,B,C<+\infty}$}

In this case, we have
\begin{equation}\label{eq:gsmall000}
g(\lambda)= \left\{
 \begin{array}{lll}
 \lim_{ r\rightarrow A^-}T(r, \lambda)=\int_{0}^{{C}}
\frac{1}{\Phi^{-1}(\lambda y) } \frac{1}{f{\circ}
F^{-1}(C-y)}\text{d}y, & & \hbox{if }
\lambda\leqslant \frac{B}{C}; \\
T\left( F^{-1}(\frac{B}{\lambda}),\lambda\right)=\int_{0}^{B}
\frac{1}{\Phi^{-1}(B- y) } \frac{1}{ \lambda f(
F^{-1}(\frac{y}{\lambda}))}\text{d}y, & & \hbox{if } \lambda>
\frac{B}{C}.
 \end{array}
\right.
\end{equation}

\begin{lemma}[\cite{Pan2014a}]\label{a-prop:twosidee23333b}
Let $A,B,C<+\infty$. Assume conditions \eqref{con:phi} and
\eqref{con:positive} hold. Then $g(\lambda)>0$ and $g$ is continuous
for all $\lambda>0$.
\end{lemma}

By Proposition 5.1 of \cite{Pan2014a}, we know
\begin{lemma}\label{a-prop:twosidee23333}
Let $A,B,C<+\infty$. 
Assume conditions \eqref{con:phi}--\eqref{f-c1} hold. Then $g$ is
strictly decreasing on $(0,\frac{B}{C})$ and
$\lim_{\lambda\rightarrow0}g(\lambda)=+\infty$.
\end{lemma}

Moreover, we know
\begin{lemma}[\cite{Pan2014a}]\label{prop:g3333}
Let $A,B,C<+\infty$. Assume conditions \eqref{con:phi},
\eqref{con:positive}, $\lim_{t\rightarrow 0} \frac{F(t)}{f(t)} =0$
and $\int_{0}^{B} \frac{1}{y \Phi^{-1}(B- y) } dy<+\infty$ hold.
Then $\lim_{\lambda\rightarrow +\infty} g(\lambda)=0$.
\end{lemma}

See e.g., Type $\gamma_0$ of Case VI in Fig.\ref{fig:gg1}. Also see
the bottom of Fig.\ref{fig:gg4} for numerical examples of
$g(\tilde{r})$ in Case VI,
 which arise in Example \ref{eg:vigamma0}.

\section{Proofs of Main Theorems}
In this section, we prove the main results which are stated in
Section 2.

Note that when $f'(0)>0$,
$\lambda_1:=\frac{\varphi'(0)}{f'(0)}(\frac{\pi}{2L})^2$ is strictly
decreasing with respect to $L$, and further from Proposition
\ref{cor:twokind}, it follows that $\lim_{r\rightarrow 0}
T(r,\lambda_1)=L$ .

{\bf The proof of Theorem \ref{thm:atmost12} } Since positive
solutions of \eqref{eq:1dmemsfr0} are always symmetric, it implies
that $u'(0)=0$. Thus the existence of solutions of
\eqref{eq:1dmemsfr0} is equivalent to that of
$$ -\varphi'(u')u''={\lambda}f(u),\quad u'(0)=0,\,\,u(L)=0.$$
By the definition of the time map, for given $L,\lambda>0$, the
number of positive solutions of \eqref{eq:1dmemsfr0}
 is precisely the number of solutions of $T(\lambda,r)=L$ for $r\in (0, r^*)$.
 Here, $r^*$ is the right endpoint of $I$.

From Theorem \ref{T'<0}, we get that $T'<0$ for all $r$.
 Hence for any $\lambda>0$, $T(\lambda,r)=L$ has at most one solutions in $(0, r^*)$.
 This completes the proof.
\hfill$\square$

{\bf The proof of Corollary \ref{thm:atmost1} } Notice that if $f$
is of class $C^2$ on $[0, A)$ satisfying conditions
\eqref{con:positive} and $f(0)=0$, then the convex condition
$f''(u)\geqslant (>) 0$ implies the superlinear condition
$f'(u)u\geqslant (>) f(u)$.

Indeed, let $\psi(u)=f'(u)u-f(u)$, then $\psi(0)=0$ and
$$\psi'(u)=f''(u)u+f'(u)-f'(u)=f''(u)u\geqslant (>)0.$$ So we can
replace \eqref{f-c1} and \eqref{con:strict2} by \eqref{con:convex}
and \eqref{con:strict}. Applying Theorem \ref{T'<0}, we get the
result. \hfill$\square$


{\bf The proof of Theorem \ref{thm:doubleinfinit02}, Corollary
\ref{cor:doubleinfinit02} and Theorem \ref{thm:f=0t1} } By Theorem
\ref{T'<0} and  Proposition \ref{cor:twokind}, we obtain the shape
of $T(r,\lambda)$ for fixed $\lambda>0$.
Further, since $g\equiv 0$ (see Section \ref{sec:3.3}), by Lemmas
\ref{a-prop:twosidee1} and \ref{a-prop:twosidee12},
 we obtain the behavior of $T(r,\lambda)$ when $\lambda$ varies through positive values
(see Fig.\ref{fig:ffI}).
  \ifpdf
\begin{figure}
\centering
\includegraphics[totalheight=8.3in]{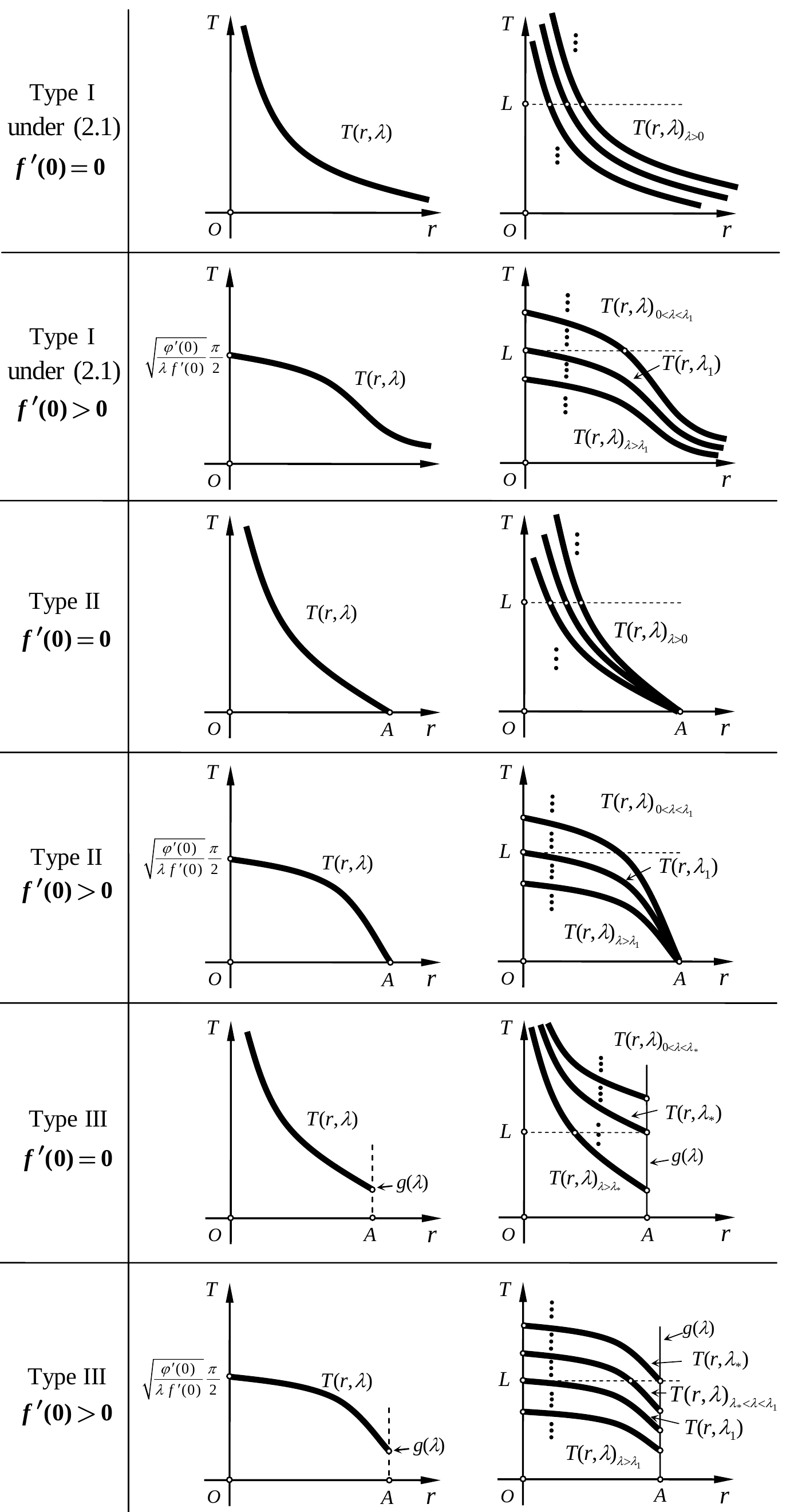}
\caption{Time maps for Types I-III with $f(0)=0$  when $\lambda$
varies.}\label{fig:ffI}
\end{figure}
  \fi

In particular,  for given $L>0$, we have

(1) If $f'(0)= 0$, then for any $\lambda>0$ there exists a unique
$r\in (0,r^*)$ such that $T(r,\lambda)=L$,

(2) If $f'(0)> 0$, then $\lim_{r\rightarrow 0} T(r,\lambda_1)=L$ and
for any $\lambda\in (0, \lambda_1)$ there exists a unique $r\in
(0,r^*)$ such that $T(r,\lambda)=L$.

Thus we completes the proof. \hfill$\square$


{\bf The proof of Theorem \ref{thm:f=0t2}} The proof is similar to
the previous one, a key difference is that $g$ is strictly
decreasing (Lemma \ref{a-prop:twosidee23}). This leads to that for
given $L>0$, there exists a unique $\lambda_*>0$ such that such that
$g(\lambda_*)=L$. The monotonicity of $g$ implies  $\lambda_*$ is
strictly decreasing with respect to $L$ (see Fig.\ref{fig:ffI}).

In particular, for given $L>0$, we have

(1) There exists a unique $\lambda_*>0$ such that
$T(r,\lambda_*)=L$.

(2) If $f'(0)= 0$, then for any $\lambda>\lambda_*$ there exists a
unique $r\in (0,r^*)$ such that $T(r,\lambda)=L$, while if $f'(0)>
0$, then $\lim_{r\rightarrow 0} T(r,\lambda_1)=L$ and for any
$\lambda\in (\lambda_*, \lambda_1)$ there exists a unique $r\in
(0,r^*)$ such that $T(r,\lambda)=L$.

Thus we completes the proof.\hfill$\square$

{\bf The proof of Theorem \ref{thm:f=0t3} } Since $(\varphi, f)$ is
of  \emph{Type IV-$\alpha_0$}, it follows that $g$ is strictly
decreasing, $\lim_{\lambda\rightarrow0}g(\lambda)=+\infty$ and
$\lim_{\lambda\rightarrow+\infty }g(\lambda)=0$ (see
Figs.\ref{fig:gg1} and \ref{fig:gg2}). Therefore the proof is the
same as that of Theorem \ref{thm:f=0t2}, we omit it (see
Fig.\ref{fig:fIV}). \hfill$\square$
  \ifpdf
\begin{figure}
\centering
\includegraphics[totalheight=8.2in]{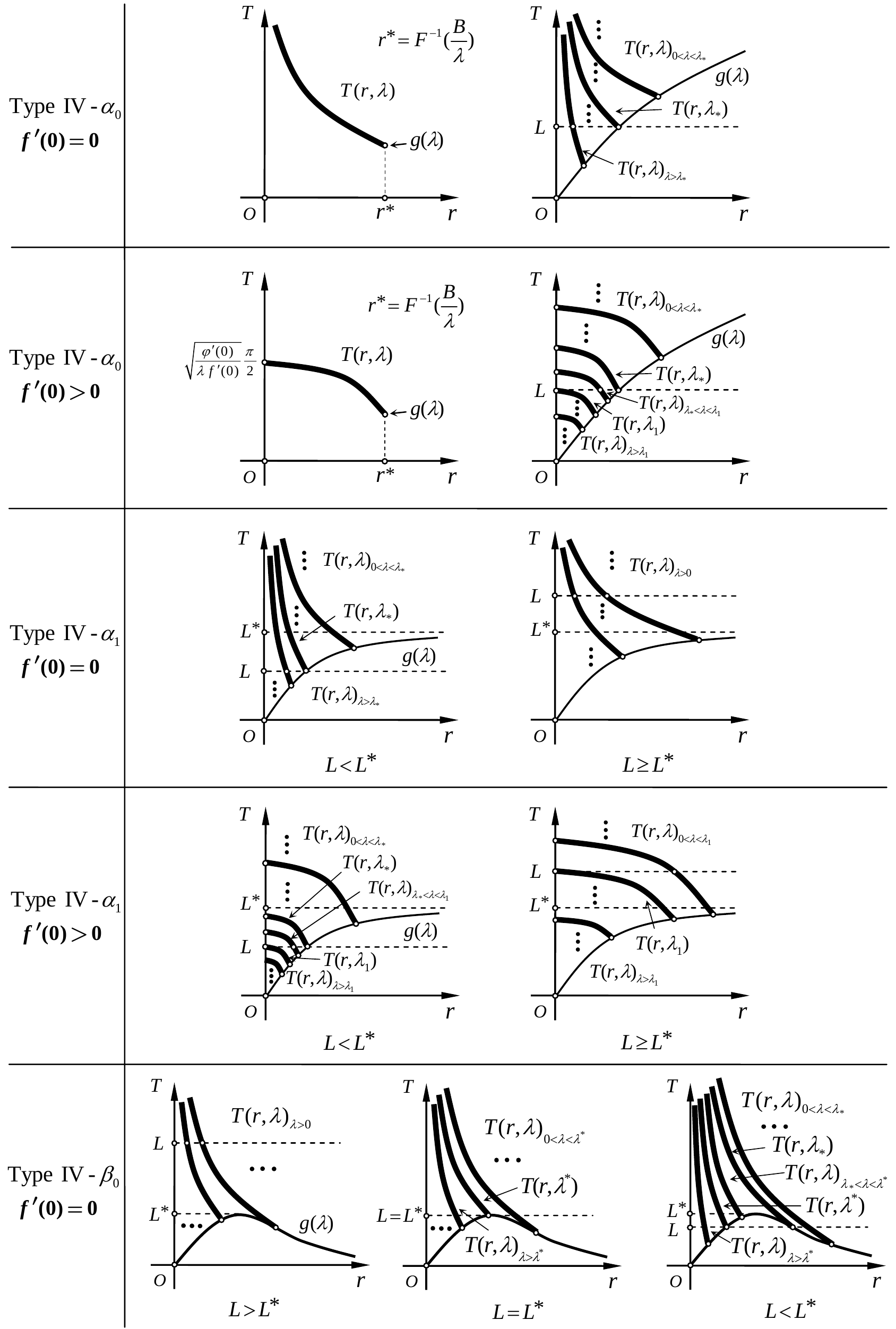}
\caption{Time maps for Types IV-$\alpha_0$, IV-$\alpha_1$ and
IV-$\beta_0$ with $f(0)=0$ when $\lambda$ varies.}\label{fig:fIV}
\end{figure}
  \fi

{\bf The proof of Corollary \ref{cor:f=0t3} } From Lemmas
\ref{prop:g} and \ref{prop:monotoneg}, it follows that $(\varphi,
f)$ is of   Type IV-$\alpha_0$ and hence the conclusions of Theorem
\ref{thm:f=0t3} hold. \hfill$\square$

  \ifpdf
\begin{figure}
\centering
\includegraphics[totalheight=8.2in]{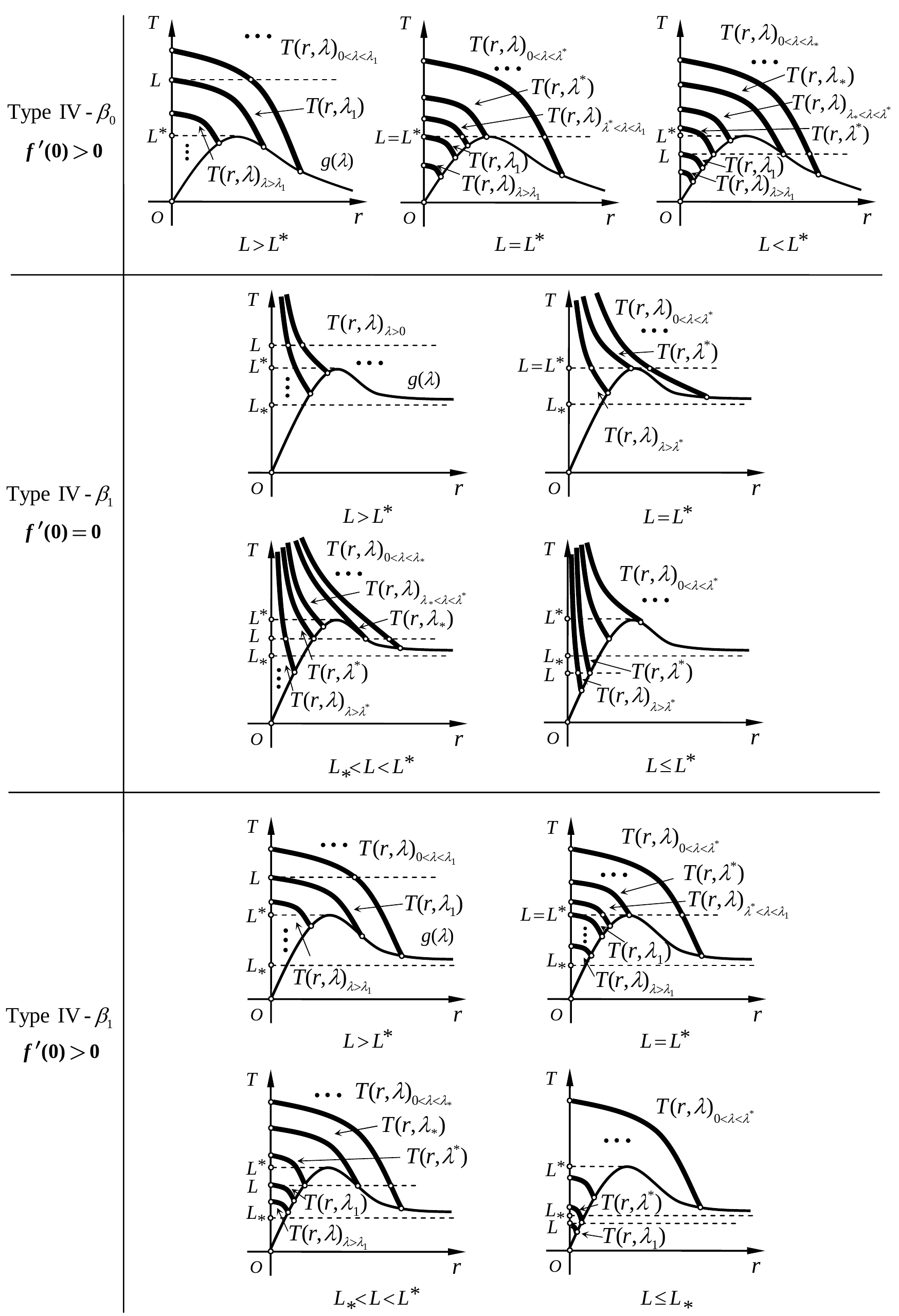}
\caption{Time maps for Types IV-$\beta_0$ and IV-$\beta_1$ with
$f(0)=0$ when $\lambda$ varies.}\label{fig:four32}
\end{figure}
\begin{figure}
\centering
\includegraphics[totalheight=8.3in]{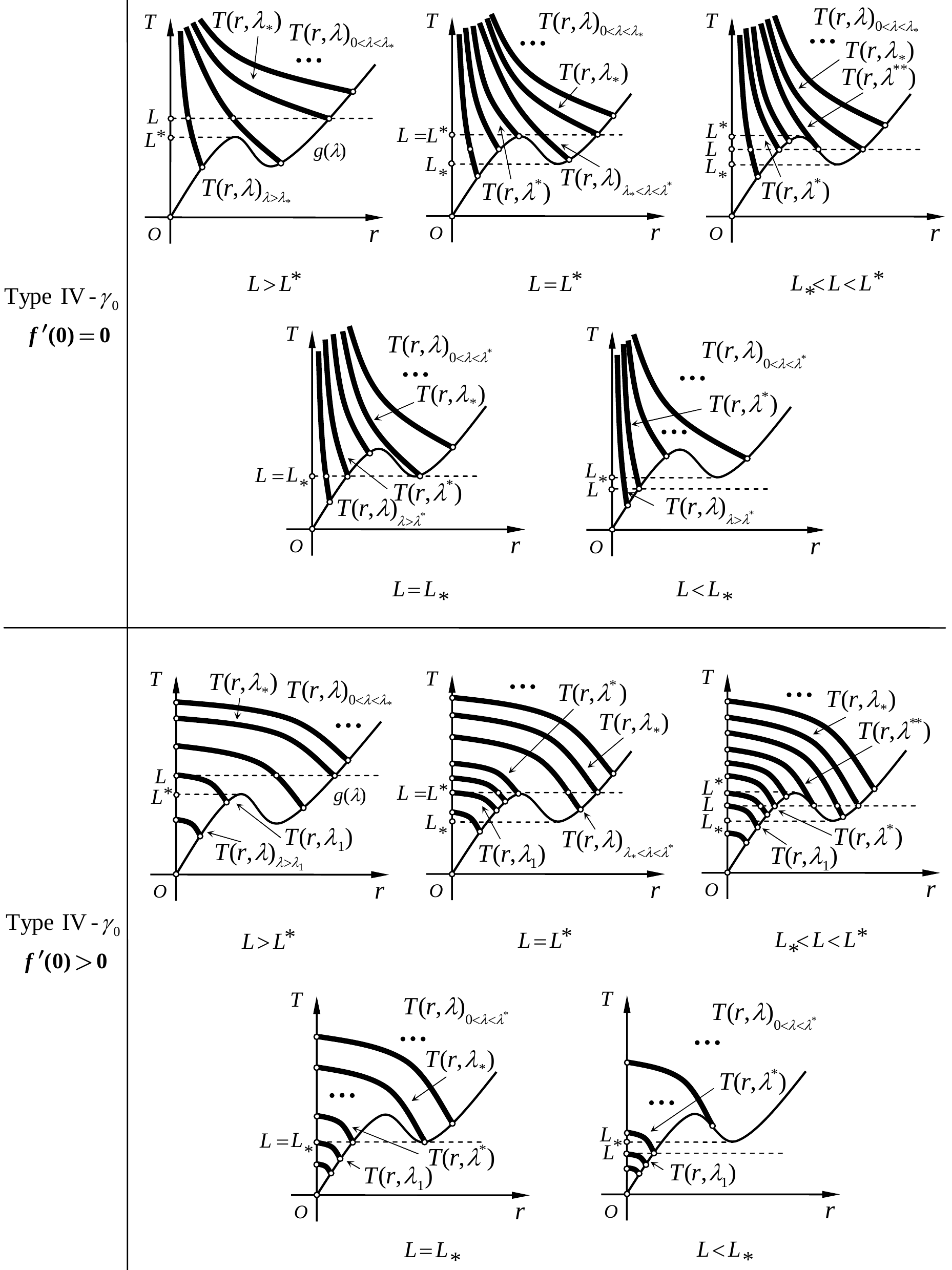}
\caption{Time maps for Type IV-$\gamma_0$ with $f(0)=0$ when
$\lambda$ varies.}\label{fig:four323}
\end{figure}
\begin{figure}
\centering
\includegraphics[totalheight=8.3in]{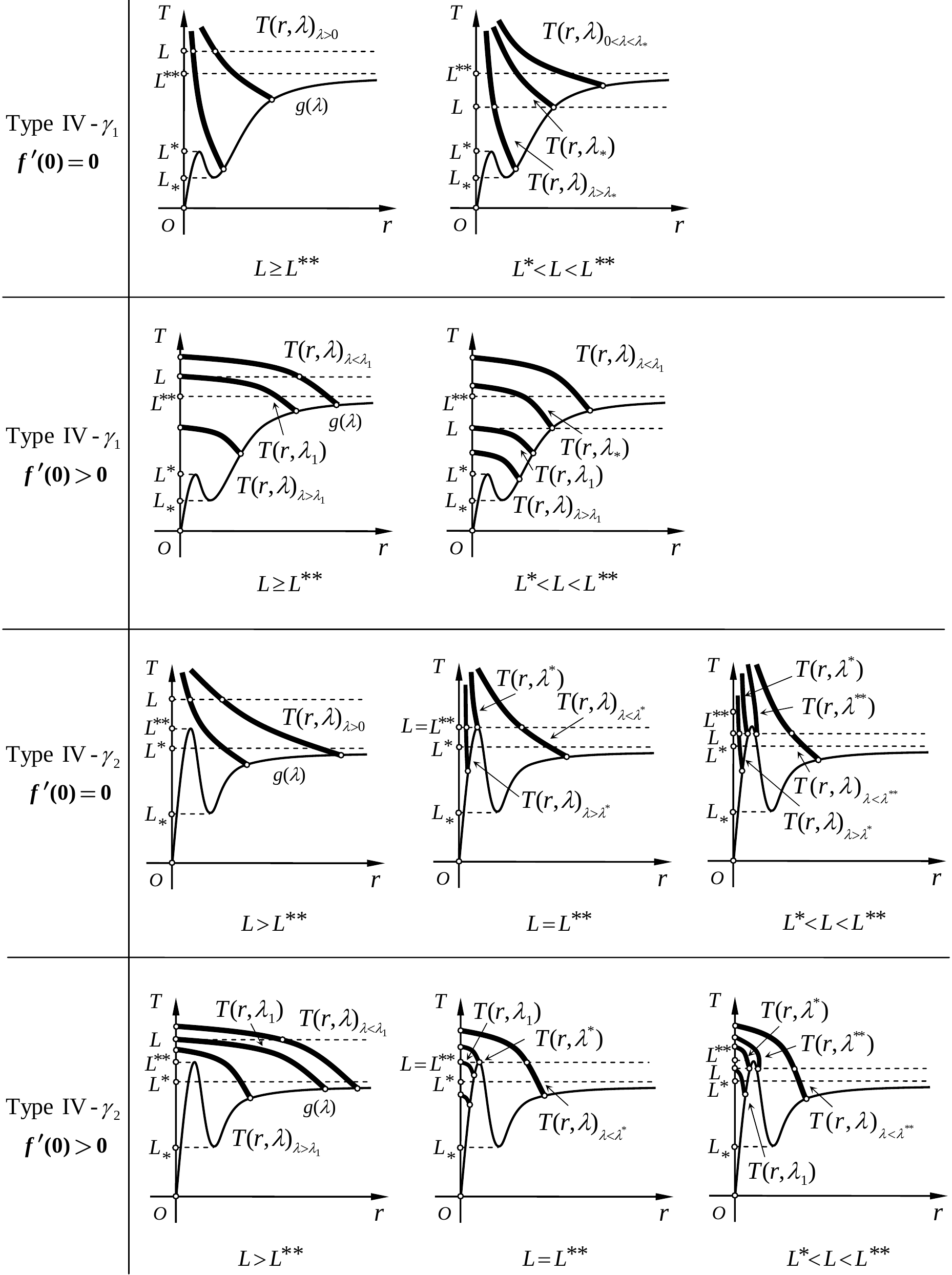}
\caption{Time maps for Types IV-$\gamma_1$ and IV-$\gamma_2$ with
$f(0)=0$ when $\lambda$ varies. The remaining cases $L=L^*$,
$L_*<L<L^*$, $L=L_*$ and $L<L^*$ are the same as Types
IV-$\gamma_0$.}\label{fig:four33}
\end{figure}
\begin{figure}
\centering
\includegraphics[totalheight=8.2in]{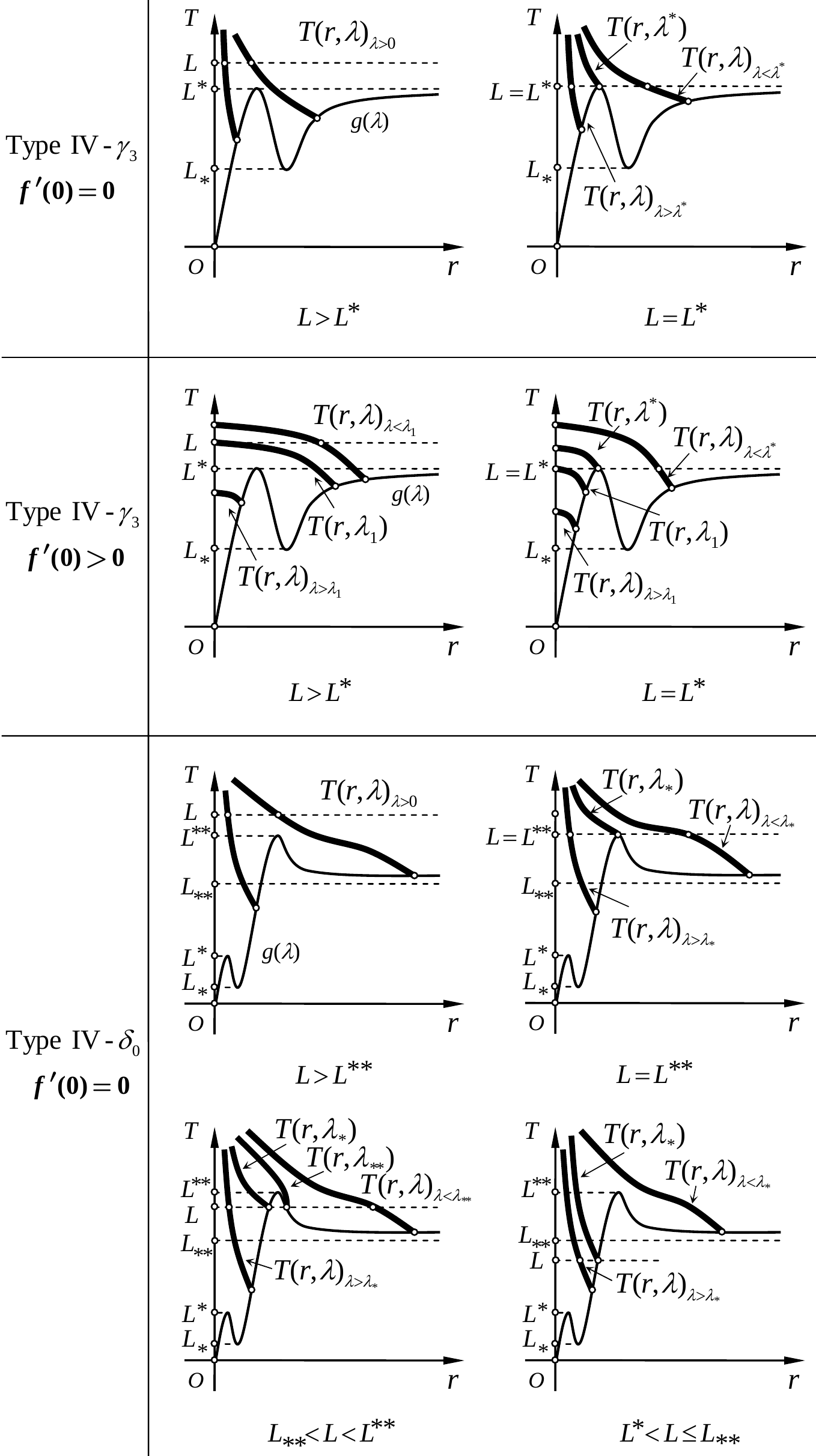}
\caption{Time maps for Types IV-$\gamma_3$ and IV-$\delta_0$ with
$f(0)=0$ when $\lambda$ varies. The remaining cases $L=L^*$,
$L_*<L<L^*$, $L=L_*$ and $L<L^*$ are the same as Types
IV-$\gamma_0$.}\label{fig:four34}
\end{figure}
  \fi

{\bf The proof of Theorems \ref{thm:f=0t4} and \ref{thm:f=0t5}} By
Theorem \ref{T'<0} and  Proposition \ref{cor:twokind}, we obtain the
shape of $T(r,\lambda)$  for fixed $\lambda>0$ (see
Fig.\ref{fig:twotimemaps} or the left of Type IV-$\alpha_0$ in
Fig.\ref{fig:fIV}).

Since $(\varphi, f)$ is of  Type IV-$\alpha_1$, it follows that $g$
is strictly decreasing, $\lim_{\lambda\rightarrow0}g(\lambda)\in
(0,+\infty)$ and $\lim_{\lambda\rightarrow+\infty}$ $g(\lambda)=0$
(see Figs.\ref{fig:gg1} and \ref{fig:gg2}). This, together with
Lemma \ref{a-prop:twosidee1}, gives the behavior of $T(r,\lambda)$
when $\lambda$ varies through positive values (see the right of
Fig.\ref{fig:fIV}). In particular, letting $L^*=\sup g$, we have

(1) For given $L\in (0,L^*)$, there exists a unique $\lambda_*>0$
such that $g(\lambda_*)=L$. Moreover, the monotonicity of $g$
implies that of $\lambda_*$ with respect to $L$.


(2) For given $L\in [L^*,+\infty)$, 
if $f'(0)= 0$, then for any $\lambda>0$ there exists a unique $r\in
(0,r^*)$ such that $T(r,\lambda)=L$, while if $f'(0)> 0$, then
$\lim_{r\rightarrow 0} T(r,\lambda_1)=L$ and for any $\lambda\in (0,
\lambda_1)$ there exists a unique $r\in (0,r^*)$ such that
$T(r,\lambda)=L$.

Thus we completes the proof.\hfill$\square$

{\bf The proof of Corollary \ref{cor:f=0t5} } From Lemmas
\ref{prop:g} and \ref{prop:monotoneg}, it follows that $(\varphi,
f)$ is of   Type IV-$\alpha_1$ and hence the conclusions of Theorems
\ref{thm:f=0t4} and \ref{thm:f=0t5} hold. \hfill$\square$

{\bf The proof of Theorems \ref{thm:classical43320} and
\ref{thm:classical43321}} By Theorem \ref{T'<0} and  Proposition
\ref{cor:twokind}, we obtain the shape of $T(r,\lambda)$  for fixed
$\lambda>0$ (see Fig.\ref{fig:twotimemaps} or the left of Type
IV-$\alpha_0$ in Fig.\ref{fig:fIV}).

Since $(\varphi, f)$ is of  Type IV-$\beta_0$, we know the graph of
$g$ (see Figs.\ref{fig:gg1} and \ref{fig:gg2}). This, together with
Lemma \ref{a-prop:twosidee1}, gives
 the behavior of $T(r,\lambda)$ when $\lambda$ varies through positive values
(see Figs.\ref{fig:fIV} and \ref{fig:four32}). In particular,
letting $L^*=\sup g$, we have

(1) For given $L=L^*$, there exists a unique $\lambda_*>0$ such that
$g(\lambda_*)=L$.

(2) For given $L\in (0,L^*)$, there exist $\lambda^*>\lambda_{*}>0$
such that $g(\lambda^{*})=g(\lambda_*)=L^*$. Moreover, the shape of
$g$ implies the monotonous relations of $\lambda_*$ and
$\lambda^{*}$ with respect to $L$.

(3) For given $L>0$, 
if $f'(0)> 0$, then $\lim_{r\rightarrow 0} T(r,\lambda_1)=L$.

(4) For $L>L^*$, the situation is similar to Type IV-$\alpha_1$.


Thus we completes the proof.\hfill$\square$

{\bf The proofs of Theorems
\ref{thm:classical43330}--\ref{thm:f=0t13}} All proofs are similar
to that of Theorems \ref{thm:classical43320} and
\ref{thm:classical43321}, we omit them. Notice that for fixed
$\lambda>0$, the graphs of all $T(r,\lambda)$ in Case IV are similar
to the left of Type IV-$\alpha_0$ in Fig.\ref{fig:fIV}, various
different types of $g$ (see Fig.\ref{fig:gg2}) essentially lead to
the differences of bifurcation diagrams.

In Figs.\ref{fig:four32} and \ref{fig:four323}, we give the graphs
of Time maps for Types IV-$\beta_1$ and IV-$\gamma_0$ when $\lambda$
varies. In Figs.\ref{fig:four33} and \ref{fig:four34}, we do not
give all graphs of Time maps for Types IV-$\gamma_1$, IV-$\gamma_2$,
IV-$\gamma_3$ and IV-$\delta_0$ because the remaining cases $L=L^*$,
$L_*<L<L^*$, $L=L_*$ and $L<L^*$ are the same as Types
IV-$\gamma_0$. Types IV-$\delta_1$, IV-$\delta_2$ and IV-$\delta_3$
can be discussed in the same way, we omit them.

Besides, the discussions for Types V-$\beta_0$ and VI-$\gamma_0$ are
also completely similar to those of Types IV-$\beta_0$ and
IV-$\gamma_0$ (see Figs.\ref{fig:twotimemaps} and \ref{fig:gg2}), we
omit them. \hfill$\square$

\vskip 2mm {\bf Acknowledgement}
The first author is supported by Pearl River Nova Program on Science and Technology
of Guangzhou City (2014J2200010) and NSF of Guangdong Province.
The second author is supported by NSF of China (11471339, 11001277).
\bibliographystyle{plain}
\bibliography{quench}
\end{document}